\documentclass[10pt,leqno,oneside,letterpaper]{amsart}
\usepackage{amssymb,amsmath,amsfonts,latexsym, amscd, amsthm, graphicx, url, pstricks, epsfig, pst-grad, pst-plot, amsbsy,epsf,graphics} 
\usepackage[letterpaper,left=1.75truein, top=1truein]{geometry} 
\usepackage{color}
\usepackage[colorlinks,linkcolor=blue,citecolor=blue]{hyperref}
\newtheorem{theorem}{Theorem}[section]
\newtheorem{lemma}[theorem]{Lemma}
\newtheorem{proposition}[theorem]{Proposition}
\newtheorem{definition}[theorem]{Definition}

\newtheorem{corollary}[theorem]{Corollary} 
\newtheorem{claim}[theorem]{Claim}
\theoremstyle{definition}
\newtheorem{remark}[theorem]{Remark} 
\newtheorem{remarks}[theorem]{Remarks}
\newtheorem{example}[theorem]{Example} 
\newtheorem{question}[theorem]{Question} 
\newtheorem{problem}[theorem]{Problem} 
\newtheorem{convention}[theorem]{Convention}
 
\newtheorem{conjecture}[theorem]{Conjecture}

\newtheoremstyle{cases}
  {12pt plus 6 pt}%       Space above
  {2pt}%       Space below
  {\bfseries}   %       Body font
  {}%          Indent amount (empty = no indent, \parindent = para indent)
  {\bfseries}% Thm head font
  {.}%         Punctuation after thm head
  {.5em}%      Space after thm head: " " = normal interword space;
  {}%          Thm head spec (can be left empty, meaning `normal')
\theoremstyle{cases}

\numberwithin{subcase}{case} \numberwithin{subsubcase}{subcase}
\numberwithin{equation}{subsection}
\usepackage{numbering}

\begin{document}

\title[Slope detection and graph manifolds]{Foliations, orders, representations, L-spaces and  graph manifolds }

\author[Steven Boyer]{Steven Boyer}
\thanks{Steven Boyer was partially supported by NSERC grant RGPIN 9446-2008}
\address{D\'epartement de Math\'ematiques, Universit\'e du Qu\'ebec \`a Montr\'eal, 201 avenue du Pr\'esident-Kennedy, Montr\'eal, QC H2X 3Y7.}
\email{boyer.steven@uqam.ca}
\urladdr{http://www.cirget.uqam.ca/boyer/boyer.html}

\author{Adam Clay}
\thanks{Adam Clay was partially supported by an NSERC postdoctoral fellowship}
\address{Department of Mathematics, 420 Machray Hall, University of Manitoba, Winnipeg, MB, R3T 2N2. } 
\email{Adam.Clay@umanitoba.ca}
\urladdr{http://server.math.umanitoba.ca/~claya/}

\begin{abstract}
We show that the properties of admitting a co-oriented taut foliation and having a left-orderable fundamental group are equivalent for rational homology $3$-sphere graph manifolds and relate them to the property of not being a Heegaard-Floer L-space. This is accomplished in several steps. First we show how to detect families of slopes on the boundary of a Seifert fibred manifold in four different fashions---using representations, using left-orders, using foliations, and using Heegaard-Floer homology. Then we show that each method of detection determines the same family of detected slopes. Next we provide necessary and sufficient conditions for the existence of a co-oriented taut foliation on a graph manifold rational homology $3$-sphere, respectively a left-order on its fundamental group, which depend solely on families of detected slopes on the boundaries of its pieces. The fact that Heegaard-Floer methods can be used to detect families of slopes on the boundary of a Seifert fibred manifold combines with certain conjectures in the literature to suggest an L-space gluing theorem for rational homology $3$-sphere graph manifolds as well as other interesting problems in Heegaard-Floer theory. 
\end{abstract}

\maketitle
\vspace{-.6cm}
\begin{center}
\today
\end{center}

\section{Introduction}

Much work has been devoted in recent years to examining relationships between the existence of a co-oriented taut foliation in a closed, connected, prime $3$-manifold $W$, the left-orderability of its fundamental group, and the property that it {\it not} be a Heegaard-Floer L-space. Indeed, it has been conjectured that the three conditions coincide. (See Conjecture 1 of \cite{BGW} and Conjecture 5 of \cite{Ju}.) When $W$ has a positive first Betti number, each condition holds (\cite[Theorem 5.5, page 477]{Ga1} and \cite[Theorem 1.1]{BRW}). Further, it follows from \cite[Theorems 1.3 and 1.7]{BRW} that when $W$ is a non-hyperbolic geometric manifold, $W$ has a left-orderable fundamental group if and only if it admits a co-oriented taut foliation. On the other hand, \cite[Theorem 1 and Corollary 1]{BGW} imply that for such manifolds, the latter is equivalent to the condition that $W$ not be an L-space. Thus understanding the relationship between the three conditions reduces to the case when $W$ is a rational homology $3$-sphere which is either hyperbolic or has a non-trivial JSJ decomposition. In this paper we show that the first two conditions are equivalent when $W$ is a graph manifold (cf. Theorem \ref{equiv 1}) and make some steps toward relating them to the third. With regards to the latter, we have

\begin{theorem} \label{LO implies NLS}  
If a graph manifold rational homology $3$-sphere has a left-orderable fundamental group, then it is not an L-space. 
\qed
\end{theorem} 

The proof of Theorem \ref{LO implies NLS} follows from Theorem \ref{equiv 1} below via an argument of Oszv\'ath and Szab\'o (\cite[Theorem 1.4]{OSz2004-genus}), which ultimately depends on a result concerning the approximation of taut foliations by contact structures. Eliashberg and Thurston proved the latter in the $C^2$ case (\cite[Theorem 2.9.1]{ET})  and following a suggestion of Jonathan Bowden, we construct certain smooth foliations on the pieces of the graph manifold $W$ in \cite{BC}  whose existence can be combined with the Eliashberg-Thurston result to produce the appropriate contact structures, even though the foliations do not necessarily glue together to form a smooth foliation on $W$. Alternately, Bowden \cite{Bn} and Kazez and Roberts \cite{KR} have extended the Eliashberg-Thurston theorem to the $C^0$ case, which can then be used to complete the proof of Theorem \ref{LO implies NLS}.

The methods we develop here allow us to formulate an L-space gluing conjecture whose truth implies the inverse implication of Theorem \ref{LO implies NLS}. See Conjecture \ref{hf gluing conjecture} and Problem \ref{special case gluing} below. 

The equivalence between the existence of a co-oriented taut foliation on a graph manifold and the left-orderability of its fundamental group takes on various forms, depending on the types of foliations considered (cf. Theorems \ref{equiv 1}, \ref{equiv 2} and \ref{equiv 3}). And, since the fundamental group of a connected, prime, orientable $3$-manifold $W$ is left-orderable if and only if it admits a homomorphism to $\hbox{Homeo}_+(\mathbb R)$ with non-trivial image (\cite{Linnell}, \cite[Theorem 1.1(1)]{BRW}), we translate the order-theoretic statements Theorems \ref{equiv 1}(2), \ref{equiv 2}(2) and \ref{equiv 3}(2) into their representation-theoretic counterparts (Theorems \ref{equiv 1}(3), \ref{equiv 2}(3) and \ref{equiv 3}(3)), which are of independent interest.

\begin{theorem} \label{equiv 1}
Let $W$ be a graph manifold rational homology $3$-sphere. The following statements are equivalent. 

$(1)$ $W$ admits a co-oriented taut foliation. 

$(2)$ $\pi_1(W)$ is left-orderable.

$(3)$ There is a homomorphism $\rho: \pi_1(W) \to \hbox{Homeo}_+(\mathbb R)$ with non-trivial image.

\end{theorem}

Theorem \ref{equiv 1} is known to hold when $W$ is Seifert fibred \cite{BRW}, \cite{BGW} (and that (1), (2), and (3) are equivalent to $W$ not being an L-space \cite{LS}). In this case the foliations can be chosen to be horizontal. In other words, they are transverse to the Seifert fibres of $W$. More generally, a co-dimension $1$ foliation in a graph manifold $W$ is called {\it horizontal} if it is transverse to the Seifert fibres in each piece of $W$. We can refine Theorem \ref{equiv 1} by restricting our attention to this important family of foliations. Let $\hbox{sh}(\pm 1): \mathbb R \to \mathbb R$ denote the homeomorphism $\hbox{sh}(\pm 1)(x) = x \pm 1$. A homeomorphism of the reals is fixed point free if and only if it is conjugate to $\hbox{sh}(\pm 1)$

\begin{theorem} \label{equiv 2}
Let $W$ be a graph manifold rational homology $3$-sphere. The following statements are equivalent. 

$(1)$ $W$ admits a co-oriented horizontal foliation. 

$(2)$ $\pi_1(W)$ admits a left-order in which the class of any Seifert fibre in any piece of $W$ is cofinal.

$(3)$ There is a homomorphism $\rho: \pi_1(W) \to \hbox{Homeo}_+(\mathbb R)$ such that the image of the class of any Seifert fibre in any piece of $W$ is conjugate in $\hbox{Homeo}_+(\mathbb R)$ to $\hbox{sh}(\pm 1)$.

\end{theorem}

Here is a consequence of the proofs of these results. Call a co-oriented taut foliation {\it rational} if up to isotopy it intersects each JSJ torus of $W$ in a fibration with a compact leaf. 

\begin{proposition} \label{rational leaves}
If $W$ admits a co-oriented taut foliation, respectively a horizontal co-oriented foliation, it admits a co-oriented taut rational foliation, respectively a horizontal co-oriented rational foliation. 
\end{proposition}

A fundamental open problem in $3$-manifold topology is to determine whether a $3$-manifold which admits a co-oriented taut foliation admits a smooth co-oriented taut foliation. The problem is open even in the case of graph manifolds which are rational homology $3$-spheres. (See the discussion after Question \ref{smooth question} below.) Though the constructions which lead to the proofs of Theorem \ref{equiv 1} and Theorem \ref{equiv 2} do not yield smooth foliations in general, we can strengthen Proposition \ref{rational leaves} and guarantee smoothness under suitable hypotheses. 

Call a co-oriented taut foliation {\it strongly rational} if up to isotopy it intersects each JSJ torus of $W$ in a fibration by simple closed curves. Since no co-oriented taut foliation on a graph manifold rational homology $3$-sphere $W$ can intersect a JSJ-torus $T$ in a fibration by simple closed curves representing the fibre slope in a piece of $W$ incident to $T$, at least up to assuming that the Seifert structures on pieces homeomorphic to twisted $I$-bundles over the Klein bottle have orientable base orbifolds (Lemma \ref{strongly rational not vertical}), a strongly rational co-oriented taut foliation is necessarily horizontal. Boileau and Boyer have shown that a graph manifold integer homology $3$-sphere admits a strongly rational co-oriented taut foliation if and only if it is neither $S^3$ nor the Poincar\'e homology $3$-sphere $\Sigma(2,3,5)$ \cite{BB}. We will see below that if $W$ admits a strongly rational co-oriented taut foliation, it admits a smooth strongly rational co-oriented taut foliation, so it is of interest to prove a version of the theorems above for this category of foliations. 

\begin{theorem} \label{equiv 3}
Let $W$ be a graph manifold rational homology $3$-sphere. The following statements are equivalent. 

$(1)$ $W$ admits a strongly rational co-oriented taut foliation. 

$(2)$ $\pi_1(W)$ admits a left-order $\mathfrak{o}$ in which the class of any Seifert fibre in any piece of $W$ is cofinal and there is an $\mathfrak{o}$-convex 
normal subgroup $C$ of $\pi_1(W)$ such that $C \cap \pi_1(T) \cong \mathbb Z$ for each JSJ-torus $T$ in $W$.

$(3)$  There is a homomorphism $\rho: \pi_1(W) \to \hbox{Homeo}_+(\mathbb R)$ such that the image of the class of any Seifert fibre in any piece of $W$ is conjugate in $\hbox{Homeo}_+(\mathbb R)$ to $\hbox{sh}(\pm 1)$ and $\hbox{ker}(\rho|\pi_1(T))  \cong \mathbb Z$ for each JSJ-torus $T$ in $W$.

\end{theorem}
The hypotheses of (i) admitting a co-oriented taut foliation, (ii) admitting a horizontal co-oriented taut foliation, and (iii) admitting a strongly rational co-oriented taut foliation are successively more constraining. (See \S \ref{examples and smoothness}.) In particular, not every graph manifold rational homology $3$-sphere which admits a co-oriented taut foliation also admits a strongly rational co-oriented taut foliation. On the other hand, the results above combine with those of the Appendix to imply that this is true generically, at least in terms of the gluing of its pieces.

Our strategy for establishing these theorems is based on the two main technical results of the paper: Theorem \ref{theorem: detection}, a slope detection theorem, and Theorem \ref{theorem: gluing}, a gluing theorem. More precisely, we introduce four different methods of detecting a family of slopes on the boundary of a Seifert fibred manifold $M$: using representations (\S \ref{section: detecting via reps}), using left-orders (\S \ref{section: detecting via orders}), using foliations (\S \ref{section: detecting via foliations}), and using Heegaard-Floer homology  (\S \ref{section: detecting via L-spaces})\footnote{It is also possible to detect slopes via contact structures, but we do not pursue this here.}. Each form of detection has a more restrictive form which is easy to understand for rational slopes. Thus a family of rational slopes on the boundary of a Seifert fibred manifold $M$ is {\it strongly representation detected} if it is horizontal and there is a homomorphism from the fundamental group of the associated Dehn filled manifold $W$ to $\widetilde{\hbox{Homeo}}_+(S^1)$ which sends the fibre class to $\hbox{sh}(1)$. It is {\it strongly order detected} if $W$ has a left-orderable fundamental group. It is {\it strongly foliation detected} if there is a co-oriented taut foliation on $M$ which restricts to a linear foliation of the given slope on each boundary component of $M$. Finally, it is {\it strongly NLS detected} if $W$ is not an L-space. Thus the various forms of strong detection of rational families of slopes are essentially characterised by whether the associated Dehn filled manifold admits a co-oriented taut foliation or is not an L-space, and whether its fundamental group admits a left-order or an appropriate representation with values in $\widetilde{\hbox{Homeo}}_+(S^1)$. On the other hand, strong detection is not well-adapted to understanding whether manifolds obtained by gluing Seifert manifolds along their boundaries admit co-oriented taut foliations or are L-spaces, or whether their fundamental groups admit a left-order or a non-trivial representation with values in $\widetilde{\hbox{Homeo}}_+(S^1)$. To make progress on these problems we are led to loosening the concept of strong detection. An instructive example is provided by $+4$-surgery on the figure-eight knot, which we denote by $W$. 

It is known that $W$ is the union of the trefoil exterior $M$ and a twisted $I$-bundle over the Klein bottle $N_2$ where the gluing map identifies the meridional slope $[\mu]$ of $M$ with the rational longitude of $N_2$. Delman \cite{De} and Roberts \cite{Ro} have shown that $W$ admits a co-oriented taut foliation, and it is possible to choose one which intersects $\partial M = \partial N_2$ transversely in a foliation with some circular leaves of slope $[\mu]$. This is what it means for $[\mu]$ to be {\it foliation detected} in $M$. On the other hand, since $S^3$ admits no co-oriented taut foliation, there is no co-oriented taut foliation on $M$ which is linear of slope $[\mu]$. In other words, $[\mu]$ is not strongly foliation detected in $M$. A similar situation arises when investigating in what way $[\mu]$ is or is not NLS detected. First, $[\mu]$ is not strongly NLS detected in $M$ because $S^3$ is an L-space. On the other hand, $W$ is not an L-space (\cite[Proposition 4.1]{OSz2005-lens}), which is a key point in showing that $[\mu]$ is NLS detected. Indeed, we introduce a family of {\it Heegaard-Floer solid tori} $N_t$ ($t \geq 2$) in \S \ref{N_t} and say that a rational slope $[\alpha]$ on $\partial M$ is {\it NLS detected} if for each $t$, a manifold obtained by gluing $M$ and $N_t$ in such a way that $[\alpha]$ is identified with the rational longitude of $N_t$ is not an L-space. 

Theorem \ref{theorem: detection} states that any two notions of (strong) detection coincide when both are defined. Theorem \ref{theorem: gluing} provides necessary and sufficient conditions for the existence of co-oriented taut foliations on a graph manifold rational homology $3$-sphere, respectively  a left-order on its fundamental group, from families of appropriately detected slopes on the boundaries of its pieces. The gluing conditions depend only on slope detectability, which leads to the equivalences of Theorems \ref{equiv 1}, \ref{equiv 2}, and \ref{equiv 3}. Here are special cases of the slope detection and gluing theorems.  

\begin{theorem} \label{theorem: special detection}
Let $M$ be a Seifert manifold with base orbifold $P(a_1, \ldots , a_n)$ or $Q(a_1, \ldots , a_n)$ where $P$ is a punctured $2$-sphere and $Q$ is a punctured projective plane. Let $\emptyset \ne \partial M = T_1\cup \ldots \cup T_r$ be the decomposition of $\partial M$ into its toral boundary components. Let $[\alpha_j]$ be a slope on $T_j$ and set $[\alpha_*] = ([\alpha_1], [\alpha_2], \ldots , [\alpha_r])$. The  following statements are equivalent.

$(1)$ $[\alpha_*]$ is detected by some co-oriented taut foliation on $M$.

$(2)$ $[\alpha_*]$ is detected by some left-order on $\pi_1(M)$.  

$(3)$ If no $[\alpha_j]$ is vertical, $[\alpha_*]$ is 
detected by some homomorphism $\rho: \pi_1(M) \to \widetilde{\hbox{Homeo}}_+(S^1)$. 

$(4)$ If $[\alpha_*]$ is rational, then for all integers $t \geq 2$, no manifold obtained by attaching $r$ copies of $N_t$ to $M$ such that the rational longitude of $N_t$ is identified with $[\alpha_j]$ is an L-space. 

\end{theorem}

Brittenham, Naimi and Roberts showed that taut foliations in graph manifolds are built from their counterparts in their JSJ pieces (\cite{BNR}) and used this to construct, for instance, graph manifolds with no co-oriented taut foliations. Part (1) of the next theorem refines this work. 

\begin{theorem} \label{theorem: special gluing} 
Let $W$ be a graph manifold rational homology $3$-sphere with JSJ pieces $M_1, \ldots, M_n$. For each piece $M_i$ and 
$m$-tuple of slopes $[\alpha_*] = ([\alpha_1], [\alpha_2], \ldots , [\alpha_m])$, one for each of the JSJ tori, let $[\alpha_*^{(i)}]$ be the sub-tuple of $[\alpha_*]$ corresponding to the boundary components of $M_i$. Then, 

$(1)$ $W$ admits a co-oriented taut foliation if and only if there is an $m$-tuple of slopes $[\alpha_*]$ such that for each $i$, $[\alpha_*^{(i)}]$ is detected by some co-oriented taut foliation on $M_i$. 

$(2)$ $\pi_1(W)$ is left-orderable if and only if there is an $m$-tuple of slopes $[\alpha_*]$ such that for each $i$, $[\alpha_*^{(i)}]$ is detected by some left-order on $\pi_1(M_i)$.
 
\end{theorem}

The fundamental group of an irreducible rational homology $3$-sphere graph manifold which admits a co-oriented taut foliation acts by orientation-preserving homeomorphisms on $S^1$ via Thurston's universal circle construction \cite{CD}, and hence is circularly-orderable.  One can promote the circular-ordering to a left-ordering whenever the action lifts to an action on $\mathbb{R}$, but the existence of such a lift depends upon the vanishing of an obstruction in the finite group $H^2(W)$. It would be interesting to see how the hypotheses of the gluing theorem can be used to show that this obstruction can be made to vanish. 

As mentioned above, statements (2) and (3) of Theorem \ref{equiv 1} are known to be equivalent (cf. \cite{Linnell}, \cite[Theorem 1.1(1)]{BRW}). The remaining equivalences claimed in Theorem \ref{equiv 1} are immediate consequences of Theorems \ref{theorem: special detection}  and \ref{theorem: special gluing}. Theorems \ref{equiv 2} and \ref{equiv 3} will follow in a similar fashion. 

Various problems and questions arise naturally from this study, most importantly with regards to the Heegaard-Floer aspects of the detection theorem and the potential for a Heegaard-Floer version of the gluing theorem. 

\begin{question}
For a given $t \geq 2$, is NLS detection determined exclusively in terms of $N_t$? In particular, is it determined in terms of the twisted $I$-bundle over the Klein bottle $N_2$ (cf. Remark \ref{NLS reduction} )? (We expect this to be the case.) More generally, can an arbitrary Heegaard-Floer solid torus with incompressible boundary be used to determine NLS detection? 
\end{question}

Although the definition of NLS detection is extrinsic to the ambient manifold, we expect that there to be an intrinsic definition.  

\begin{problem}
Determine an intrinsic definition of NLS detection in terms, for instance, of the bordered Heegaard-Floer theory of the ambient manifold. Do this in such a way so as to remove the restriction that NLS detection be defined only for families of rational slopes.   
\end{problem}

Conjecture 1 of \cite{BGW} contends that an irreducible rational homology $3$-sphere $W$ is not an L-space if and only if its fundamental group is left-orderable. Consideration of Theorem \ref{LO implies NLS} reduces the conjecture to showing that if $W$ has a non-left-orderable fundamental group then $W$ is not an L-space. From the point of view of the detection and gluing theorems, this leads to the following conjecture. 

\begin{conjecture} \label{hf gluing conjecture}
The gluing theorem holds in the context of NLS detection. That is, a rational homology $3$-sphere graph manifold $W$ is {\it not} an L-space if and only if there is an $m$-tuple of rational slopes $[\alpha_*] = ([\alpha_1], [\alpha_2], \ldots , [\alpha_m])$, one for each JSJ torus of $W$, such that for each $i$, $[\alpha_*^{(i)}]$ is NLS detected (cf. Theorem \ref{theorem: special gluing}).
\end{conjecture}
The conjecture is unknown even when $W$ has only two pieces. Here is the simplest open case. 

\begin{problem} \label{special case gluing} 
Show that the union $W = M_1 \cup M_2$ of two trefoil exteriors along $T = \partial M_1 = \partial M_2$ is an L-space if and only if the set of slopes on $T$ detected from $M_1$ is disjoint from the set of slopes on $T$ detected from $M_2$. 
\end{problem}

Our notions of order detection, representation detection, foliation detection and NLS detection extend to general compact connected orientable $3$-manifolds whose boundaries consist of tori. 

\begin{question}
To what extent do the detection and gluing theorems hold in this more general setting? (Compare with \cite[Conjecture 4.3]{CLW}.) 
\end{question}

This question is of particular importance considering the approach to Conjecture 1 of \cite{BGW} and Conjecture 5 of \cite{Ju} developed in this paper.

Finally, we ask:

\begin{question} \label{smooth question}
If a graph manifold rational homology $3$-sphere admits a co-oriented taut foliation, does it admit a smooth co-oriented taut foliation? 
\end{question}

Brittenham, Naimi and Roberts have constructed examples of graph manifolds with positive first Betti number which admit $C^0$ co-oriented taut foliations but no $C^2$ co-oriented taut foliations, and their construction is generic in terms of the gluing maps of the graph manifold's JSJ pieces (\cite[Theorem D]{BNR}). On the other hand, we noted above that the constructions used in the proof of Theorem \ref{equiv 1} show that the generic graph manifold rational homology $3$-sphere which admits a co-oriented taut foliation admits a smooth co-oriented taut foliation (cf. \S \ref{subset: smoothness}). It is possible that they all admit smooth co-oriented taut foliations. 

Here is how the paper is organised. Background material is introduced in \S \ref{section: assumptions} while the notions of representation detection, order detection and foliation detection are developed in \S \ref{section: detecting via reps}, \S \ref{section: detecting via orders} and \S \ref{section: detecting via foliations} respectively. The goal of \S \ref{section: Left-orders, dynamic realisations, foliations and slopes} is to show how to relate represention detection to both order detection and foliation detection in a slope preserving fashion. The equivalence of statements (2) and (3) of Theorems \ref{equiv 2} and \ref{equiv 3} are dealt with there. See Remark \ref{orders versus reps}. We introduce NLS-detection (i.e. not an L-space detection) in \S \ref{section: detecting via L-spaces} and develop the background to show that it is equivalent to foliation detection when restricted to families of rational slopes. One of the main technical results of this paper is the slope detection theorem, Theorem \ref{theorem: detection}, which is stated and proved in \S \ref{section detection}. The second main technical result is the gluing theorem, Theorem \ref{theorem: gluing}, which is stated in \S \ref{section gluing} and then proved over the next two sections. Proposition \ref{rational leaves} is proved in  \S \ref{proof foliation case gluing}. We make a few comments on smoothness issues and provide examples which illustrate Theorems \ref{equiv 1}, \ref{equiv 2} and \ref{equiv 3} and their differences in \S \ref{examples and smoothness}. Finally we collect the results of Eisenbud, Hirsch and Neumann, of Jankins and Neumann, and of Naimi on representations of fundamental groups of Seifert manifolds with values in $\widetilde{\hbox{Homeo}}_+(S^1)$ in Appendix \ref{sec: ehnjn} and translate them into the form needed for the purposes of this paper. 

{\bf Acknowledgements}.  The authors would like to thank Michel Boileau and Liam Watson for helpful conversations concerning the material of this paper. They would also like to thank Danny Calegari for insightful comments on an earlier version of the manuscript and Jonathan Bowden who pointed out a gap in our original proof of Theorem \ref{LO implies NLS} and who provided the technique needed to fill it. Finally, they would like to thank an anonymous referee for a detailed report which led to a greatly improved exposition.

\section{Assumptions and notation} \label{section: assumptions} 
We introduce assumptions and notation here which will be used throughout the paper. 

\subsection{Slopes}  \label{slopes} A {\it slope} on a torus $T$ is the class $[\alpha]$ of a non-zero element $\alpha \in H_1(T; \mathbb R)$ in the projective space 
$$\mathcal{S}(T) = \mathbb P^1(H_1(T; \mathbb R)) \cong S^1$$ 
We call a slope on $T$ {\it rational} if it is represented by a class $\alpha \in H_1(T)$. Otherwise we call it {\it irrational}. 

A {\it rational longitude} of a compact, connected, orientable $3$-manifold $N$ with boundary a torus is a primitive class $\lambda_N \in H_1(\partial N)$ which represents a torsion element when considered as an element of $H_1(N)$. Rational longitudes exist and are well-defined up to sign. Thus they determine a well-defined slope $[\lambda_N] \in \mathcal{S}(\partial N)$. 

\subsection{Seifert manifolds} \label{assumptions seifert} 
Throughout this paper $P$ will denote a punctured $2$-sphere, $Q$ a punctured projective plane and $Q_0$ a M\"{o}bius band. 
We use $M$ to denote a compact, connected, orientable Seifert fibred $3$-manifold, distinct from $S^1 \times D^2$ and $S^1 \times S^1 \times I$, whose boundary is a non-empty union of tori $T_1, \ldots , T_r$. We also assume that $M$ embeds in a rational homology $3$-sphere. Equivalently, $M$ has  base orbifold of the form $P(a_1, a_2, \ldots, a_n)$ or $Q(a_1, a_2, \ldots, a_n)$ where $n \geq 0$ and $a_1, \ldots , a_n \geq 2$. The Seifert fibring on $M$ is unique up to isotopy unless $M$ is a twisted $I$-bundle over the Klein bottle, denoted $N_2$, which admits exactly two isotopy classes of Seifert structures. One has base orbifold $Q_0$ and the other has base orbifold $D^2(2,2)$. Let $h_0, h_1 \in H_1(\partial N_2)$ denote, respectively, primitive classes carried by a Seifert fibre of the structure with base orbifold $Q_0$, respectively $D^2(2,2)$. Then $\{h_0, h_1\}$ is a basis of $H_1(\partial N_2)$ well-defined up to sign change of $h_0$ or $h_1$. The rational longitude of $N_2$ is represented by $h_0$.  

When $M \not \cong N_2$ the class of a regular Seifert fibre of $M$ is well-defined up to taking inverses and we use $h \in \pi_1(M)$ to denote it. For each boundary component $T_j$ of $M$ we will also use $h$ to denote a primitive class of $H_1(T_j)$ represented by a Seifert fibre. When $M \cong N_2$, $h$ will correspond to either $h_0$ or $h_1$, depending on the Seifert structure chosen for $M$. 

Define 
$$\mathcal{S}(M) = \{([\alpha_1], [\alpha_2], \ldots , [\alpha_r]) : [\alpha_j] \in \mathcal{S}(T_j) \hbox{ for each } j\} \cong (S^1)^r$$ 

We call $[\alpha_*] \in \mathcal{S}(M)$ {\it rational} if each $[\alpha_j]$ is rational, and {\it horizontal} if no $[\alpha_j]$ coincides with the slope of the fibre class $[h]$.  We use $v([\alpha_*])$ will denote the number of vertical $[\alpha_j]$: 
$$v([\alpha_*]) = |\{j : [\alpha_j] = [h]\}|$$ 
Thus $[\alpha_*]$ is horizontal if and only if $v([\alpha_*]) = 0$. 

Without loss of generality we suppose that the Seifert invariants $(a_1, b_1), \ldots, (a_n, b_n)$ 
of the exceptional fibres of $M$ satisfy $0 < b_i < a_i$ for each $i$. Set
$$\gamma_i = \frac{b_i}{a_i} \in (0,1)$$
The fundamental group of $M$ admits a presentation of the following form.

\subsubsection{Seifert manifolds over $Q(a_1, a_2, \ldots, a_n)$} \label{presentation seifert mobius} 

\begin{eqnarray}\pi_1(M) = \langle y_1,  \ldots , y_n, x_1, \ldots , x_r, z, h_0 : [x_j, h_0] = 1, [y_i, h_0] = 1, y_i^{a_i} = h_0^{b_i}, zh_0z^{-1} = h_0^{-1}, \nonumber \\ y_1y_2\ldots y_n x_1 \ldots x_r z^2 = 1 \rangle  \nonumber 
\end{eqnarray}
Here $x_j$ carries a {\it dual class} $h_j^*$ to $h_0$ on $T_j$, $1 \leq j \leq r$. This means that $\{h_0, h_j^*\}$ is a basis of $H_1(T_j) = \pi_1(T_j)$.

\subsubsection{Seifert manifolds over $P(a_1, a_2, \ldots, a_n)$} \label{presentation seifert planar} 

$$\pi_1(M) = \langle y_1,  \ldots , y_n, x_1, \ldots , x_r, h : h \hbox{ central, } y_i^{a_i} = h^{b_i}, y_1y_2\ldots y_n x_1 \ldots x_r = 1 \rangle$$ 
Again, $x_j$ carries a {\it dual class} to $h$ on $T_j$, $1 \leq j \leq r$.

\subsubsection{A special family of Seifert fibred manifolds}  \label{N_t}
For each integer $t \geq 2$ let $N_t$ be the Seifert fibred space with base orbifold a 2-disk $D^2(t,t)$ and $\gamma_1 = \frac{1}{t}, \gamma_2 = \frac{t-1}{t}$. (Thus $N_2$ is the twisted $I$-bundle over the Klein bottle, as above.) There is a unique Seifert structure on $N_t$ with an orientable base orbifold. We use $h_1 \in H_1(\partial N_t)$ to be a primitive class carried by a fibre of this structure (cf. \S \ref{presentation seifert planar}). In analogy with the case $t = 2$ we will use $h_0$ to denote a primitive class in $H_1(\partial N_t)$ representing the rational longitude of $N_t$. The reader will verify using the presentation for $\pi_1(N_t)$ in \S \ref{presentation seifert planar} that $h_0 = h_1^* + h_1$ has order $t$ in $H_1(N_t)$. Thus 
$$\Delta(h_0, h_1) = 1$$
and there is a connected, oriented, horizontal surface $F$ in $N_t$ with $t$ boundary components, each like-oriented on $\partial N_t$ and of slope $[h_0]$. 
It follows that the restriction of the Seifert map $N_t \to D^2(t,t)$ to each boundary component of $F$ is a homeomorphism onto its image $\partial D^2(t,t)$. In particular, $F$ is non-separating and so is the fibre of a locally trivial fibring $N_t \to S^1$.  

\begin{lemma} \label{homeo N_t} 
The image of $\hbox{Homeo}(N_t) \to GL_2(\mathbb Z)$ which sends $F \in \hbox{Homeo}(N_t)$ to the matrix of $(F|\partial N_t)_*: H_1(\partial N_t) \to H_1(\partial N_t)$ is given by $\{\left(\begin{smallmatrix} 1  &   0 \\ 0   &  1 \end{smallmatrix}\right), \left(\begin{smallmatrix} -1  &   0 \\ 0   &  -1 \end{smallmatrix}\right), \left(\begin{smallmatrix} -1  &   0 \\ 0   &  1 \end{smallmatrix}\right), \left(\begin{smallmatrix} 1  &   0 \\ 0   &  -1 \end{smallmatrix}\right)\}$ when expressed in terms of the basis $\{h_0, h_1\}$.  
\end{lemma}

\begin{proof}
By their definition, $h_0$ and $h_1$ are preserved by any homeomeorphism of $N_t$, at least up to sign. It is easy to see that there is an $F_1 \in \hbox{Homeo}(N_t)$ which simultaneously inverts the orientations of the base and fibre of $N_t$. Thus $(F_1|\partial N_t)_* = \left(\begin{smallmatrix} -1  &   0 \\ 0   &  -1 \end{smallmatrix}\right)$. Similarly using the fact that the exceptional fibres of $N_t$ have Seifert invariants $(t, 1)$ and $(t, t-1)$, we can construct an orientation-reversing homeomorphism $F_2$ of $N_t$ which switches the two exceptional fibres. More precisely, think of $N_t$ as obtained by Dehn filling the two ``inner" boundary components of a product $P \times S^1$ where $P$ is a twice-punctured disk. Let $r$ be a reflection of $P$ along a properly embedded arc which separates its two inner boundary components. We can post-compose $r \times 1_{S^1}$ with Dehn twists along a pair of disjoint vertical annuli in $P \times S^1$ which connect its inner boundary components to its outer one to obtain an orientation-reversing homeomorphism of $P \times S^1$ which exchanges its two inner boundary components and extends to $F_2: N_t \to N_t$.  This implies the result since $\det((F_2|\partial N_t)_*) = -1$. 
\end{proof}

\subsection{Graph manifolds} \label{assumptions graph}  
Throughout this paper $W$ will denote a graph manifold rational homology $3$-sphere. Thus $W$ contains a disjoint family of incompressible tori  $\mathcal{T} = \{T_1, T_2,  \ldots, T_m\}$ which split it into a family $M_1, M_2, \ldots, M_n$ of connected Seifert manifolds of the type described in \S \ref{assumptions seifert}. Define
$$\mathcal{S}(W; \mathcal{T}) = \{([\alpha_1], [\alpha_2], \ldots, [\alpha_m]) : [\alpha_j] \in \mathcal{S}(T_j) \hbox{ for all } j\} \cong (S^1)^m$$ 
An element $[\alpha_*]$ of $\mathcal{S}(W; \mathcal{T})$ will be called {\it rational} if each of its components is a rational slope. 

An element $[\alpha_*]$ of $\mathcal{S}(W; \mathcal{T}) $  will be called {\it horizontal} if for each $T_j$ in $\mathcal{T}$ the associated component of $[\alpha_*]$ is horizontal in the two pieces of $W$ incident to $T_j$.  

For each $i$ we have a projection map 
$$\Pi_i: \mathcal{S}(W; \mathcal{T}) \to \mathcal{S}(M_i)$$ 
which associates to each $[\alpha_*] \in \mathcal{S}(W; \mathcal{T})$ the $|\partial M_i|$-tuple of associated slopes corresponding the components of $\partial M_i$. We shall write
$$\Pi_i([\alpha_*]) = [\alpha_*^{(i)}]$$

\section{Detecting horizontal slopes via representations} \label{section: detecting via reps}

 Let $M$ be a compact orientable Seifert fibred manifold $M$ as in \S \ref{assumptions seifert}. In this section we review the results of Eisenbud-Hirsch-Neumann, of Jankins-Neumann, and of Naimi concerning the relationship between slopes on $\partial M$ and representations of $\pi_1(M)$ with values in $\widetilde{\hbox{Homeo}}_+(S^1)$.

 \subsection{Representation detection of horizontal slopes}

For $\gamma \in \mathbb R$ we use $\hbox{sh}(\gamma) \in \hbox{Homeo}_+(\mathbb R)$ to denote the translation homeomorphism $\hbox{sh}(\gamma)(x) = x + \gamma$. The universal cover $\widetilde{\hbox{Homeo}}_+(S^1)$ of $\hbox{Homeo}_+(S^1)$ can be identified in a natural way with the centraliser of $\hbox{sh}(1)$ in $\hbox{Homeo}_+(\mathbb R)$:
$$\widetilde{\hbox{Homeo}}_+(S^1) = \{f \in \hbox{Homeo}_+(\mathbb R): f(x + 1) = f(x) + 1\}$$
There is a continuous, conjugation-invariant {\it translation number} quasimorphism 
$$\tau: \widetilde{\hbox{Homeo}}_+(S^1) \to \mathbb R$$
for which $\tau(\hbox{sh}(\gamma)) = \gamma$ and which is a homomorphism when restricted to an abelian subgroup of $\widetilde{\hbox{Homeo}}_+(S^1)$ (cf. \cite[\S 5]{Ghys}). It is known that $\tau(f) = 0$ if and only if $f$ has a fixed point. 

We leave the proof of the following elementary lemma to the reader.

\begin{lemma} \label{trans} 
Let $f \in \hbox{Homeo}_+(\mathbb R)$. 

\noindent $(1)$ If $f$ has no fixed point, then it is conjugate in $\hbox{Homeo}_+(\mathbb R)$ to $\hbox{sh}(1)$ if $f(x) > x$ for all $x$ and to $\hbox{sh}(-1)$ if $f(x) < x$ for all $x$.  

\noindent $(2)$ If there is a $k \geq 1$ such that $f^k = \hbox{sh}(1)$, then there is a $g \in \widetilde{\hbox{Homeo}_+}(S^1)$ such that $g \circ f \circ g^{-1} = \hbox{sh}(\frac1k)$. 
\qed 
\end{lemma}
Let $M$ be a compact orientable Seifert fibred manifold as in \S \ref{assumptions seifert} and define
$$\mathcal{R}_0(M) = \{\rho \in \hbox{Hom}(\pi_1(M), \hbox{Homeo}_+(\mathbb R)): \rho(h) = \hbox{sh}(1)\}$$

\begin{lemma} \label{non-orientable implies empty}  
Let $M$ be a compact orientable Seifert fibred manifold $M$ as in \S \ref{assumptions seifert}.

$(1)$ If $M$ has base orbifold $Q(a_1, \ldots , a_n)$ then $\mathcal{R}_0(M) = \emptyset$.
         
$(2)$ Suppose that $M$ has base orbifold $P(a_1, \ldots , a_n)$.  

\indent \hspace{.3cm} $(a)$ $\mathcal{R}_0(M) \subset  \hbox{Hom}(\pi_1(M), \widetilde{\hbox{Homeo}}_+(S^1))$. 

\indent \hspace{.3cm} $(b)$ Consider the presentation of $\pi_1(M)$ given in \S \ref{presentation seifert planar}. Then for each $\rho \in \mathcal{R}_0(M)$ and \\
\indent \hspace{1cm} $i \in \{1,2, \ldots , n\}$, $\rho(y_i)$ is conjugate to $\hbox{sh}(\gamma_i)$. 
\end{lemma}

\begin{proof}
Suppose that $M$ has base orbifold $Q(a_1, \ldots , a_n)$ and $\rho \in \mathcal{R}_0(M)$. There is an element $z \in \pi_1(M)$ such that $z h z^{-1} = h^{-1}$ (cf. the presentation of $\pi_1(M)$ given in \S \ref{presentation seifert mobius}). Since $\rho(z) \in \hbox{Homeo}_+(\mathbb R)$, $\rho(z)(x) < \rho(z)(y)$ for each pair of real numbers $x < y$. But then for $x \in \mathbb R$, $\rho(z)(x) < \rho(z)(x + 1) = \rho(z)(\rho(h)(x)) = \rho(h^{-1})(\rho(z)(x)) =  \rho(z)(x) - 1 < \rho(z)(x)$, a contradiction. Thus assertion (1) of the lemma holds. 

If $M$ has base orbifold $P(a_1, \ldots , a_n)$ then $h$ is central in $\pi_1(M)$ (\S \ref{presentation seifert planar}) so the image of any $\rho \in \mathcal{R}_0(M)$ is contained in $\widetilde {\hbox{Homeo}}_+(S^1)$. Therefore part (a) of assertion (2) holds. Part (b) follows immediately from Lemma \ref{trans}(2).
\end{proof}

If $\rho \in \mathcal{R}_0(M)$, then for each $1 \leq j \leq r$, 
$\hbox{kernel}((\tau \circ \rho|) \otimes {\bf{1}_{\mathbb{R}}}: \pi_1(T_j) \otimes \mathbb R = H_1(T_j; \mathbb R) \to \mathbb R) \cong \mathbb R$ and hence determines a slope $[\alpha_j(\rho)] \in \mathcal{S}(T_j)$. Note that if $h_j^* \in H_1(T_j)$ is the dual class to $h$ corresponding to $x_j$, then as $\rho(h) = \hbox{sh}(1)$ we have 
$$[\alpha_j(\rho)] = [\tau(\rho(h_j^*)) h - h_j^*]$$
Thus $[\alpha_j(\rho)]$ is horizontal. We call $[\alpha_*(\rho)] = ([\alpha_1(\rho)], [\alpha_2(\rho)], \ldots , [\alpha_r(\rho)])$ the {\it slope} of $\rho$. 

\begin{definition}
{\rm Let $\rho \in \mathcal{R}_0(M)$. A slope $[\alpha_j] \in \mathcal{S}(T_j)$ is {\it detected} by $\rho$, or $\rho$-{\it detected}, if $[\alpha_j] = [\alpha_j(\rho)]$. It is {\it strongly $\rho$-detected} if it is $\rho$-detected and $\rho|\pi_1(T_j)$ conjugates into the translation subgroup of $\mathbb R$. For $J \subset \{1, 2, \ldots, r\}$ and $[\alpha_*] = ([\alpha_1], [\alpha_2], \ldots , [\alpha_r]) \in \mathcal{S}(M)$, we say that $(J; [\alpha_*])$ is $\rho$-{\it detected} if $\rho|\pi_1(T_j)$ detects $[\alpha_j]$ for all $j$ and strongly detects $[\alpha_j]$ for $j \in J$. Finally, we say that $(J; [\alpha_*])$ is {\it representation}-{\it detected} if it is $\rho$-detected for some $\rho \in \mathcal{R}_0(M)$. 
}
\end{definition}  

We shall often simplify the phrase ``$(\emptyset; [\alpha_*])$ is $\rho$-detected, resp. representation detected", to ``$[\alpha_*]$ is $\rho$-detected, resp. representation detected". Similarly, we simplify ``$(\{1, 2, \ldots, r\}; [\alpha_*])$ is $\rho$-detected, resp. representation detected", to ``$[\alpha_*]$ is {\it strongly} $\rho$-detected, resp. {\it strongly} representation detected".

Set 
$$\mathcal{D}_{rep}(M; J) = \{[\alpha_*] \in \mathcal{S}(M) : (J; [\alpha_*]) \mbox{ is representation detected}\}$$ 
When $J = \emptyset$ we will often simplify $\mathcal{D}_{rep}(M; J)$ to $\mathcal{D}_{rep}(M)$. 

Determining $\mathcal{D}_{rep}(M; J)$ is a subtle problem which was completely resolved in a series of papers by Eisenbud, Hirsch, Neumann, Jankins and Naimi (\cite{EHN}, \cite{JN2}, \cite{Na}). See Appendix \ref{sec: ehnjn}. The interested reader should also see \cite{CW}, and in particular Theorem 3.9 of that paper, for a simpler, more direct approach to these results. One of the main results of this area implies that if $(J; [\alpha_*])$ is representation detected, then it is $\rho$-detected where $\rho$ takes values in a certain family of $3$-dimensional Lie groups. We describe this result next. 

\subsection{JN-realisability} 

For a subset $J$ of $\{1, 2, \ldots , r\}$ and an $r$-tuple  $(\tau_1, \ldots ,\tau_r) \in \mathbb R^r$, we say that $(J; 0; \gamma_1, \ldots , $ $\gamma_n; \tau_1, \ldots ,\tau_r)$ is {\it JN-realisable} (after Jankins-Neumann) if there is some $\rho \in \mathcal{R}_0(M)$ such that $\tau_j = \tau(\rho(x_j))$ ($1 \leq j \leq r$) and $\rho(x_j)$ is conjugate to $\hbox{sh}(\tau_j)$ for $j \in J$. (Our notation differs slightly from that of \cite{JN2}.) Clearly $(J; 0; \gamma_1, \ldots , $ $\gamma_n; \tau_1, \ldots ,\tau_r)$ is JN-realisable if and only if $(J; [\alpha_*])$ is representation detected where $[\alpha_j] = [\tau_j h - h_j^*]$ ($1 \leq j \leq r$). 

More generally, given $J \subseteq \{1, 2, \ldots , r\}$, $b \in \mathbb Z$ and $(\tau_1, \ldots ,\tau_r) \in \mathbb R^r$, we say that $(J; b; \gamma_1, \ldots , $ $\gamma_n; \tau_1 \ldots ,\tau_r)$ is {\it JN-realisable} if there exist $f_1, \ldots, f_n, g_1, \ldots, g_r \in \widetilde{\hbox{Homeo}_+}(S^1)$ such that
\begin{itemize}

\vspace{.1cm} \item $f_i$ is conjugate to $\hbox{sh}(\gamma_i)$ for $1 \leq i \leq n$;

\vspace{.1cm} \item $\tau(g_j) = \tau_j$ for $1 \leq j \leq r$; 

\vspace{.1cm} \item $g_j$ is conjugate to $\hbox{sh}(\tau_j)$ for each $j  \in J$; 

\vspace{.1cm} \item $f_1 \circ \ldots \circ f_n \circ g_1 \circ \ldots \circ g_r = \hbox{sh}(b)$. 

\end{itemize} 
 
If $f_1, \ldots , f_n, g_1, \ldots , g_r$ satisfying these conditions can be chosen to lie in a subgroup $\widetilde G$ of $\widetilde{\hbox{Homeo}_+}(S^1)$, we say that $(J; b; \gamma_1, \ldots , \gamma_n; \tau_1 \ldots ,\tau_r)$ is {\it JN-realisable in $\widetilde G$}.  

A particularly important family of subgroups $\widetilde G$ of $\widetilde{\hbox{Homeo}}_+(S^1)$ correspond to the universal covers $\widetilde{PSL}(2,\mathbb R)_k$ of the $k$-fold cyclic covers $PSL(2, \mathbb R)_k$ of $PSL(2, \mathbb R)$ ($k \geq 1$). These groups are conjugate in $\hbox{Homeo}_+(\mathbb R)$, though not in $\widetilde{\hbox{Homeo}}_+(S^1)$. More precisely, let $F_k: \mathbb R \to \mathbb R$ be the homeomorphism $F_k(x) = kx$. Then 
$$\widetilde{PSL}(2,\mathbb R)_k = F_k^{-1} \widetilde{PSL}(2,\mathbb R)F_k$$
Note that $\widetilde{PSL}(2,\mathbb R)_1 = \widetilde{PSL}(2,\mathbb R)$. 

The elements of $\widetilde{PSL}(2,\mathbb R)_k$ are either {\it elliptic, parabolic} or {\it hyperbolic} depending on whether the image in $PSL(2, \mathbb R)$ of its conjugate by $F_k$ has that property. Thus an element is elliptic if and only if it is conjugate to a translation. The parabolic and hyperbolic elements of $\widetilde{PSL}(2,\mathbb R)$ have integral translation numbers, so the translation number of a parabolic or hyperbolic element of $\widetilde{PSL}(2,\mathbb R)_k$ is of the form $\frac{d}{k}$ where $d \in \mathbb Z$.  

\begin{theorem} \label{realized standardly} {\rm (\cite{EHN}, \cite{JN2}, \cite{Na})} 
$(J; b; \gamma_1, \ldots , \gamma_n; \tau_1 \ldots , \tau_r)$ is JN-realisable if and only if it is JN-realisable in $\widetilde{PSL}(2,\mathbb R)_k$ for some $k \geq 1$. 
\end{theorem} 

\begin{proof}
The conclusion of this theorem is the substance of \cite[Conjecture 1]{JN2} whose proof is a consequence of results contained in \cite{EHN}, \cite{JN2}, \cite{Na}. See the discussion at the end of \cite[\S 1]{JN2} and \cite[Theorem 1]{Na}. 
\end{proof}

\subsection{JN-realisability and representation detection} 

\begin{proposition} \label{detected in psl2k}
Let $J \subset \{1, 2, \ldots, r\}$ and suppose that $[\alpha_*] = ([\alpha_1], [\alpha_2], \ldots , [\alpha_r]) \in \mathcal{S}(M)$ is horizontal. Then $(J; [\alpha_*])$ is 
representation detected if and only if it is $\rho$-detected for some $\rho \in \mathcal{R}_0(M)$ with values in some $\widetilde{PSL}(2,\mathbb R)_k$.
\end{proposition} 

\begin{proof}
Since $[\alpha_*]$ is horizontal we can find real numbers $\tau_1, \tau_2, \ldots, \tau_r$ such that $[\alpha_j] = [\tau_j h - h_j^*]$. 
Then $(J; [\alpha_*])$ is representation detected if and only if $(J; 0; \gamma_1, \ldots , \gamma_n; \tau_1 \ldots ,\tau_r)$ is JN-realisable and by Theorem \ref{realized standardly} this is equivalent to it being JN-realisable in $\widetilde{PSL}(2,\mathbb R)_k$ for some $k \geq 1$, which is what we had to prove.
\end{proof}

\begin{corollary} \label{irrational implies strong} 
Suppose that $J \subset \{1, 2, \ldots, r\}$ and $(J; [\alpha_*])$ is representation detected where $[\alpha_*] \in \mathcal{S}(M)$ is horizontal. Then $(J^\dagger; [\alpha_*])$ is representation detected where $J^\dagger = J \cup \{j : [\alpha_j] \hbox{ is irrational}\}$. 
\end{corollary} 

\begin{proof}
It follows from Proposition \ref{detected in psl2k} that $(J; [\alpha_*])$ is $\rho$-detected for some $\rho \in \mathcal{R}_0(M)$ with values in some $\widetilde{PSL}(2,\mathbb R)_k$. In particular $[\alpha_j] = [\alpha_j(\rho)] = [\tau_j(\rho(x_j)) h - h_j^*]$. If $[\alpha_j]$ is irrational then so is $\tau_j(\rho(x_j))$. But as $\rho(x_j) \in \widetilde{PSL}(2,\mathbb R)_k$, it is therefore elliptic and so is conjugate to a translation. Since $\rho(h) = \hbox{sh}(1)$, this implies that $\rho(\pi_1(T_j))$ conjugates into the group of translations of $\mathbb R$. Thus $[\alpha_j]$ is strongly $\rho$-detected, which completes the proof. 
\end{proof}

\begin{corollary} \label{irrational to rational - reps} 
Suppose that $J \subset \{1, 2, \ldots, r\}$ and $(J; [\alpha_*])$ is representation detected where $[\alpha_*] \in \mathcal{S}(M)$ is horizontal and some $[\alpha_j]$ is irrational. Reindex the boundary components of $M$ so that $[\alpha_j]$ is irrational if and only $1 \leq j \leq s$ and set $J^\dagger = J \cup \{1, 2, \ldots , s\}$. Then for $1 \leq j \leq s$ there is an open sector $U_j \subset \mathcal{S}(T_j)$ containing $[\alpha_j]$ such that one of the following two statements holds.

$(1)$ $(J^\dagger; [\alpha_*'])$ is representation detected for all $[\alpha_*']$ such that $[\alpha_j'] \in U_j$ for $1 \leq j \leq s$ and $[\alpha_j'] = [\alpha_j]$ otherwise. 

$(2)$ $M$ has no singular fibres, $s = 2$, $J^\dagger = \{1, 2, \ldots , r\}$ and $[\alpha_j] = [\tau_j h - h_j^*]$ where $\tau_3, \ldots, \tau_r \in \mathbb Z$. Further, there is a homeomorphism $\varphi: U_1 \to U_2$ which preserves both rational and irrational slopes and for which $(J^\dagger; [\alpha_*'])$ is representation detected for all $[\alpha_*'] = ([\alpha_1'], \varphi([\alpha_1']), [\alpha_3], \ldots, [\alpha_r])$ whenever $[\alpha_1'] \in U_1$. 

\end{corollary} 

\begin{proof}
By Corollary \ref{irrational implies strong} it suffices to deal with the case that $J^\dagger = J$. Let $n$ be the number of exceptional fibres in $M$ and let $r_1$ and $s_0$ be the non-negative integers defined in Appendix \ref{sec: ehnjn}. Note that $r_1 \geq 1$ by hypothesis. If $n+r_1+s_0 \geq 2$ then Proposition \ref{tau fibre 2} implies that statement (1) holds. Otherwise $n = s_0 = 0$ and $r_1 = 1$, in which case the proof of Proposition \ref{tau fibre 1} implies that statement (2) holds. 
\end{proof}

\section{Detecting slopes via left-orders} \label{section: detecting via orders}

\subsection{Left-orders} 
Let $\mathfrak{o}$ be a strict total ordering of a group $G$ and use $<$ to denote the associated relation. We say that $\mathfrak{o}$ is a {\it left-ordering} of $G$ if $<$ is invariant under left multiplication:
$$g < h  \Rightarrow fg < fh \; \hbox{ for all } \; f,g, h \in G$$
We call $G$ {\it left-orderable} if it admits a left-ordering. 
While the trivial group satisfies the criterion for being left-orderable, we will adopt the convention that it is {\it not} left-orderable in this paper. 
We use $LO(G)$ to denote the set of left-orderings on $G$. 

For example, the group $\hbox{Homeo}_+(\mathbb R)$ is left-orderable \cite{Co} (cf. the proof of Proposition \ref{reps to orders}), as is any of its non-trivial subgroups. Moreover, a countable group is left-orderable if and only if it is isomorphic to a non-trivial subgroup of $\hbox{Homeo}_+(\mathbb R)$ (\cite{Linnell}). See Proposition \ref{prop:dynamical}. A much stronger result holds for many $3$-manifold groups: {\it The fundamental group of a compact $\mathbb P^2$-irreducible $3$-manifold is left-orderable if and only if it admits an epimorphism to a left-orderable group} (\cite[Theorem 1.1(1)]{BRW}). Thus for $W$ a graph manifold rational homology $3$-sphere, $\pi_1(W)$ is left-orderable if and only if there is a homomorphism $\pi_1(W) \to \hbox{Homeo}_+(\mathbb R)$ whose image is non-trivial. 

Given $\mathfrak{o} \in LO(G)$, we call an element $g \in G$ {\it $\mathfrak{o}$-positive}, or simply {\it positive}, if $g>1$. Similarly we call $g$ {\it negative} if $g<1$.  The set $P(\mathfrak{o})$ of $\mathfrak{o}$-positive elements of $G$ is called the {\it positive cone} of $\mathfrak{o}$, which we simplify by writing $P$ when there is no risk of ambiguity. A left-ordering is uniquely determined by its positive cone, for a subset $P$ of $G$ which is closed under multiplication and for which $G = \{1\} \sqcup P \sqcup P^{-1}$ uniquely determines a left-ordering $\mathfrak{o}$ by defining $g_1 < g_2$ if and only if $g_1^{-1} g_2 \in P$. Evidently $P = P(\mathfrak{o})$. 

The {\it opposite order} of $\mathfrak{o}$ is the order $\mathfrak{o}_{op}$ defined by the subset $P(\mathfrak{o})^{-1}$ of $G$. 

\subsection{Order detection of slopes} It is elementary to verify the following proposition. 

\begin{proposition} \label{prop:slope}  {\rm (\cite[Lemma 3.3]{CR})}
Every left-ordering $\mathfrak{o}$ of $\mathbb{Z}^2$ determines a unique line $L_{\mathfrak{o}}$ in $\mathbb{R}^2$ characterised by the property that all elements of $\mathbb{Z}^2$ lying to one side of $L_{\mathfrak{o}}$ are positive and all elements lying  to the other are negative. 
\qed 
\end{proposition}
It follows that every left-ordering $\mathfrak{o}$ of the fundamental group of a torus $T$ determines a unique slope 
$[\alpha(\mathfrak{o})] \in \mathcal{S}(T)$.  

Every non-trivial subgroup $K$ of $G$ is left-ordered by the restriction of the ordering $\mathfrak{o}$, which  will be denoted by $\mathfrak{o}|K$, or simply $\mathfrak{o}$ when there is no risk of ambiguity.  A subgroup $C$ of $G$ is called {\it $\mathfrak{o}$-convex} if whenever $g_1, g_2 \in C$ and $g_0 \in G$ satisfy $g_1 < g_0 < g_2$, then $g_0 \in C$.  The condition of $\mathfrak{o}$-convexity is equivalent to requiring that the left cosets $\{ gC \}_{g \in G}$ inherit a well-defined ordering from $\mathfrak{o}$ that is invariant under the left action of $G$ on $\{ gC \}_{g \in G}$.  In particular, when $C$ is normal, the quotient $G/C$ is left-orderable.

Conversely, if $1 \rightarrow K \rightarrow G \rightarrow H \rightarrow 1 $ is a short exact sequence of groups and $K$ and $H$ are left-ordered, then $G$ can be left-ordered lexicographically so that $K \subset G$ becomes a convex subgroup.  When $G$ is the fundamental group of a compact $\mathbb P^2$-irreducible $3$-manifold, we need only know that the quotient $H$ is left-orderable in order to lexicographically order $G$ in this way, because $G$ (and thus $K$) is left-orderable by \cite[Theorem 1.1(1)]{BRW}.

Let $M$ be a Seifert fibred manifold as in \S \ref{assumptions seifert}.  In what follows we consider $\pi_1(T_j) \cong H_1(T_j)$ as a subgroup of $H_1(T_j, \mathbb{R})$, and denote by $\langle \alpha_j \rangle$ the vector subspace of $H_1(T_j, \mathbb{R})$ generated by $\alpha_j$.

\begin{definition} \label{def: o-detected}
{\rm Let  $\mathfrak{o}$ be a left-ordering of $\pi_1(M)$.   A slope $[\alpha_j] \in \mathcal{S}(T_j)$ is {\it detected} by $\mathfrak{o}$, or $\mathfrak{o}$-{\it detected}, if  $[\alpha_j] = [\alpha(\mathfrak{o} | \pi_1(T_j))]$.  In this case we will simply write $[\alpha_j] = [\alpha_j(\mathfrak{o})]$.  It is \textit{strongly $\mathfrak{o}$-detected} if it is $\mathfrak{o}$-detected and there is an $\mathfrak{o}$-convex, normal subgroup $C$ of $\pi_1(M)$ such that $C \cap \pi_1(T_j) = \langle \alpha_j \rangle \cap \pi_1(T_j)$.  For $J \subset \{1, 2, \ldots, r\}$ and $[\alpha_*] = ([\alpha_1], [\alpha_2], \ldots , [\alpha_r]) \in \mathcal{S}(M)$, we say that $(J; [\alpha_*])$ is $\mathfrak{o}$-{\it detected} if $[\alpha_j]$ is $\mathfrak{o}$-detected for all $j$ and there exists a $\mathfrak{o}$-convex, normal subgroup $C$ of $\pi_1(M)$ such that $C \cap \pi_1(T_j) = \langle \alpha_j \rangle \cap \pi_1(T_j)$ for $j \in J$ and $C \cap \pi_1(T_j) \leq \langle \alpha_j \rangle \cap \pi_1(T_j)$ otherwise.  }
\end{definition} 

We'll often simplify ``$(\emptyset; [\alpha_*])$ is $\mathfrak{o}$-detected, resp. order detected", to ``$[\alpha_*]$ is $\mathfrak{o}$-detected, resp. order detected", and ``$(\{1, 2, \ldots, r\} ; [\alpha_*])$ is $\mathfrak{o}$-detected, resp. order detected", to ``$[\alpha_*]$ is \textit{strongly} $\mathfrak{o}$-detected, resp. order detected".   

Set
\[ \mathcal{D}_{ord}(M ; J) = \{ [\alpha_*] \in \mathcal{S}(M) : (J; [\alpha_*]) \mbox{ is order detected} \}
\]
When $J = \emptyset$ we write $\mathcal{D}_{ord}(M) $ in place of $\mathcal{D}_{ord}(M ; J)$.

\begin{remarks} \label{irrational implies strong order detection} 
{\rm (1) Order detected irrational slopes are always strongly order detected; if $[\alpha_j]$ is irrational then $\langle \alpha_j \rangle \cap \pi_1(T_j) = \{ 1\}$ is contained in every convex subgroup $C$ of $G$.  Therefore whenever $(J; [\alpha_*])$ is $\mathfrak{o}$-detected we can enlarge the set $J$ to create a new set $J^\dagger$ that contains all $j$ for which $[\alpha_j]$ is irrational, and $(J^\dagger; [\alpha_*])$ is $\mathfrak{o}$-detected.

(2) The reader will verify that $(J; [\alpha_*])$ is $\mathfrak{o}$-detected if and only if it is $\mathfrak{o}_{op}$-detected. Thus if $(J; [\alpha_*])$ is order detected it is $\mathfrak{o}$-detected where $h > 1$. }
\end{remarks}

\begin{lemma}
\label{order_fill_lemma}
Fix $[\alpha_*] \in \mathcal{S}(M)$ and $J \subset \{1, 2, \ldots, r\}$ such that $\{j \in J : [\alpha_j] = [h]\} = \emptyset$. Define $J_0 = \{j \in J : [\alpha_j] \hbox{ is rational}\}$, $J' = J \setminus J_0$, and let $M'$ be the Seifert manifold obtained by $[\alpha_j]$-Dehn filling of $M$ where $j \in J_0$.  Suppose that $M' \neq S^1 \times D^2$. Then $(J; [\alpha_*])$ is order detected in $M$ if and only if either $\partial M' \neq \emptyset$ and $(J'; [\alpha_*'])$ is order detected in $M'$, or $\partial M' = \emptyset$ and $\pi_1(M')$ is left-orderable.
\end{lemma}

\begin{proof}  
Suppose that $M'$ is closed. Then $J_0 = J = \{1, \ldots , r\}$, so our assumptions imply that $[\alpha_*]$ is rational and horizontal. If $(J; [\alpha_*])$ is $\mathfrak{o}$-detected, there is a $\mathfrak{o}$-convex, normal subgroup $C$ of $\pi_1(M)$ such that $C \cap \pi_1(T_j) = \langle \alpha_j \rangle \cap \pi_1(T_j)$ for all $j$. The quotient homomorphism $\pi_1(M) \to \pi_1(M)/C$ induces a left-order $\mathfrak{o}'$ on $G = \pi_1(M)/C \ne \{1\}$ and factors through an epimorphism $\pi_1(M') \to G$. It follows that $M' \not \cong P^3 \# P^3$ and so is prime. Therefore by \cite[Theorem 1.1]{BRW}, $\pi_1(M')$ is left-orderable. Conversely suppose that $\pi_1(M')$ is left-orderable. Since the kernel of the epimorphism $\pi_1(M) \to \pi_1(M')$ is also left-orderable, we obtain an induced left-order $\mathfrak{o}$ on $\pi_1(M)$. Now the cores of the filling tori in $M'$ cannot be null-homotopic as otherwise $M'$ would be the $3$-sphere (see \cite[Proposition 4.1]{BRW} for instance). It follows that $(J; [\alpha_*])$ is $\mathfrak{o}$-detected. 

Next suppose that $M'$ is not closed. Let $J' = J \setminus J_0$ and let $[\alpha_*']$ be the projection of $[\alpha_*]$ in $\mathcal{S}(M')$. If $(J; [\alpha_*])$ is $\mathfrak{o}$-detected, there is a $\mathfrak{o}$-convex, normal subgroup $C$ of $\pi_1(M)$ such that $C \cap \pi_1(T_j) = \langle \alpha_j \rangle \cap \pi_1(T_j)$ for $j \in J$ and $C \cap \pi_1(T_j) \leq \langle \alpha_j \rangle \cap \pi_1(T_j)$ otherwise. The quotient homomorphism $\pi_1(M) \to \pi_1(M)/C$ induces a left-order $\bar{\mathfrak{o}}$ on $G = \pi_1(M)/C \ne \{1\}$ and factors through an epimorphism $\pi_1(M') \to G$. Since the kernel of this  epimorphism is left-orderable, we obtain an induced left-order $\mathfrak{o}'$ on $\pi_1(M')$ for which $(J'; [\alpha_*'])$ is $\mathfrak{o}'$-detected. Conversely if $(J'; [\alpha_*'])$ is $\mathfrak{o}'$-detected, we can use the epimorphism $\pi_1(M) \to \pi_1(M')$ to construct a left-order $\mathfrak{o}$ on $\pi_1(M)$ for which $(J; [\alpha_*])$ is $\mathfrak{o}$-detected. 
\end{proof}

\subsection{Representation detection implies order detection} 

Given $\mathfrak{o} \in LO(G)$ and a subgroup $H \subset G$, we call $f \in G$ {\it $\mathfrak{o}$-cofinal} in $H$ if for all $g \in H$ there exists $n \in \mathbb{Z}$ such that $f^{-n} < g < f^n$.    

\begin{proposition} \label{reps to orders}
Let $M$ be a compact orientable Seifert fibred manifold  as in \S \ref{assumptions seifert}. Suppose that $J \subset \{1, 2, \ldots, r\}$ and $(J; [\alpha_*])$ is representation detected. Then $(J; [\alpha_*])$ is $\mathfrak{o}$-detected for some ordering $\mathfrak{o} \in LO(\pi_1(M))$ for which $h$ is $\mathfrak{o}$-cofinal in $\pi_1(M)$. 
\end{proposition}

\begin{proof}
An enumeration $\{0 = r_1, r_2, r_3, \ldots \}$ of the rationals yields a left-order on $\hbox{Homeo}_+(\mathbb R)$ by taking $f > f'$ if and only if $f(r_{k_0}) > f'(r_{k_0})$ where $k_0 = \min\{k : f(r_k) \ne f'(r_k)\}$. Hence if we fix $\rho \in \mathcal{R}_0(M)$ such that $(J; [\alpha_*])$ is $\rho$-detected, there is an induced left-order $\mathfrak{o}_1$ on $\hbox{image}(\rho)$. Note that as $\rho(h) = \hbox{sh}(1)$, $\rho(h)$ is $\mathfrak{o}_1$-cofinal in $\hbox{image}(\rho)$. Let $C = \hbox{ker}(\rho)$ and observe that as $C$ is a subgroup of $\pi_1(M)$, it admits a left-ordering $\mathfrak{o}_0$ (cf. \cite[Theorem 1.1(1)]{BRW}). Let $\mathfrak{o}$ denote the lexicographic ordering of $\pi_1(M)$ arising from the left-orderings $\mathfrak{o}_0$ and $\mathfrak{o}_1$, by construction $C$ is $\mathfrak{o}$-convex and $h$ is $\mathfrak{o}$-cofinal. We prove that $(J; [\alpha_*])$ is $\mathfrak{o}$-detected next.  

Since $(J; [\alpha_*])$ is $\rho$-detected, $[\alpha_j] = [\tau(\rho(x_j)) h - h_j^*]$ for each $j$. Further, $\rho(\pi_1(T_j))$ conjugates into the group of translations when $j \in J$. Since translation number is an injective homomorphism when restricted to the translation subgroup of $\widetilde{\hbox{Homeo}}_+(S^1)$ we find $C \cap \pi_1(T_j) = \hbox{ker}(\rho) \cap \pi_1(T_j) \leq  \langle \alpha_j \rangle \cap \pi_1(T_j)$ for all $j$ with equality when $j \in J$. To complete the proof it suffices to show that $[\alpha_j]$ is $\mathfrak{o}$-detected for all $j$. To that end, note that the complement of the line containing $\langle \alpha_j \rangle$ in $H_1(T_j; \mathbb R)$ is the union of two components $H_+ \cup H_-$ where $\tau(\rho(\gamma)) > 0$ for $\gamma \in \pi_1(T_j) \cap H_+$ and $\tau(\rho(\gamma)) < 0$ for $\gamma \in \pi_1(T_j) \cap H_-$. (Here we identify $\pi_1(T_j)$ with $H_1(T_j)$.) It follows that for each $k$, $\rho(\gamma)(r_k) > r_k$ when $\gamma \in \pi_1(T_j) \cap H_+$ and $\rho(\gamma)(r_k) < r_k$ when $\gamma \in \pi_1(T_j) \cap H_-$. In particular, for $\gamma \in (H_+ \cup H_-) \cap \pi_1(T_j)$ we have $\rho(\gamma) \in P(\mathfrak{o})$ if and only if $\gamma \in H_+ \cap \pi_1(T_j)$. Thus $L_{\mathfrak{o}| \pi_1(T_j)}$ contains $\langle \alpha_j \rangle$, so $[\alpha_j] = [\alpha_j(\mathfrak{o})]$. 
\end{proof}

\subsection{Order detection of horizontal $[\alpha_*]$} 

\begin{lemma}
\label{cofinal conjugation}
Let $G$ be a left-ordered group with ordering $\mathfrak{o}$.  

\noindent $(1)$ If $f \in G$ is $\mathfrak{o}$-cofinal and positive, then $gfg^{-1} > 1$ for all $g \in G$.

\noindent $(2)$ Suppose there exists $h \in G$ which is central and $\mathfrak{o}$-cofinal. If $f \in G$ is $\mathfrak{o}$-cofinal and positive, then $gfg^{-1}$ is $\mathfrak{o}$-cofinal and positive for all $g \in G$.

\end{lemma}
\begin{proof} We first prove (1).  Given $g \in G$, suppose $g<1$.    Choose $k>0$ so that $g^{-1} < f^k$, then $1 < gf^k$, and since $g^{-1}$ is positive $1<gf^kg^{-1} = (gfg^{-1})^k$.   Since the $k$-th power of $gfg^{-1}$ is positive, $gfg^{-1} >1$. The case of $g>1$ is similar.

Now assume that $h \in G$ is central and $\mathfrak{o}$-cofinal, we may assume also that $h>1$.  Suppose $g<1$ and choose $k$ so that $hg^{-1} <f^k$. Then $1<gh^{-1}f^k$, so $1< gh^{-1}f^kg^{-1}$ since $g^{-1}$ is positive.  Since $h$ is central, we find $h<gf^kg^{-1}$ which implies $h^n <(gf^kg^{-1})^n$ for all $n>0$.  Thus $gf^kg^{-1}$ is both positive and cofinal.  As above the case $g>1$ is similar.
\end{proof}

\begin{proposition} \label{non-orientable implies not detected}  
Let $M$ be a compact orientable Seifert fibred manifold $M$ as in \S \ref{assumptions seifert}.

$(1)$ Suppose that $M$ has base orbifold $Q(a_1, \ldots , a_n)$ and that $[\alpha_*] \in \mathcal{S}(M)$ is horizontal. Then $[\alpha_*]$ is not order detected. 

$(2)$ Suppose that $M$ has base orbifold $P(a_1, \ldots , a_n)$ and that $[\alpha_*] \in \mathcal{S}(M)$ is $\mathfrak{o}$-detected. Then $[\alpha_*]$ is horizontal if and only if  $h$ is $\mathfrak{o}$-cofinal in $\pi_1(M)$.  
\end{proposition}

\begin{proof}
First we establish some inequalities that must hold in $\pi_1(M)$.

Suppose that $[\alpha_*]$ is horizontal and refer to the presentations of $\pi_1(M)$ given in \S \ref{presentation seifert mobius} and \S \ref{presentation seifert planar}. Set $H = \langle y_1, \ldots , y_n, x_1, \ldots , x_r, h \rangle \subset \pi_1(M)$. We claim that $h$ is cofinal in $H$. As $h$ is central in $H$, it suffices to show that there are integers $c_1, \ldots, c_n, d_1, \ldots, d_r$ such that 
$h^{-c_i} < y_i < h^{c_i}$ ($1 \leq i \leq n$) and $h^{-d_j} < x_j < h^{d_j}$ ($1 \leq j \leq r$). The existence of the $c_i$ is obvious from the relations $y_i^{a_i} = h^{b_i}$. On the other hand, $x_j \in \pi_1(T_j)$ and therefore as $[\alpha_j]$ is horizontal, $h$ is $\mathfrak{o}$-cofinal in $\pi_1(T_j)$. Thus we can find integers $d_1, d_2, \ldots, d_r$ as claimed.  

When $M$ has base orbifold $P(a_1, \ldots , a_n)$, $H = \pi_1(M)$ and it follows that $h$ is $\mathfrak{o}$-cofinal.  Conversely, if $h$ is $\mathfrak{o}$-cofinal in $\pi_1(M)$, it is $\mathfrak{o}$-cofinal in $\pi_1(T_j)$ for each $j$. In particular $[\alpha_j]$ cannot be the fibre slope $[h] \in \mathcal{S}(T_j)$. Thus $[\alpha_*]$ is horizontal. 

On the other hand when $M$ has base orbifold $Q(a_1, \ldots , a_n)$, $h$ is cofinal since the generator $z$ in \S \ref{presentation seifert mobius} satisfies $z^2 = (y_1 \ldots  y_n x_1 \ldots  x_r)^{-1} \in H$.  By Remark \ref{irrational implies strong order detection} (2) we may assume $h>1$, but then the generators $z, h$ satisfy $z h z^{-1} = h^{-1}<1$.  This is not possible by Lemma \ref{cofinal conjugation}(1).
\end{proof}

\subsection{Order detection of non-horizontal $[\alpha_*]$}
Next we examine the order detectability of a pair $(J; [\alpha_*])$ in the case $[\alpha_*]$ is not horizontal. The main result, Proposition \ref{order vertical}, is a characterization of the non-horizontal $[\alpha_*]$ which are order detected.  

\begin{lemma}
\label{strongly detected vertical slopes}
Suppose that $(J; [\alpha_*])$ is order detected and that $\{j \in J : [\alpha_j] = [h]\}$ is nonempty.  Then $[\alpha_j] = [h]$ for all $j$.
\end{lemma}
\begin{proof}
Say that $(J; [\alpha_*])$ is $\mathfrak{o}$-detected. Then there is a convex, normal subgroup $C$ of $\pi_1(M)$ such that $C \cap \pi_1(T_j) = \langle \alpha_j \rangle \cap \pi_1(T_j)$ for $j \in J$ and $C \cap \pi_1(T_j) \leq \langle \alpha_j \rangle \cap \pi_1(T_j)$ otherwise. Since $\{j \in J : [\alpha_j] = [h]\}$ is nonempty, $C \cap \pi_1(T_j)$ contains $h$ for some, and therefore all, $j$. Thus $[h]$ is $\mathfrak{o}|\pi_1(T_j)$-detected for all $j$, which completes the proof.  \end{proof}

Next we prepare some preliminary gluing results that will be used to deal with detection of non-horizontal $[\alpha_*]$.   For a left-ordering $\mathfrak{o}$ of $G$, and for every $g \in G$, we will use $\mathfrak{o}^g$ to denote the conjugate ordering of $G$, defined by $h_1 <^g h_2$ if and only if $g h_1 g^{-1}  <g h_2 g^{-1}$.

\begin{definition}
{\rm A family of left-orderings $\mathcal{L} \subset LO(G)$ is called \textit{normal in $G$} if it is non-empty, and for all $g \in G$ if $\mathfrak{o} \in \mathcal{L}$ then $\mathfrak{o}^g \in \mathcal{L}$.}
\end{definition}

\begin{definition}
\label{compatible families}
{\rm Let $H_i \subset G_i$ be groups with left-orderings $\mathfrak{o}_i$ for $i = 1, 2$.  Suppose that $\phi_i: H_i \rightarrow H$ are isomorphisms and let $\mathcal{L}_i \subset LO(G_i)$ be families of left-orderings on $G_i$ for $i=1,2$.

We say that $\phi_2\phi_1^{-1}$ is compatible for the pair $( \mathcal{L}_1, \mathcal{L}_2)$ if and only if for every ordering $\mathfrak{o}_1 \in \mathcal{L}_1$ there exists an ordering $\mathfrak{o}_2 \in \mathcal{L}_2$ such that for all $h \in H_1$, $1 <_1 h$ implies $1  <_2 \phi_2 \phi_1^{-1}(h)$.  When this holds, we also say that $\phi_2\phi_1^{-1}$ is compatible for the pair $(\mathfrak{o}_1, \mathfrak{o}_2 )$.}
\end{definition}

\begin{theorem} \cite[Theorem A]{BG} 
\label{BG basic gluing}
Let $G_1$ and $G_2$ be left-ordered groups with orderings $\mathfrak{o}_1$ and $\mathfrak{o}_2$ respectively.  Suppose that $H_i \subset G_i$ are subgroups and that $\phi_i : H_i \rightarrow H$ are isomorphisms for $i=1,2$.  Set $\phi = \phi_2\phi_1^{-1}$.  Then $G_1 *_{\phi} G_2$ admits a left-ordering which extends each of the $\mathfrak{o}_i$ if and only if $\phi_2\phi_1^{-1}$ is compatible for the pair $(\mathfrak{o}_1, \mathfrak{o}_2 )$ and there exist normal families $\mathcal{L}_i$ such that $\phi_2\phi_1^{-1}$ is compatible for $( \mathcal{L}_1, \mathcal{L}_2)$ and $\phi_1\phi_2^{-1}$ is compatible for $( \mathcal{L}_2, \mathcal{L}_1)$.
\end{theorem}

\begin{lemma}
\label{basic normal families}
$(1)$ Suppose that $M$ is Seifert fibred over $P(a_1, \ldots, a_n)$ and $(J; [\alpha_*])$ is order detected where $\{ j \in J : [\alpha_j] = [h] \} = \emptyset$. Assume $v([\alpha_*]) \geq 2$ and $[\alpha_1] = [h]$.  Then there exists a normal family $\mathcal{L} \subset LO(\pi_1(M))$ such that the set $$\{ \mathfrak{o} \in LO(\pi_1(T_1)) :  \mbox{ $\mathfrak{o}= \mathfrak{o}'|\pi_1(T_1)$ for some $\mathfrak{o}' \in \mathcal{L}$} \} $$ contains exactly four left-orderings, all detecting $[h]$, each of which arises as the restriction of an ordering $\mathfrak{o}' \in \mathcal{L}$ which detects $(J; [\alpha_*])$.

$(2)$ Suppose $M$ is Seifert fibred over $Q_0(a_1, \ldots, a_n)$.  Then there exists a normal family $\mathcal{L} \subset LO(\pi_1(M))$ containing an ordering which detects $[h]$  such that the set $$\{ \mathfrak{o} \in LO(\pi_1(\partial M)) :  \mbox{ $\mathfrak{o}= \mathfrak{o}'|\pi_1(\partial M)$ for some $\mathfrak{o}' \in \mathcal{L}$} \} $$ contains exactly four left-orderings, all detecting $[h]$.
\end{lemma}
\begin{proof}

To prove $(1)$, let $\mathfrak{o}'$ be an ordering of $\pi_1(M)$ detecting  $(J; [\alpha_*])$.  Consider the short exact sequence
\[ 1 \rightarrow K \rightarrow \pi_1(M) \rightarrow \mathbb{Z}^{v([\alpha_*])-1} \rightarrow 1
\]
where the map $\pi_1(M) \rightarrow \mathbb{Z}^{v([\alpha_*])-1}$ is the result of killing the fibre class  $h \in \pi_1(M)$, the resulting torsion, and the dual classes $x_{v([\alpha_*])+1}, \ldots, x_r$.  Create a set $S$ of four distinct orderings of $\pi_1(M)$ by choosing an arbitrary ordering of $\mathbb{Z}^{v([\alpha_*])-1}$,  using the restrictions of $\mathfrak{o}'$ and $(\mathfrak{o}')^{op}$ on $K$, and ordering $\pi_1(M)$ lexicographically.   Set $\mathcal{L} = \{ \mathfrak{o}^g : \mathfrak{o} \in S \mbox{ and } g \in \pi_1(M)\}$. By construction, every ordering in $S$ detects $(J; [\alpha_*])$ and $\mathcal{L}$ has the required property.

For $(2)$, we use the short exact sequence 
\[ 1 \rightarrow K \rightarrow \pi_1(M) \rightarrow \mathbb{Z} \rightarrow 1
\]
arising from Dehn filling along the fibre slope and killing the resulting torsion, then argue as above.
\end{proof}

\begin{lemma}
\label{composingspace2}
Suppose that $M$ has base orbifold $P(a_1, \ldots, a_n)$ where $r \geq 2$.  If $r=2$ then there exists a left-ordering of $\pi_1(M)$ detecting $([h], [h])$; if $r \geq 3$ then for each slope $[\alpha] \in \mathcal{S}(T_r)$ there exists a left-ordering of $\pi_1(M)$ detecting $( \{r \} ; ( [h], \ldots, [h],[\alpha] ))$.

\end{lemma}

\begin{proof}
If $r=2$ or if $r \geq 3$ and $[\alpha] = [h]$ then we use the short exact sequence
\[ 1 \rightarrow K \rightarrow \pi_1(M) \rightarrow \mathbb{Z}^{r-1} \rightarrow 1
\]
where the epimorphism $\pi_1(M) \rightarrow \mathbb{Z}^{r-1}$ is the result of killing the fibre class $h \in \pi_1(M)$ as well as the resulting torsion.  Since $h \in K$, though no dual class $x_j \in \pi_1(T_j)$ is, the sequence can be used to create a lexicographic left-ordering of $\pi_1(M)$ that strongly detects $([h], \ldots, [h])$.

 Assume next that $r \geq 3$ and $\alpha$ is a primitive rational element of $H_1(T_r)$ such that $[\alpha]$ is horizontal. 
In this case, $\pi_1(M(\alpha))$ is left-orderable since it is the fundamental group of an irreducible $3$-manifold with positive first Betti number, and therefore $\pi_1(M)$ admits a left-ordering $\mathfrak{o}$ with $[\alpha_r(\mathfrak{o})] = [\alpha]$ arising from the short exact sequence $ 1 \rightarrow \langle \langle \alpha \rangle \rangle \rightarrow
\pi_1(M) \rightarrow \pi_1(M(\alpha)) \rightarrow 1 $. 

On the other hand if $r \geq 3$ and $[\alpha]$ is irrational then we use the topology on $LO(\pi_1(M))$ defined by Sikora \cite{Si}. Recall that a left-order $\mathfrak{o}$ is determined by its positive cone $P(\mathfrak{o}) \subset \pi_1(M)$. Given $x \in \pi_1(M)$, let $U_x =  \{P \in LO(\pi_1(M)) : x \in P\}$. Now endow $LO(\pi_1(M))$ with the topology with subbasic open sets $\{U_x : x \in \pi_1(M)\}$. Sikora shows that $LO(\pi_1(M))$ is compact and metrizable in this topology. He also identifies $LO(\pi_1(T_j))$ with a space $X$ which has a circle as a natural quotient and which double covers $\mathcal{S}(T_j) \cong S^1$. (This circle quotient of $X$ is simply the space of oriented slopes in $H_1(T_j; \mathbb R)$. See \cite[\S 3]{Si}.) If $H$ is a non-trivial subgroup of $\pi_1(M)$, the restriction map $LO(\pi_1(M)) \to LO(H)$ is easily seen to be continuous, so taking $H$ to be $\pi_1(T_r)$ we obtain a continuous map $LO(\pi_1(M)) \to \mathcal{S}(T_r)$ whose image contains all 
rational points.  Since $LO(\pi_1(M))$ is compact so is its image in $\mathcal{S}(T_r)$; thus every irrational slope is in the image as well. In particular  we may fix $\mathfrak{o} \in LO(\pi_1(M))$ with $[ \alpha_r(\mathfrak{o})] = [\alpha]$.

Now since $r\geq3$ we can construct, as above, a short exact sequence
\[ 1 \rightarrow K \rightarrow \pi_1(M) \rightarrow \mathbb{Z}^{r-2} \rightarrow 1 \]
with $h \in K, x_r \in K$ and $x_j \not \in K$ for $j = 1, \ldots, r-1$.  Use the restriction $\mathfrak{o} | K$ to left-order $K$, and give $\mathbb{Z}^{r-2}$ an arbitrary ordering.  The corresponding lexicographic ordering of $\pi_1(M)$ detects  $( \{r \} ; ( [h], \ldots, [h],[\alpha] ))$.
\end{proof}

\begin{lemma}
\label{gluing strong orders}
Let $M_1$ and $M_2$ be Seifert fibred manifolds with 
$\partial M_1 = T_1 \cup \cdots \cup T_s$ and $\partial M_2 = T_s \cup  \cdots \cup T_r$.  Suppose that $(J; [\alpha_*])$ is
order detected on $M_1$ and $(K;[\beta_*])$ is
order detected on $M_2$.  Set $M = M_1 \cup_{\phi} M_2$, where
$\phi:T_s \rightarrow T_s$ is a homeomorphism satisfying
$\phi_*([\alpha_s]) = [\beta_s]$. If $s \in J \cap K$, then $((J \cup
K)\setminus \{s \} ; ([\alpha_1], \ldots,[\alpha_{s-1}],
[\beta_{s+1}], \ldots, [\beta_r]))$ is order detected on $M$.
\end{lemma}

\begin{proof}  
First consider the case that $[\alpha_s]$ is rational (cf. \cite[Theorem 8]{CLW}). Suppose that $(J; [\alpha_*])$ is $\mathfrak{o}_1$-detected, and
$(K;[\beta_*])$ is $\mathfrak{o}_2$-detected. Let $C_i \subset
\pi_1(M_i)$ be a normal, $\mathfrak{o}_i$-convex subgroup such that $C_i \cap \pi_1(T_j) \subseteq \langle \alpha_j \rangle  \cap \pi_1(T_j)$
for all $j$ with equality if $j \in J$ (similarly for all $k \in K$). Each ordering $\mathfrak{o}_i$
descends to an ordering $\mathfrak{o}_i'$ of $\pi_1(M_i)/C_i$.  

There is a map $\bar{\phi}: \pi_1(T_s)/ \langle \alpha_s \rangle \rightarrow \pi_1(T_s)/\langle
\beta_s \rangle$ induced by $\phi$ which we use to amalgamate the left-orderable groups
$\pi_1(M_1)/C_1$ and $\pi_1(M_2)/C_2$ along cyclic subgroups.  We arrive a short exact
sequence
\[ 1 \rightarrow C \rightarrow \pi_1(M_1) *_{\phi} \pi_1(M_2)
\rightarrow \pi_1(M_1)/C_1 *_{\bar{\phi}} \pi_1(M_2)/C_2 \rightarrow 1
\]
where the middle term is isomorphic to the fundamental group
$\pi_1(M)$.  By \cite[Corollary 5.3]{BG},  the group  $\pi_1(M_1)/C_1
*_{\bar{\phi}} \pi_1(M_2)/C_2$ is left-orderable and admits a
left-ordering $\mathfrak{o}'$ that extends $\mathfrak{o}_1'$ and
$\mathfrak{o}_2'$.  Using the ordering $\mathfrak{o}'$ and the short
exact sequence above to left-order $\pi_1(M)$, we arrive at an
ordering $\mathfrak{o}$ of $\pi_1(M)$ detecting the tuple
$([\alpha_1], \ldots,[\alpha_{s-1}], [\beta_{s+1}], \ldots,
[\beta_r])$ of slopes.  Moreover, the kernel $C$ is convex in
$\mathfrak{o}$ and satisfies $C \cap \pi_1(T_j) = C_1 \cap \pi_1(T_j)$ for $j<s$ and $C \cap \pi_1(T_j) = C_2 \cap \pi_1(T_j) $ for $j>s$. Therefore $\mathfrak{o}$ detects
$((J \cup K)\setminus \{s \} ; ([\alpha_1], \ldots,[\alpha_{s-1}],
[\beta_{s+1}], \ldots, [\beta_r]))$.

Now consider the case that $[\alpha_s]$ is irrational. Without loss of generality we can assume that $J \cap \{1,2, \ldots, s-1\}$ and $K \cap \{s+1, s+2, \ldots, r\}$ are empty (cf. Remark \ref{irrational implies strong order detection} and Lemma \ref{order_fill_lemma}).  

Consider the torus $Z = \mathcal{S}(T_1)  \times \ldots \times \mathcal{S}(T_s) \times \ldots \times \mathcal{S}(T_r)$  
 and the closed compact subset $X = \{([\alpha_1], [\alpha_2], \ldots, [\alpha_{s-1}])\} \times \mathcal{S}(T_s) \times \{([\beta_{s+1}], [\beta_{s+2}], \ldots, [\beta_{r}])\}$ 
of $Z$. As above, there is a continuous map $s: LO(\pi_1(M)) \to Z$ which associates to an element of $LO(\pi_1(M))$ the $r$-tuple of slopes that it detects on $T_1, T_2, \ldots , T_s, \ldots ,T_r$. Then $Y = s^{-1}(X)$ is closed and compact in $LO(\pi_1(M))$, so the image of $Y$ in $\mathcal{S}(T_s)$ is closed. If there are rational slopes $[\alpha_s']$ arbitrarily close to $[\alpha_s]$ for which the lemma holds, then $[\alpha_s]$ is contained in the image of $Y$ in $\mathcal{S}(T_s)$, which implies the lemma holds. Otherwise, at least one of $M_1$ and $M_2$, say $M_1$, is of a very special sort (cf. Propositions \ref{reps to orders} and Corollary \ref{hen and detected slopes}). Set $J_0 = \{ j \in J \mid [\alpha_j] \mbox{ is rational} \}$, and let $M_1'$ denote the Seifert fibred manifold obtained from  $[\alpha_j]$-Dehn filling $M_1$ where $j \in J_0$. Indeed, in this case $M_1'$ (as in Lemma \ref{order_fill_lemma}) is just $S^1 \times S^1 \times I$ with boundary $T_1 \cup T_s$ (after reindexing $T_1, \ldots, T_{s-1}$). It follows that $[\alpha_1]$ equals $[\alpha_s]$ under the identification $\mathcal{S}(T_1) = \mathcal{S}(T_s)$ induced by $M_1'$. Thus the order detectability of $(K; [\beta_*])$ combines with the obvious homeomorphism $M \cong M_2$ to complete the proof of the lemma. 
\end{proof}

\begin{proposition} \label{order vertical}
Let $M$ be a compact orientable Seifert fibred manifold as in \S \ref{assumptions seifert} and $J \subseteq \{1, \ldots , r\}$. Fix $[\alpha_*] \in \mathcal{S}(M)$ such that $\{j \in J : [\alpha_j] = [h]\} = \emptyset$. 

$(1)$ If $M$ has base orbifold $Q(a_1, \ldots , a_n)$ then $(J; [\alpha_*])$ is order detected if and only if $v([\alpha_*]) \geq 1$. 

$(2)$ If $M$ has base orbifold $P(a_1, \ldots , a_n)$ and $v([\alpha_*]) \geq 1$ then $(J; [\alpha_*])$ is order detected if and only if $v([\alpha_*]) \geq 2$. 

\end{proposition}

\begin{proof} 
Suppose that $M$ has base orbifold $Q(a_1, \ldots , a_n)$. By Lemma \ref{non-orientable implies not detected}(1), if $(J; [\alpha_*])$ is order detected then $v([\alpha_*]) \geq 1$. Assume conversely that $v([\alpha_*]) \geq 1$. If $r = 1$, $[\alpha_*] = [\alpha_1]$ is the fibre class $[h]$. Hence $J = \emptyset$. Since $Q$ is non-orientable, $h$ is the rational longitude of $M$ and therefore $(\{1\}; [h])$ is order detected.  This implies that $( \emptyset ; [h])$ is order detected, as required.

Assume that $r > 1$ and note that $M$ splits along an essential vertical torus $T$ as the union of a Seifert manifold $N$ with base orbifold $Q_0(a_1, \ldots , a_n)$ and a Seifert manifold $M_0$ where $\partial M_0 = \partial M \cup T$. It is clear that $M_0$ has base orbifold a planar surface with $r+1 \geq 3$ boundary components.  If assertion (2) holds, then by using assertion (1) in the case $r=1$ each of $\pi_1(M_0)$ and $\pi_1(N)$ can be equipped with normal families of left-orderings $\mathcal{L}$ and $\mathcal{L}'$ respectively satisfying the hypotheses of Lemma \ref{basic normal families}.  By construction, $\mathcal{L}$ and $\mathcal{L}'$ are compatible with a gluing map which identifies the fibre of $M_0$ with the fibre of $N$ (\textit{cf.} Proposition \ref{gluing ready pieces}).  Then Theorem \ref{BG basic gluing} allows us to construct the required ordering of $\pi_1(M)$. Thus we are reduced to proving assertion (2).

Suppose that $M$ has base orbifold $P(a_1, \ldots , a_n)$, $v([\alpha_*]) = 1$, and $(J; [\alpha_*])$ is $\mathfrak{o}$-detected. Without loss of generality we can suppose that $[\alpha_1] = [h]$. Recall the presentation from \S \ref{presentation seifert planar} and observe that $\pi_1(M)$ is generated by $y_1,  \ldots , y_n, x_2, \ldots , x_r, h$.  Since $v([\alpha_*])=1$ the slopes $[\alpha_2], \ldots , [\alpha_r]$ are horizontal, and then Proposition \ref{non-orientable implies not detected} implies that $[\alpha_1 ]$ is horizontal as well.  Thus
 $(J; [\alpha_*])$ is not $\mathfrak{o}$-detected when $v([\alpha_*]) = 1$. 

Suppose that $v([\alpha_*]) \geq 2$. Without loss of generality we can suppose that $[\alpha_j ] = [h]$ for $1 \leq j \leq v([\alpha_*])$ and 
$[\alpha_j ]$ is horizontal for $v([\alpha_*]) + 1 \leq j \leq r$, and $J \subset \{v([\alpha_*])+1, \ldots, r \}$. Choose an essential vertical torus $T$ cutting $M$ into two pieces $M_1$ and $M_2$ where 
$M_1 \cong P_1 \times S^1$ ($P_1$ a planar surface) has boundary tori $T_1, \ldots, T_{v([\alpha_*])}, T$ and $M_2$
is a Seifert fibred space with boundary tori $T_{v([\alpha_*])+1}, \ldots,
T_r, T$.  Set $J' = \{ j - v([\alpha_*]) \mid j \in J \}$.  Then $J'$ is the set of indices whose slopes on $M_2$ are strongly detected.

  By applying Corollary \ref{hen and detected slopes} and Proposition \ref{reps to orders} there exists a slope $[\alpha] \in \mathcal{S}(T)$ and a left-ordering of $\pi_1(M_2)$ that detects $( J' \cup \{ r+1-v([\alpha_*]) \} ; ([\alpha_{v([\alpha_*])+1}], \ldots, [\alpha_r], [\alpha]))$, and by Lemma \ref{composingspace2}, there exists a
left-ordering of $\pi_1(M_1)$ detecting $(\{ v([\alpha_*]) \}; ([h],
\ldots, [h], [\alpha]))$.   It follows from Lemma \ref{gluing strong orders}
that $\pi_1(M) \cong \pi_1(M_1) *_{\pi_1(T)} \pi_1(M_2)$ admits a
left-ordering $\mathfrak{o}$ detecting $(J; ([h], \ldots, [h],
[\alpha_{v([\alpha_*])+1}], \ldots, [\alpha_r])$. This completes the proof of assertion (2). 
\end{proof}

\section{Left-orders, dynamic realisations and slopes}\label{section: Left-orders, dynamic realisations, foliations and slopes}

We have seen how representations yield orders in a slope preserving fashion. The goal of this section is to reverse this process. 

\subsection{The dynamical realisation of a left-ordering}
We begin with a classic result concerning countable left-orderable groups (cf. \cite[Theorem 6.8]{Ghys}, \cite[Proposition 2.1]{Navas}). We include a brief sketch of the proof of the forward implication for later use.  

\begin{proposition} 
\label{prop:dynamical}
A non-trivial countable group $G$ is left-orderable if and only if there exists a faithful 
representation $\rho: G \rightarrow \mathrm{Homeo_+}(\mathbb{R})$.
\end{proposition}
\begin{proof}
It suffices to prove the forward implication (\cite{Co}). Suppose that $G$ is left-orderable with ordering $<$ and fix an enumeration $\{ g_0, g_1, g_2, \ldots \}$ of $G$ with $g_0 = id$. Inductively define an order-preserving embedding $t:G\rightarrow \mathbb{R}$ as follows:  Set $t(g_0) = 0$.  If $t(g_0), \ldots, t(g_i)$ have already been defined and $g_{i+1}$ is either larger or smaller than all previously embedded elements, set:
\begin{displaymath} 
   t(g_{i+1}) = \left\{
     \begin{array}{ll}
       \mathrm{max}\{t(g_0), \ldots , t(g_i) \} +1 & \mbox{ if } g_{i+1} > \mathrm{max}\{g_0, \ldots , g_i \} \\
       \mathrm{min}\{t(g_0), \ldots , t(g_i) \} -1 & \mbox{ if } g_{i+1} < \mathrm{min}\{g_0, \ldots , g_i \} \\

     \end{array}
   \right.
\end{displaymath} 
On the other hand, if there exist $j, k \in \{ 0, \ldots, i \}$ such that $g_j < g_{i+1} < g_k$ and there is no $n \in \{ 0, \ldots ,i \}$ such that $g_j < g_{n} < g_k$, set $ t(g_{i+1}) = \frac{t(g_j)+t(g_k)}{2}$.
The group $G$ acts in an order-preserving way on $t(G)$ according to the rule $g(t(h)) = t(gh)$. This action extends uniquely to an order-preserving action of $G$ on the closure $\overline{t(G)}$. Since the complement of $\overline{t(G)}$ is a disjoint union of open intervals whose end-points lie in $\overline{t(G)}$, we can extend the $G$-action on $\overline{t(G)}$ affinely over $\mathbb{R} \setminus \overline{t(G)}$.  This defines a faithful representation $\rho_{\mathfrak{o}} : G \rightarrow \mathrm{Homeo}_+(\mathbb{R})$.   
\end{proof}

\begin{definition}
{\rm Given a left-ordering $\mathfrak{o}$ of a countable group $G$, a representation $\rho_{\mathfrak{o}}$ constructed as in Proposition \ref{prop:dynamical} is called a \textit{dynamical realisation} of $\mathfrak{o}$.
}
\end{definition}

It follows from the method of proof of \cite[Lemma 2.8]{Navas} that any two dynamical realisations of a left-orderable group $G$ are conjugate in $\hbox{Hom}(G, \hbox{Homeo}_+(\mathbb R))$. 

\begin{lemma} \label{lemma: non-trivial} 
Let $\mathfrak{o}$ be a left-order on a group $G$ and $\rho_{\mathfrak{o}}: G \rightarrow \mathrm{Homeo}_+(\mathbb{R})$ a dynamical realisation of $\mathfrak{o}$.

$(1)$ The action on $\mathbb R$ induced by $\rho_{\mathfrak{o}}$ is nontrivial, i.e. there are no global fixed points.

$(2)$ An element $g \in G$ is $\mathfrak{o}$-cofinal if and only if $\rho_{\mathfrak{o}}(g)$ is fixed point free $($and if and only if it is conjugate in $\mathrm{Homeo}_+(\mathbb{R})$ to $\hbox{sh}(\pm 1)$$)$. 

\end{lemma}

\begin{proof}
Assertion (1) follows from the observation that the set $t(G)$ is unbounded above and below. 

Next suppose that $g \in G$ is $\mathfrak{o}$-cofinal and $x \in \mathbb R$. There is an integer $n$ such that $x \in [t(g^n), t(g^{n+1}))$, so $\rho_{\mathfrak{o}}(g)(x) \in  [t(g^{n+1}), t(g^{n+2}))$ and therefore $\rho_{\mathfrak{o}}(g)(x) \ne x$. Conversely suppose that $\rho_{\mathfrak{o}}(g)$ is fixed point free. It follows from the construction of $\rho_{\mathfrak{o}}$ that the intersection of $\{t(g^n) : n \in \mathbb Z\}$ with any bounded subset of $\mathbb R$ is finite. Hence $g$ must be $\mathfrak{o}$-cofinal. Finally observe that Lemma \ref{trans}(1) implies that an element of $\hbox{Homeo}_+(\mathbb R)$ is conjugate to $\hbox{sh}(\pm1)$ if and only if it is fixed point free. This proves (2). 
\end{proof}

\subsection{Left-orders, dynamic realisations and slopes}

The next proposition is a converse to Proposition \ref{reps to orders}. 

\begin{proposition} \label{dynamical detects}
Let $M$ be a compact orientable Seifert fibred manifold  as in \S \ref{assumptions seifert}. Suppose that $J \subset \{1, 2, \ldots, r\}$ and $(J; [\alpha_*])$ is $\mathfrak{o}$-detected where $[\alpha_*]$ is horizontal. 

$(1)$ If $\rho_{\mathfrak{o}}: \pi_1(M) \rightarrow {Homeo}_+(\mathbb R)$ is a dynamic realisation of $\mathfrak{o}$, then $\rho_{\mathfrak{o}}$ is conjugate in $\hbox{Homeo}_+(\mathbb R)$ to a representation $\rho$ with values in $\widetilde{\hbox{Homeo}}_+(S^1)$. Further, each $[\alpha_j]$ is $\rho$-detected.

$(2)$  $(J; [\alpha_*])$ is $\rho'$-detected for some $\rho' \in \mathcal{R}_0(M)$ which takes values in some $\widetilde{PSL}(2,\mathbb R)_k$.
\end{proposition}

\begin{proof}
First note that if $[\alpha_*]$ is horizontal, Proposition \ref{non-orientable implies not detected} implies that $M$ has base orbifold $P(a_1, \ldots, a_n)$ and $h$ is $\mathfrak{o}$-cofinal. Therefore  Lemma \ref{lemma: non-trivial} (2) implies that $\rho(h)$ is conjugate to $\hbox{sh}(\pm 1)$. Hence $\rho_{\mathfrak{o}}$ is conjugate in $\hbox{Homeo}_+(\mathbb R)$ to a representation $\rho$ with values in $\widetilde{\hbox{Homeo}}_+(S^1)$.  

Fix $j$ and let $L_\mathfrak{o} \supset \langle \alpha_j \rangle$ be the line in $H_1(T_j; \mathbb R)$ determined by $\mathfrak{o}$. By construction, each element of $\pi_1(T_j) \setminus L_\mathfrak{o}$ is $\mathfrak{o}$-cofinal while
$$t(\{g \in \pi_1(T_j) \setminus L_\mathfrak{o}: g > 1\}) = \{t(g) : g \in \pi_1(T_j) \setminus L_\mathfrak{o}\} \cap \mathbb R_+$$ 
and 
$$t(\{g \in \pi_1(T_j) \setminus L_\mathfrak{o}: g < 1\}) = \{t(g) : g \in \pi_1(T_j) \setminus L_\mathfrak{o}\} \cap \mathbb R_-$$ 
It follows that $\tau(\rho(g)) > 0$ for all positive $g \in \pi_1(T_j) \setminus L_\mathfrak{o}$ and $\tau(\rho(g)) < 0$ for all negative $g \in \pi_1(T_j) \setminus L_\mathfrak{o}$. Thus $[\alpha_j]$ is $\rho$-detected, so assertion (1) holds. 

Up to replacing $\mathfrak{o}$ by $\mathfrak{o}_{op}$, we can suppose that $h$ is $\mathfrak{o}$-positive and therefore that $\rho(h) = \hbox{sh}(1)$. (See Remark \ref{irrational implies strong order detection}(2).) We shall assume this below. 

Fix an $\mathfrak{o}$-convex, normal subgroup $C$ of $\pi_1(M)$ such that $C \cap \pi_1(T_j) = \langle \alpha_j \rangle \cap \pi_1(T_j)$ for $j \in J$ and $C \cap \pi_1(T_j) \leq \langle \alpha_j \rangle \cap \pi_1(T_j)$ otherwise. As $C$ is $\mathfrak{o}$-convex and normal, $\mathfrak{o}$ induces a left-order $\bar {\mathfrak{o}}$ on $\pi_1(M)/C$ by taking the positive cone of $\bar {\mathfrak{o}}$ to be the image in $\pi_1(M)/C$ of the positive cone of $\mathfrak{o}$. The convexity of $C$ implies that $\bar{\mathfrak{o}}$ is a well-defined left-order. 

Let $\rho_{\bar{\mathfrak{o}}}$ be a dynamic realisation of $\bar {\mathfrak{o}}$. Since the image of $h$ in $\pi_1(M)/C$ is $\bar {\mathfrak{o}}$-cofinal, we can assume that $\rho_{\bar{\mathfrak{o}}}(hC) = \hbox{sh}(1)$.  Therefore $\bar \rho \in \mathcal{R}_0(M)$ where $\bar \rho$ is the composition of the dynamic realisation $\rho_{\bar{\mathfrak{o}}}$ and the quotient $\pi_1(M) \to \pi_1(M)/C$. We claim that $(J; [\alpha_*])$ is $\bar \rho$-detected. Fix $j$ and suppose first of all that $C \cap \pi_1(T_j) =  \{1\}$. In this case, $\pi_1(T_j) \to \pi_1(M)/C$ is injective, so the left-orders on $\pi_1(T_j)$ induced from $\mathfrak{o}$ and $\bar{\mathfrak{o}}$ coincide. The method of proof of part (1) of this proposition then shows that $[\alpha_j(\bar \rho)] = [\alpha_j(\bar{\mathfrak{o}})] = [\alpha_j(\mathfrak{o})] = [\alpha_j]$. On the other hand, if $C \cap \pi_1(T_j) =  \langle \alpha_j \rangle \cap \pi_1(T_j)$, then $\hbox{ker}(\bar \rho) \cap \pi_1(T_j) = \langle \alpha_j \rangle \cap \pi_1(T_j)$ so $[\alpha_j]$ is $\bar \rho$-detected for such $j$. Further, when $[\alpha_j]$ is rational, it is strongly is $\bar \rho$-detected. Thus if $J_0 = \{j \in J : [\alpha_j] \hbox{ is rational}\}$, $(J_0; [\alpha_*])$ is $\bar \rho$-detected. Proposition \ref{detected in psl2k} and Corollary \ref{irrational implies strong} now show that $(J; [\alpha_*])$ is $\rho'$-detected for some $\rho' \in \mathcal{R}_0(M)$ which takes values in some $\widetilde{PSL}_2(\mathbb R)_k$.
\end{proof}

\begin{remark} \label{orders versus reps}
{\rm The methods of this section can be used to show the equivalence of statements (2) and (3) of Theorem \ref{equiv 2} and of Theorem \ref{equiv 3}. For instance, consider a dynamical realisation $\rho_{\mathfrak{o}}$ of a left-order $\mathfrak{o}$ on the fundamental group of a graph manifold rational homology $3$-sphere $W$. If $\mathfrak{o}$ satisfies the condition of statement (2) of Theorem \ref{equiv 2} then Lemma \ref{lemma: non-trivial}(2) implies that the image by $\rho_{\mathfrak{o}}$ of the class each Seifert fibre of a piece of $W$ is conjugate to $\hbox{sh}(\pm 1)$. Thus statement (3) of Theorem \ref{equiv 2} holds. If $\mathfrak{o}$ satisfies the condition of statement (2) of Theorem \ref{equiv 3}, then $\pi(W)/C$ admits an induced left-order $\overline{\mathfrak{o}}$ in which the image of the class of each Seifert fibre of a piece of $W$ is cofinal. The composition of the quotient homomorphism $\pi_1(W) \to \pi_1(W)/C$ with a dynamical realisation of $\overline{\mathfrak{o}}$ is a representation which satisfies the condition of statement (3) of Theorem \ref{equiv 3}. 

Conversely consider a representation $\rho: \pi_1(W) \to \hbox{Homeo}_+(\mathbb R)$ with non-trivial image. It follows from \cite[Theorem 1.1(1)]{BRW}) that $\pi_1(W)$ admits a left-ordering and therefore so does $\hbox{ker}(\rho)$ as long as it is non-trivial. Since $\hbox{Homeo}_+(\mathbb R)$ is left-orderable, there is a left-ordering $\mathfrak{o}$ on $\pi_1(W)$ induced by the exact sequence $1 \to \hbox{ker}(\rho) \to \pi_1(W) \to \hbox{image}(\rho) \to 1$ (cf. the proof of Proposition \ref{reps to orders}). Since $\hbox{sh}(1)$ is cofinal in the natural left-orderings on $\hbox{Homeo}_+(\mathbb R)$ (cf. the first paragraph of the proof of \cite[Proposition 2.1]{Navas}), the reader will verify that if $\rho$ satisfies statement (3) of Theorem \ref{equiv 2} then $\mathfrak{o}$ satisfies statement (2) of that theorem. Further if $\rho$ satisfies statement (3) of Theorem \ref{equiv 3} then $\mathfrak{o}$ satisfies statement (2) of that theorem when we take $C = \hbox{ker}(\rho)$. 
}
\end{remark}

\section{Detecting slopes via foliations} \label{section: detecting via foliations}

Let $M$ be a compact orientable Seifert fibred manifold  as in \S \ref{assumptions seifert}. We use $\mathcal{F}(M)$ to denote the set of isotopy classes of co-oriented, taut foliations on $M$ which are transverse to $\partial M$. The set of isotopy classes of co-oriented, horizontal foliations on $M$ is contained in $\mathcal{F}(M)$. 

\subsection{Co-oriented taut foliations in Seifert manifolds} 
Brittenham has made a detailed study of essential laminations in Seifert fibred manifolds (\cite{Br1, Br2}). The next proposition is a consequence of his results. \\ \\

\begin{proposition} \label{britt}  $\;$ 

$(1)$ Suppose that $W$ is a Seifert fibred space which admits a co-oriented taut foliation $\mathcal{F}$. If $W$ is a rational homology $3$-sphere, then $\mathcal{F}$ can be isotoped to be horizontal. Consequently, the base orbifold of $W$ has underlying space a $2$-sphere. 

$(2)$ Suppose that $M$ is a compact orientable Seifert fibred manifold as in \S \ref{assumptions seifert}. 

$(a)$ If $\mathcal{F} \in \mathcal{F}(M)$ has no compact leaves, then $\mathcal{F}$ can be isotoped to be horizontal in $M$. Consequently, the base orbifold of $M$ is orientable. 

$(b)$ If $\mathcal{F} \in \mathcal{F}(M)$ has a compact leaf $L$, then $L$ is non-separating and has non-empty boundary. Consequently, it is either a vertical annulus or a horizontal surface. In the former case, $\mathcal{F}$ can be isotoped so that every leaf is either vertical or horizontal in $M$. In the latter case, $\mathcal{F}$ can be isotoped to be horizontal in $M$, so the base orbifold of $M$ is orientable. 

\end{proposition}

\begin{proof} 
A co-oriented taut foliation $\mathcal{F}$ on a rational homology $3$-sphere Seifert fibred space $W$ has no compact leaves, since each closed, connected, orientable surface in $W$ separates. Hence, as $W$ has first Betti number $0$, each leaf of $\mathcal{F}$ has exponential growth (\cite[Corollary 6.4]{Pl}). This implies that no leaf of $\mathcal{F}$ can be isotoped to be vertical. To see this, observe that there is a Riemannian metric on $W$ for which the lengths of the Seifert fibres are uniformly bounded, which implies that the growth of any leaf which can be isotoped to be vertical is polynomial. Thus $\mathcal{F}$  can be isotoped to be horizontal (\cite[Theorem 1 and Proposition 6]{Br1}). The co-orientability of $\mathcal{F}$ then shows that the Seifert fibres of $W$ can be coherently oriented and so its base orbifold is orientable. This proves (1).  

Next we consider Assertion (2)(a). Suppose that $\mathcal{F} \in \mathcal{F}(M)$ has no compact leaves but cannot be isotoped to be horizontal in $M$. Then by \cite[Theorem 1 and Proposition 4]{Br2}, $\mathcal{F}$ can be split open along a finite number of leaves to produce an essential lamination $\mathcal{L}$ which contains a vertical sublamination $\mathcal{L}_0$. We claim that $\mathcal{L}_0$ contains a compact vertical leaf, contrary to our initial assumption, thus completing the proof. In the case that the base orbifold of $M$ is not orientable, it suffices to show that the pull back of $\mathcal{L}_0$ to a $2$-fold cover of $M$ whose base orbifold has a planar underlying surface has a compact vertical leaf. Thus, without loss of generality we can assume that the base orbifold of $M$ is of the form $P(a_1, \ldots , a_n)$ where $P$ is a planar surface. 

After splitting $\mathcal{L}_0$ along its leaves which contain exceptional fibres of $M$, we may assume that it is a vertical lamination in the exterior $M_0$ of the exceptional fibres of $M$. As $P$ is orientable,  $M_0 \cong P_0 \times S^1$ where $P_0$ is the exterior of the cone points of $P(a_1, a_2, \ldots, a_n)$. Since $\mathcal{L}_0$ is vertical, it projects to a $1$-dimensional lamination $L_0$ in $P_0$. Further, since a compressing or end-compressing disk for a leaf of $L_0$ in $P_0$ is also one for $\mathcal{L}_0$ in $M$, $L_0$ is incompressible in $P_0$. 

Let $\tau \subset \hbox{int}(P_0)$ be an incompressible train track which fully carries $L_0$. Since $\mathcal{F}$ is co-oriented, so is $L_0$ and therefore the branches of $\tau$ can be coherently oriented. Thus, any embedded loop in $\tau$ consists of a sequence of coherently oriented branches. Each such loop separates $P_0$ and so we can choose one, say $C$, such that $P_0$ splits as the union of two subsurfaces $S_0, S_1$ such that $S_0 \cap S_1 = C$ and $\tau \cap \hbox{int}(S_0) = \emptyset$. Since $\tau$ is essential, there is a point $x \in \hbox{int}(S_0) \setminus L_0$. Given an arc in $S_0$ which is  transverse to $L_0$ and connects a point of $C$ to $x$, it contains a unique point $y$ ``closest" to $x$, and it is not hard to see that the leaf of $L$ which passes through $y$ must be a simple closed curve. The inverse image of this leaf in $M_0$ is a separating torus leaf of $\mathcal{L}_0$, which contradicts the fact that $\mathcal{F}$ has no compact leaves. Hence combinatorially, $\tau$ is a tree whose extrema are contained in $\partial P$. But then there is an arc $A$ contained in $\tau$ which is properly embedded in $(P_0, \partial P)$ and which has a tubular neighbourhood $N(A)$ such that $\tau \cap (N(A) \setminus A)$, if non-empty,  is contained in one of the two components of $N(A) \setminus A$. Arguing as in the proof that $\tau$ contains no loops shows that $\mathcal{L}_0 \subset \mathcal{F}$ has an annular leaf, contrary to our assumption that $\mathcal{F}$ has no compact leaves. Thus Assertion (2)(a) holds. 

Finally, for Assertion (2)(b), suppose that $\mathcal{F}$ has a compact leaf $F$. Since $\mathcal{F}$ is taut and co-oriented, $F$ must be non-separating. In particular, as $M$ is contained in a rational homology $3$-sphere, $\partial F$ cannot be empty. Further, as a leaf of a taut foliation, it is incompressible in $M$. Consequently, it is isotopic to either a vertical annulus or a horizontal surface. In the former case, \cite[Theorem 1]{Br2} implies that $\mathcal{F}$ can be isotoped so that every leaf is either vertical or horizontal. In the latter case, we can apply \cite[Proposition 5]{Br2} to see that if $\mathcal{F}$ cannot be isotoped to be horizontal, it contains the product of a Reeb annulus and an interval. But this possibility is ruled out by the tautness of $\mathcal{F}$. Thus $\mathcal{F}$ can be isotoped to be horizontal and the base orbifold of $M$ is orientable, which completes the proof. 
\end{proof}

\begin{corollary} \label{horizontal or vertical annulus}
Let $M$ be a compact orientable Seifert fibred manifold as in \S \ref{assumptions seifert} and fix $\mathcal{F} \in \mathcal{F}(M)$. Then either $\mathcal{F}$ contains a vertical annulus leaf or it is horizontal and $M$ has an orientable base orbifold. In particular, for each boundary component $T_j$ of $M$, either $\mathcal{F} \cap T_j$ contains a vertical leaf or it is horizontal.
\end{corollary}

\begin{proof} 
The first statement follows immediately from Proposition \ref{britt}(2) while the second follows from Proposition \ref{britt}(2) and \cite[Theorem 1]{Br2}. 
\end{proof}

\begin{corollary} \label{horizontal implies horizontal}
Let $M$ be a compact orientable Seifert fibred manifold as in \S \ref{assumptions seifert} and fix $\mathcal{F} \in \mathcal{F}(M)$. Then $\mathcal{F}$ is horizontal if and only if it restricts to a horizontal foliation on each boundary component of $M$. In the case that such a foliation exists, the base orbifold of $M$ is orientable. 
\end{corollary}

\begin{proof} 
The forward direction of the first statement is obvious while the reverse implication follows from Proposition \ref{britt}(2). The second statement follows as in the proof of Proposition \ref{britt}. 
\end{proof}

\subsection{Foliation detection of slopes} \label{conventions} 
An orientation-preserving homeomorphism $f: S^1 \to S^1$ gives rise to a {\it suspension foliation} $\mathcal{F}_f$ on the quotient torus $T_f = (S^1 \times \mathbb R)/\big( (x, t) \equiv (f(x), t+1) \big)$ whose leaves are the image of the lines $\{x\} \times \mathbb R$. More generally, a foliation $\mathcal{F}$ on a torus $T$ is a {\it suspension foliation} if there is a homeomorphism $(T, \mathcal{F}) \to (T_f, \mathcal{F}_f)$ for some orientation-preserving homeomorphism $f$ of the circle. It is simple to see that $\mathcal{F}$ is a suspension if and only if it is transverse to a foliation $\mathcal{F}'$ of $T$ by simple closed curves. The homeomorphism $f$ is the first return map on a fixed leaf of $\mathcal{F}'$ determined by $\mathcal{F}$.  

A {\it linear} foliation on a torus is any foliation homeomorphic to the suspension of a rotation $S^1 \to S^1$. A linear foliation has closed leaves if and only if the rotation is of finite order, in which case each of its leaves is closed. Linear foliations can be isotoped to either coincide or be transverse. 

For each $\mathcal{F} \in \mathfrak{F}(M)$ and $1 \leq j \leq r$, $\mathcal{F}$ determines a co-dimension $1$ foliation $\mathcal{F} \cap T_j$ of $T_j$ which, after an isotopy, either contains a closed leaf or is horizontal on $T_j$ (Corollary \ref{horizontal or vertical annulus}). Suppose that the latter occurs and let $C_j$ denote a horizontal simple closed curve on $T_j$ which carries $x_j$. Then $T_j \cong C_j \times S^1$ where the second factor is vertical. The first return map $\mathbb R \to \mathbb R$ of the pull back of $\mathcal{F} \cap T_j$ to $C_j \times \mathbb R$ under the cover $C_j \times \mathbb R \to C_j \times S^1$ determines an element $f_j \in \widetilde{\hbox{Homeo}}_+(S^1)$. (See \cite[page 654]{EHN} for the details.) We can define a uniquely determined slope $[\alpha_j(\mathcal{F})] \in \mathcal{S}(T_j)$ as follows:  
\begin{itemize}

\item if $\mathcal{F} \cap T_j$ contains a closed leaf, we define $[\alpha_j(\mathcal{F})]$ to be the slope of that leaf. 

\vspace{.2cm} \item if $\mathcal{F} \cap T_j$ is horizontal we define $[\alpha_j(\mathcal{F})] = [\tau(f_j) h - h_j^*]$. 

\end{itemize}
 
These definitions coincide when $\mathcal{F} \cap T_j$ is horizontal and contains a closed leaf. To see this, note that if $p, q$ are coprime integers, then $\mathcal{F} \cap T_j$ has a closed leaf representing $[ph - qh^*]$ if and only if there is a $t \in \mathbb R$ such that $f_j^q(t) = t + p$. On the other hand, it follows from the basic properties of translation numbers (\cite[\S 5]{Ghys}) that the latter condition is equivalent to $\tau(f_j) = \frac{p}{q}$, that is $[\tau(f_j) h - h_j^*] = [ph - qh^*]$.

\begin{lemma}
$[\alpha_j(\mathcal{F})]$ depends only on $\mathcal{F} \cap T_j$. 
\end{lemma}

\begin{proof}
We have just seen that the lemma holds if $[\alpha_j(\mathcal{F})]$ is rational, so suppose that it isn't. In other words, suppose that $\tau(f_j)$ is irrational. In this case $\mathcal{F}_j = \mathcal{F} \cap T_j$ has no closed leaves and therefore can be assumed to be transverse to each slice $\{x\} \times S^1$ of the factorisation $T_j = C_j \times S^1$, where the second factor is vertical (Corollary \ref{horizontal or vertical annulus}). It follows that $\mathcal{F}_j$ is a suspension and so each leaf of the lift $\widetilde{\mathcal{F}}_j$ of $\mathcal{F}_j$ to $\mathbb R^2 = \mathbb R \times \mathbb R$, the universal cover of $T_j = C_j \times \mathbb R$, is the graph of a continuous function $\mathbb R \to \mathbb R$. In particular, it is closed and separating in $\mathbb R^2$. This allows us to identify the leaf space of $\widetilde{\mathcal{F}}_j$ with $\mathbb R \equiv \{0\} \times \mathbb R \subset \mathbb R^2$. The induced action of $\pi_1(T_j)$ on $\mathbb R$ is free since $\mathcal{F}_j$ has no closed leaves. Consider the associated homomorphism $\rho_j: \pi_1(T_j) \to \widetilde{{\rm Homeo}}_+(S^1)$ where by construction, $\rho_j(h) = \hbox{sh}(1)$. 

Set $\varphi = \tau \circ \rho_j: \pi_1(T_j) \to \mathbb R$ and note that for all integers $p, q$ we have $\varphi(h^px_j^{-q}) = p\varphi(h) - q\varphi(x_j) = p - q\tau(f_j)$. Since $\tau(f_j)$ is irrational, $\varphi(h^px_j^{-q}) = 0$ if and only if $p = q = 0$, and therefore by the properties of $\tau$ (\cite[\S 5]{Ghys}), $\rho_j(h^px_j^{-q})(t) = t$ for some $t$ if and only if $p = q = 0$. Then for each $1 \ne \gamma \in \pi_1(T_j)$, either $\rho_j(\gamma)(t) > t$ for all $t$ or $\rho_j(\gamma)(t) < t$ for all $t$. In the former case $\varphi(\gamma) > 0$ and in the latter $\varphi(\gamma) < 0$. Put another way, for any leaf $L$ of $\widetilde{\mathcal{F}}_j$, $\varphi(\gamma) > 0$ if and only if $\gamma(L)$ lies above $L$ and $\varphi(\gamma) < 0$ if and only if $\gamma(L)$ lies below $L$. Since $\varphi(h^px_j^{-q}) > 0$ if and only if $q = 0$ and $p > 0$ or $\tau(f_j) < \frac{p}{q}$, the slope $[\tau(f_j)h - h_j^*] \subset \mathbb R^2$ is the dividing line between the elements of $\pi_1(T_j) = H_1(T_j) \leq H_1(T_j; \mathbb R) = \mathbb R^2$ which move a leaf $L$ to one of its sides from the ones that move it to its other. Thus $[\alpha_j(\mathcal{F})]$ depends only on $\mathcal{F} \cap T_j$. 
\end{proof}

We call $[\alpha_*(\mathcal{F})] = ([\alpha_1(\mathcal{F})], [\alpha_2(\mathcal{F})], \ldots , [\alpha_r(\mathcal{F})])$ 
the {\it slope} of $\mathcal{F}$. 

\begin{definition}
{\rm Let $\mathcal{F}$ be a taut co-oriented foliation in $M$ which is transverse to $\partial M$.  
A slope $[\alpha_j] \in \mathcal{S}(T_j)$ is {\it detected} by $\mathcal{F}$, or $\mathcal{F}$-{\it detected}, if $[\alpha_j] = [\alpha_j(\mathcal{F})]$. It is {\it strongly $\mathcal{F}$-detected} if it is $\mathcal{F}$-detected and $\mathcal{F}|T_j$ is linear. For $J \subset \{1, 2, \ldots, r\}$ and $[\alpha_*] = ([\alpha_1], [\alpha_2], \ldots , [\alpha_r]) \in \mathcal{S}(M)$, we say that $(J; [\alpha_*])$ is $\mathcal{F}$-{\it detected} if $\mathcal{F}$ detects $[\alpha_j]$ for all $j$ and $\mathcal{F}$ strongly detects $[\alpha_j]$ for $j \in J$. Finally, we say that $(J; [\alpha_*])$ is {\it foliation}-{\it detected} if it is $\mathcal{F}$-detected for some $\mathcal{F} \in \mathfrak{F}(M)$.
}
\end{definition} 

We shall often simplify the phrase ``$(\emptyset; [\alpha_*])$ is $\mathcal{F}$-detected, resp. foliation detected", to ``$[\alpha_*]$ is $\mathcal{F}$-detected, resp. foliation detected". Similarly, we simplify ``$(\{1, 2, \ldots, r\}; [\alpha_*])$ is $\mathcal{F}$-detected, resp. foliation detected", to ``$[\alpha_*]$ is {\it strongly} $\mathcal{F}$-detected, resp. {\it strongly} foliation detected".

Set 
$$\mathcal{D}_{fol}(M; J) = \{[\alpha_*] \in \mathcal{S}(M) : (J; [\alpha_*]) \mbox{ is foliation detected}\}$$
When $J = \emptyset$ we simplify $\mathcal{D}_{fol}(M; J)$ to $\mathcal{D}_{fol}(M)$. 

Here is a rephrasing of Corollary \ref{horizontal implies horizontal}. 

\begin{proposition} \label{prop: horizontal foliation} 
Suppose that $[\alpha_*] \in \mathcal{S}(M)$ is $\mathcal{F}$-detected. Then $[\alpha_*]$ is horizontal if and only if $\mathcal{F}$ is horizontal. 
\qed
\end{proposition}

\subsection{Representation detection and foliation detection} 

The relations between horizontal foliations in Seifert manifolds and representations with values in $\widetilde{\hbox{Homeo}}_+(S^1)$ was worked out in detail in \cite{EHN}. We summarize this work in the context of detected slopes below. 

\begin{proposition} {\rm (\cite{EHN}) }\label{reps to foliations}
Let $M$ be a compact orientable Seifert fibred manifold as in \S \ref{assumptions seifert}. Suppose that $J \subseteq \{1, 2, \ldots, r\}$ and $[\alpha_*] \in \mathcal{S}(M)$ is horizontal. Then $(J; [\alpha_*])$ is representation detected if and only if it is foliation detected. 
\end{proposition}

\begin{proof}
Suppose that $(J; [\alpha_*])$ is $\rho$-detected for some $\rho \in \mathcal{R}_0(M)$. By Proposition \ref{detected in psl2k} we can suppose that $\rho$ takes values in $\widetilde{PSL}(2,\mathbb R)_k$ for some $k \geq 1$. A standard construction associates a  horizontal foliation $\mathcal{F}(\rho)$ on $M$ to $\rho$, which  we describe next. 

Let $X \to P(a_1, a_2, \ldots , a_n)$ be the universal cover so that $\hbox{int}(X) \cong \mathbb R^2$. Then $\pi_1(M)$ acts properly discontinuously on $X$ via the quotient homomorphism $\varphi: \pi_1(M) \to \pi_1(P(a_1, a_2, \ldots , a_n))$ and freely and  properly discontinuously on $X \times \mathbb R$ via 
$$\gamma \cdot (x, t) = (\varphi(\gamma)(x), \rho(\gamma)(t))$$ 
since $\rho(h) = \hbox{sh}(1)$. 
Consider the quotient $M' = X \times_\rho \mathbb R = (X \times \mathbb R)/ \pi_1(M)$. Then $M'$ is a $3$-manifold whose fundamental group is isomorphic to the group of deck transformations of the cover $X \times \mathbb R \to M'$, which is $\pi_1(M)$ acting on $X \times \mathbb R$ as above. This determines an identification $\pi_1(M') = \pi_1(M)$. Observe that the image of each $\{x\} \times \mathbb R$ in $M'$ is a circle, so $M'$ is Seifert fibred in a natural way. Now $(x, t)$ and $(x', t')$ map to the same fibre of $M'$ if and only if $x' = \varphi(\gamma)(x)$ for some $\gamma \in \pi_1(M') = \pi_1(M)$. Thus the base orbifold $\mathcal{B}'$ of $M'$ is the quotient of $X$ by $\pi_1(P(a_1, a_2, \ldots , a_n))$. In other words, $\mathcal{B}' = P(a_1, a_2, \ldots , a_n)$. Further, the action of $\pi_1(M')$ on $X$ induced by the quotient $\varphi': \pi_1(M') \to \pi_1(\mathcal{B}')$ coincides with the action of $\pi_1(M)$ on $X$ under the $\varphi$-action. It follows that $M' \cong M$. (This is because the Seifert invariant of the $i^{th}$ Seifert fibre is determined by the action of $\varphi(x_i)$ on $X$: $a_i$ is the order of $\varphi(x_i)$ while $b_i$ is the unique integer in the interval $(0, a_i)$ such that $\varphi(x_i^{b_i})$ acts by rotation by $2 \pi/a_i$ about the fixed point of $\varphi(x_i)$ in $X$.) Finally observe that $M'$, and therefore $M$, inherits a horizontal foliation $\mathcal{F}(\rho)$ from the foliation $\{X \times \{t\} : t \in \mathbb R\}$ of $X \times \mathbb R$. 

Let $\widetilde L_j$ be the component of $\partial X$ such that $\widetilde T_j = \widetilde L_j \times \mathbb R$ is the $\pi_1(T_j)$-invariant component of $\partial (X \times \mathbb R)$. Fix an identification $\widetilde L_j = \mathbb R$ such that $\varphi(h_j^*)(x) = x+1$ for all $x \in \widetilde L_j$. Assume first that $\rho(\pi_1(T_j))$ conjugates into the group of translations of $\mathbb R$. Then after conjugation, for $(x, t) \in \widetilde T_j$ we have $h \cdot (x, t) = (x, t + 1)$ and $h_j^* \cdot (x, t) = (x + 1, t + \tau(\rho(h_j^*)))$. It is well known (cf. \cite[Example 1.2.4 and Exercise 1.2.6]{CC1}) that $\mathcal{F}(\rho)|_{T_j}$ is a linear foliation of slope $[\tau(\rho(h_j^*)) h - h_j^*] = [\alpha_j(\rho)]$. Thus $\mathcal{F}(\rho)$ strongly detects $[\alpha_j]$ for $j \in J$. Suppose now that $\rho(\pi_1(T_j))$ does not conjugate into the group of translations of $\mathbb R$. It follows from the conventions set in \S \ref{conventions} that $[\alpha_j(\mathcal{F})] = [\tau(\rho(h_j^*)) h - h_j^*] = [\alpha_j(\rho)] = [\alpha_j]$. Thus $(J; [\alpha_*])$ is $\mathcal{F}(\rho)$-detected.  

Conversely suppose that $(J; [\alpha_*])$ is $\mathcal{F}$-detected. Since $[\alpha_*]$ is horizontal, so is $\mathcal{F}$ (Proposition \ref{prop: horizontal foliation}). Let $\widetilde M \to M$ be a universal cover and $\widetilde{\mathcal{F}}$ the lift of $\mathcal{F}$ to $\widetilde M$. Then the preimage in $\widetilde M$ of any Seifert fibre of $M$ intersects each leaf of $\widetilde{\mathcal{F}}$ once and only once. (See \cite[\S 3]{Br2} for instance.) It follows that the leaf space of $\widetilde{\mathcal{F}}$ is homeomorphic to the real line, so $\pi_1(M)$ acts on $\mathbb R$. As $\mathcal{F}$ is co-oriented, this action is by orientation-preserving homeomorphisms. Further, $h$ acts by $\hbox{sh}(1)$, at least up to conjugation and an appropriate orientation on $\widetilde{\mathcal{F}}$. Let $\rho: \pi_1(M) \to \widetilde{\hbox{Homeo}}_+(S^1)$ be the associated homomorphism. It follows from our conventions that $[\alpha_*(\rho)] = [\alpha_*(\mathcal{F})] = [\alpha_*]$. If $j \in J$, then $\mathcal{F} \cap T_j$ is linear, which implies that $\rho(x_j)$ is conjugate to a translation. Hence $\rho|\pi_1(T_j)$ conjugates into the subgroup of translations of $\mathbb R$ so that $\rho$ strongly detects $[\alpha_j]$. This completes the proof.
\end{proof}

\subsection{Foliation detection and non-horizontal $[\alpha_*]$} 

Our first lemma shows that a vertical slope can rarely be strongly foliation detected. 

\begin{lemma} \label{strongly rational not vertical}
Suppose that $\mathcal{F} \in \mathfrak{F}(M)$ and $\mathcal{F} \cap T_j$ is a foliation by simple closed curves. Then either $[\alpha_j(\mathcal{F})]$ is horizontal or $M$ is a twisted $I$-bundle over the Klein bottle with base orbifold a M\"{o}bius band. In the latter case we can alter the Seifert structure on $M$ so that $[\alpha_j(\mathcal{F})]$ becomes horizontal.
\end{lemma}

\begin{proof}
Suppose that $\mathcal{F} \cap T_j$ is a foliation by simple closed curves of slope $[h_j]$. Let $M'$ be the manifold obtained by Dehn filling $M$ along its fibre slope on $T_j$. If $M$ has base orbifold of the form $P(a_1, \ldots, a_n)$, then $M'$ is homeomorphic to $(\#_{i = 1}^{n} L_{a_i}) \# (\#_{j= 1}^{r-1} S^1 \times D^2)$. On the other hand, $M'$ admits a co-oriented taut foliation and so is either prime or $S^2 \times I$ (see e.g. \cite[Corollary 9.1.9]{CC2}). The latter is clearly impossible. Therefore $M$ is prime and so $n + (r-1) \leq 1$. But then $M$ is either a solid torus or $S^1 \times S^1 \times I$, which contradicts our assumptions. Similarly if $M$ has base orbifold $Q(a_1, \ldots, a_n)$, $M'$ is homeomorphic to $(\#_{i = 1}^{n} L_{a_i}) \# (S^1 \times S^2) \# (\#_{j= 1}^{r-1} S^1 \times D^2)$ so as it is prime, $n =  r -1 = 0$. Hence $M$ is a twisted $I$-bundle over the Klein bottle with the Seifert structure having base orbifold a M\"{o}bius band. After changing the structure on $M$ to that with base orbifold $D^2(2,2)$, $[\alpha_j(\mathcal{F})]$ becomes horizontal. Thus the lemma holds. 
\end{proof} 

Here is one of the main result of this section. 

\begin{proposition} \label{foliation vertical}
Let $M$ be a compact orientable Seifert fibred manifold $M$ as in \S \ref{assumptions seifert} and $J \subseteq \{1, \ldots , r\}$. Fix $[\alpha_*] \in \mathcal{S}(M)$ and suppose that $j \in J$ implies that $[\alpha_j] \ne [h]$. 

$(1)$ If $M$ has base orbifold $Q(a_1, \ldots , a_n)$, then $(J; [\alpha_*])$ is foliation detected if and only if $v([\alpha_*]) \geq 1$. 

$(2)$ If $M$ has base orbifold $P(a_1, \ldots , a_n)$ and $v([\alpha_*]) > 0$, then $(J; [\alpha_*])$ is foliation detected if and only if $v([\alpha_*]) \geq 2$.

\end{proposition}

\begin{proof} 
We use $C_j$ to denote the image of $T_j$ in the base orbifold of $M$. Without loss of generality we suppose that $[\alpha_j] = [h]$ if and only if $1\leq j \leq v([\alpha_*])$. Then $J \subset \{v([\alpha_*]) + 1, \ldots , r\}$. 

First assume that $M$ has base orbifold $Q(a_1, \ldots , a_n)$. Proposition \ref{prop: horizontal foliation} shows that $v([\alpha_*]) > 0$. Let $M_0$ be a connected manifold obtained by cutting $M$ open along disjoint vertical annuli $A_0, A_1, \ldots, A_{v([\alpha_*]) - 1}$ where $A_0$ is non-separating and connects $T_1$ to itself and $A_i$ connects $T_i$ and $T_{i+1}$ for $1 \leq i \leq v([\alpha_*]) - 1$. Then $M_0$ is Seifert with base orbifold $P_0(a_1, \ldots , a_n)$ where $P_0$ is planar with $|\partial P_0| = r - v([\alpha_*]) + 1$. We can write $\partial M_0$ = $T_0 \cup T_{v([\alpha_*])+1} \cup \ldots \cup T_r$ where $T_0$ is a torus containing $2v([\alpha_*])$ disjoint vertical annuli  $A_0^+, A_0^-, A_1^+, A_1^-, \ldots , A_{v([\alpha_*]) - 1}^+, A_{v([\alpha_*]) - 1}^-$ indexed so that $A_j^+$ and $A_j^-$ are identified by a homeomorphism $f_j$ in reconstructing $M$ from $M_0$. We can find a horizontal slope $[\alpha_0]$ on $T_0$ and a representation $\rho_0$ which detects $(\{0\} \cup J; ([\alpha_0], [\alpha_{v([\alpha_*])+1}], \ldots , [\alpha_r]))$ (Proposition \ref{tau fibre 2}). By Proposition \ref{reps to foliations}, 
there is a horizontal foliation $\mathcal{F}_0$ in $M_0$ which detects $(\{0\} \cup J; ([\alpha_0], [\alpha_{v([\alpha_*])+1}], \ldots , [\alpha_r]))$. Then $\mathcal{F}_0 \cap A_j^\pm$ is a foliation by horizontal arcs and we can assume that $f_j$ preserves this foliation. Then $\mathcal{F}_0$ determines a horizontal foliation $\mathcal{F}_1$ in $M$ which is transverse to the annuli $A_j$ and detects $[\alpha_j]$ on $T_j$ for $v([\alpha_*]) + 1 \leq j \leq r$. Note that $f_0$ reverses the transverse orientation of the leaves of $\mathcal{F}_1$ while $f_1, \ldots , f_{v([\alpha_*])}$ preserves them, so $\mathcal{F}_1$ is not co-oriented. But spinning $\mathcal{F}_1$ vertically around $A_1, A_2, \ldots, A_{v([\alpha_*]) - 1}$  in an appropriate fashion produces a co-oriented taut foliation $\mathcal{F}$ which detects $(J; [\alpha_*])$ (cf. \cite[Example 4.9]{Ca}). This completes the proof of (1). 

Next assume that $M$ has base orbifold $P(a_1, \ldots , a_n)$. Since $v([\alpha_*]) > 0$, Proposition \ref{britt}(2) implies that if $(J; [\alpha_*])$ is detected by $\mathcal{F} \in \mathcal{F}(M)$, then $\mathcal{F}$ has a leaf which is a non-separating vertical annulus, which necessarily intersects distinct boundary components of $M$. Hence $v([\alpha_*]) \geq 2$.

Conversely suppose that $v([\alpha_*]) \geq 2$.  Fix $J \subset \{1, 2, \ldots , r\}$ and $[\alpha_*] \in \mathcal{S}(M)$ such that $j \in J$ implies that $[\alpha_j] \ne [h]$. Without loss of generality we can suppose that $[\alpha_j] = [h]$ if and only if $1\leq j \leq v([\alpha_*])$. Then $J \subset \{v([\alpha_*]) + 1, \ldots , r\}$. Let $M_0$ be a connected manifold obtained by cutting $M$ open along disjoint vertical annuli $A_1, A_2, \ldots, A_{v([\alpha_*]) - 1}$ where $A_i$ connects $T_i$ and $T_{i+1}$. Then $M_0$ is Seifert with base orbifold $P_0(a_1, \ldots , a_n)$ where $P_0$ is planar with $|\partial P_0| = r - v([\alpha_*]) + 1$. We can write $\partial M_0$ = $T_0 \cup T_{v([\alpha_*])+1} \cup \ldots \cup T_r$ where $T_0$ is a torus containing $2v([\alpha_*]) - 2$ disjoint vertical annuli  $A_1^+, A_1^-, \ldots , A_{v([\alpha_*]) - 1}^+, A_{v([\alpha_*]) - 1}^-$ indexed so that $A_j^\pm$ are identified by a homeomorphism $f_j$ in reconstructing $M$ from $M_0$. We can find a horizontal foliation $\mathcal{F}_0$ in $M_0$ which detects some $(\{0\} \cup J; [\beta_*])$ where $[\beta_j] = [\alpha_j]$ for $v([\alpha_*]) + 1 \leq j \leq r$. Then $\mathcal{F}_0 \cap A_j^\pm$ is a foliation by horizontal arcs and we can assume that $f_j$ preserves this foliation. Then $\mathcal{F}_0$ determines a horizontal foliation $\mathcal{F}_1$ in $M$ which is transverse to the annuli $A_j$ and detects $[\alpha_j]$ on $T_j$ for $v([\alpha_*]) + 1 \leq j \leq r$. Spinning $\mathcal{F}_1$ vertically around these annuli produces a co-oriented taut foliation $\mathcal{F}$ which detects $(J; [\alpha_*])$, which completes the proof of (2).
\end{proof}

\begin{proposition} \label{irrational to rational - foliations} 
Suppose that $J \subset \{1, 2, \ldots, r\}$ and $(J; [\alpha_*])$ is foliation detected where some $[\alpha_j]$ is irrational. Reindex the boundary components of $M$ so that $[\alpha_j]$ is irrational if and only $1 \leq j \leq s$. Set $J^\dagger = J \cup \{1, 2, \ldots , s\}$. Then for $1 \leq j \leq s$ there is an open sector $U_j \subset \mathcal{S}(T_j)$ containing $[\alpha_j]$ such that one of the following two statements holds.

$(1)$ $(J^\dagger; [\alpha_*'])$ is foliation detected for all $[\alpha_*']$ such that $[\alpha_j'] \in U_j$ for $1 \leq j \leq s$ and $[\alpha_j'] = [\alpha_j]$ otherwise. 

$(2)$ $M$ has no singular fibres, $s = 2$, $J^\dagger = \{1, 2, \ldots , r\}$ and $[\alpha_*]$ is horizontal with $[\alpha_j] = [\tau_j h - h_j^*]$ where $\tau_3, \ldots, \tau_r \in \mathbb Z$. Further, there is a homeomorphism $\varphi: U_1 \to U_2$ which preserves both rational and irrational slopes and for which $(J^\dagger; [\alpha_*'])$ is foliation detected for all $[\alpha_*'] = ([\alpha_1'], \varphi([\alpha_1']), [\alpha_3], \ldots, [\alpha_r])$ whenever $[\alpha_1'] \in U_1$. 
\end{proposition} 

\begin{proof}
If $[\alpha_*]$ is horizontal this is a consequence of Proposition \ref{reps to foliations} and Corollary \ref{irrational to rational - reps}. On the other hand, if $v([\alpha_*]) > 0$ conclusion (1) holds by Proposition \ref{foliation vertical}. 
\end{proof}

\subsection{Controlling boundary behaviour} \label{controlling boundary behaviour} 
Our goal in this section is to show that we can choose our foliations to display standardized behaviour on the boundary components of $M$. This will be a key component of the proof of Theorem \ref{theorem: gluing}. 

Here is a consequence of Proposition \ref{reps to foliations} and its proof combined with Proposition \ref{no parabolics}.

\begin{proposition} \label{standard realisation} 
Let $M$ be a compact orientable Seifert fibred manifold as in \S \ref{assumptions seifert} and fix a horizontal $[\alpha_*] \in \mathcal{S}(M)$ and $J \subseteq \{1, 2, \ldots, r\}$. If $(J; [\alpha_*])$ is foliation detected, then it is $\mathcal{F}(\rho)$-detected where $\rho$ takes values in $\widetilde{PSL}(2,\mathbb R)_k$ for some $k \geq 1$. Further, we can suppose that for each $j$, $\rho(\pi_1(T_j))$ contains no parabolics. 
\qed
\end{proposition}

Thus each horizontal $[\alpha_*] \in \mathcal{D}_{fol}(M; J)$ is detected by some $\mathcal{F}$ where for each $i$,  $\mathcal{F} \cap T_i$ is a suspension foliation of either a rotation, and so is a circle fibration of slope $[\alpha_i(\mathcal{F})]$ (this always happens if $j \in J$), or a hyperbolic element $g_i$ of some $\widetilde{PSL}(2,\mathbb R)_k$. In the latter case, $\mathcal{F} \cap T_j$ has a non-zero even number of closed leaves, each of slope $[\alpha_i(\mathcal{F})]$, and $g_i$ is alternately increasing or decreasing in the complementary intervals of the fixed points of $g_i$. 

For an orientation-preserving homeomorphism $f: S^1 \to S^1$, let $T(f) \cong S^1 \times S^1$ denote its mapping torus $(S^1 \times I)/((x,1) \equiv (f(x), 0))$ and let $\mathcal{F}(f)$ denote its suspension foliation on $T(f)$. 

\begin{definition} \label{j interval-hyperbolic}
{\rm (1) For each positive integer $k$ let $IH^0(k)$ be the set of orientation-preserving homeomorphisms $f: S^1 \to S^1$ whose fixed point set consists of $2k$ disjoint closed non-degenerate intervals and for which $f$ is alternately increasing or decreasing on the $2k$ complementary open intervals.  

(2) For a positive integer $k$ we say that a codimension one foliation $\mathcal{F}$ on a torus $T$ is {\it $k$ interval-hyperbolic} if it is homeomorphic to the suspension foliation $\mathcal{F}(f)$ on $T(f)$ of a homeomorphism $f \in IH^0(k)$. The {\it slope} of a $k$ interval-hyperbolic foliation on $T$ is the slope of its closed leaves. }
\end{definition}

\begin{lemma} \label{standard are isotopic} 
$(1)$ Let $f \in IH^0(k)$. Then the foliation $\mathcal{F}(f)$ on $T(f)$ is invariant up to isotopy under any homeomorphism of $T(f)$ which leaves the slope of $\mathcal{F}(f)$ invariant. 

$(2)$ Two $k$ interval-hyperbolic foliations of the same slope on a torus $T$ are isotopic.
\end{lemma}

\begin{proof}
(1) Let $\alpha$ be a primitive class in $H_1(T(f))$ representing the slope of $\mathcal{F}(f)$ and $\alpha^* \in H_1(T(f))$ a primitive class carried by the image of $S^{1} \times \{0\}$ in $T(f)$. Then $\alpha^*$ is dual to $\alpha$. Use the ordered  basis $\{\alpha, \alpha^*\}$ of $H_1(T(f))$ to identify the mapping class group of $T(f)$ with $GL(2, \mathbb Z)$ in the usual way. Then a homeomorphism of $T(f)$ which leaves $[\alpha]$ invariant corresponds to a matrix of the form $\left(\begin{matrix} \epsilon  &   b \\ 0   &  \delta \end{matrix}\right)$ where $\epsilon, \delta \in \{\pm 1\}$ and $b \in \mathbb Z$. As this matrix factors $\left(\begin{matrix} 1 &   \delta b \\ 0   &  1 \end{matrix}\right) \left(\begin{matrix} 1  &   0 \\ 0   &  \delta \end{matrix}\right) \left(\begin{matrix} \epsilon  &   0 \\ 0   &  1 \end{matrix}\right)$, it suffices to show that $\mathcal{F}(f)$ on $T(f)$ is invariant up to isotopy under the homeomorphisms corresponding to the matrices $\left(\begin{matrix} 1 &   1 \\ 0   &  1 \end{matrix}\right), \left(\begin{matrix} -1  &   0 \\ 0   &  1 \end{matrix}\right)$ and $\left(\begin{matrix} 1  &   0 \\ 0   &  -1 \end{matrix}\right)$. 

To that end, first observe that $\mathcal{F}(f)$ is invariant up to isotopy by a Dehn twist of slope $[\alpha]$ (i.e. simply perform the Dehn twist in one of the annuli composed of circle leaves of $\mathcal{F}$) and that the matrix of such a Dehn twist is $\left(\begin{matrix} 1  &   1 \\ 0   &  1\end{matrix}\right)$.

Next observe that $IH^0(k)$ is a conjugacy class of elements of $\hbox{Homeo}_+(S^1)$. (See the discussion on page 355 of \cite{Ghys} for instance.) In particular there is an $h \in \hbox{Homeo}_+(S^1)$ such that $f^{-1} = hfh^{-1}$. Then if $[x,t]$ denotes the class of $(x, t)$ in $T(f)$, the correspondence $[x,t] \mapsto [h(x),1-t]$ determines a homeomorphism of $T(f)$ which leaves $\mathcal{F}(f)$ invariant and which corresponds to the matrix $\left(\begin{matrix} -1  &   0 \\ 0   &  1\end{matrix}\right)$.

Finally note that if $r$ is an orientation-reversing homeomorphism of $S^1$ then $rfr^{-1} \in IH^0(k)$. Thus there is some $h \in \hbox{Homeo}_+(S^1)$ such that 
$rfr^{-1} = hfh^{-1}$. Hence if $r_1 = h^{-1} r$, then $r_1$ is orientation-reversing and $f = r_1 f r_1^{-1}$. The correspondence $[x,t] \mapsto [r_1(x),t]$ determines a homeomorphism of $T(f)$ which leaves $\mathcal{F}(f)$ invariant and which corresponds to the matrix $\left(\begin{matrix} 1  &   0 \\ 0   &  -1\end{matrix}\right)$. This completes the proof of (1). 

(2) Let $\mathcal{F}_1$ and $\mathcal{F}_2$ be $k$ interval-hyperbolic foliations of slope $[\alpha]$ on a torus $T$ and fix homeomorphisms $f_1, f_2 \in IH^0(k)$ and $\varphi_1: (T(f_1), \mathcal{F}(f_1)) \to (T, \mathcal{F}_1)$ and $\varphi_2: (T(f_2), \mathcal{F}(f_2)) \to (T, \mathcal{F}_2)$. Fix $h \in \hbox{Homeo}_+(S^1)$ such that $f_2 = hf_1h^{-1}$ and note that the correspondence $[x,t] \mapsto [h(x),t]$ determines a homeomorphism $\psi: (T(f_1), \mathcal{F}(f_1)) \to (T(f_2), \mathcal{F}(f_2))$. By construction, the composition $\theta = \varphi_2 \circ \psi \circ \varphi_1^{-1}: (T, \mathcal{F}_1) \stackrel{\cong}{\longrightarrow} (T, \mathcal{F}_2)$. Since $\theta([\alpha]) = [\alpha]$, part (1) of the lemma implies that $\mathcal{F}_1$ is isotopic to $\theta(\mathcal{F}_1) = \mathcal{F}_2$. 
\end{proof}

\begin{lemma} \label{ready for gluing horizontal non-compact} 
Let $M$ be a compact orientable Seifert fibred manifold as in \S \ref{assumptions seifert}. Consider a co-oriented taut foliation $\mathcal{F}(\rho)$ where $\rho \in \mathcal{R}_0(M)$ takes values in some $\widetilde{PSL}(2,\mathbb R)_k$ and $\rho(\pi_1(T_i))$ contains no parabolics for each $i$. Suppose that $\mathcal{F}(\rho)$ has no compact leaves and that $[\alpha_*(\mathcal{F})]$ is rational. Then there is a constant $k(\mathcal{F}) > 0$ such that for positive integers $k_1, \ldots, k_r \geq k(\mathcal{F})$, there is a co-oriented horizontal foliation $\mathcal{F}'$ in $M$ such that $[\alpha_*(\mathcal{F})]$ is $\mathcal{F}'$-detected and $\mathcal{F}'$ is $k_i$ interval-hyperbolic on the boundary component $T_i$ of $\partial M$. 
\end{lemma}

\begin{proof}
Our hypotheses imply that $\mathcal{F} \cap T_i$ is a suspension foliation of an elliptic or hyperbolic element $f$ of $\widetilde{PSL}(2,\mathbb R)_k$. In the first case, $\mathcal{F} \cap T_i$ is a circle fibration of slope $[\alpha_i(\mathcal{F})]$. In the latter, $f$ has an even number $2l_i > 0$ of fixed points and is alternately increasing or decreasing on the $2l_i$ complementary open intervals. Further, $\mathcal{F} \cap T_i$ has exactly $2l_i$ compact leaves $C_{i1}, C_{i2}, \ldots, C_{i2l_i}$, each of slope $[\alpha_i(\mathcal{F})]$. If $\mathcal{F}\cap T_i$ is a circle fibration, set $l_i = 0$ and choose a circle fibre $C_{i1}$. Define $k(\mathcal{F}) = \max\{l_1, l_2, \ldots, l_r\}$. 

Let $L_1, L_2, \ldots, L_s$ be the leaves of $\mathcal{F}$ which contain some $C_{ij}$. Since $\mathcal{F}(\rho)$ is taut, the fundamental group of each $L_i$ injects into $\pi_1(M)$ (\cite[Theorem 4.35]{Ca}). Since each $L_i$ is non-compact, its fundamental group is free and each of its circle boundary component provides a free generator. Since $\mathcal{F}(\rho)$ is co-oriented, each $L_j$ is orientable. Now replace $\mathcal{F}$ by a foliation $\mathcal{F}'$ obtained by thickening $L_1 \cup L_2 \cup \ldots \cup L_s$ (cf. \cite[Operation 2.1.1]{Ga2}). Each $L_j$ has a product neighbourhood $V_j \cong L_j \times I$, which we can assume are mutually disjoint, and each slice $L_j \times \{t\}$ is a leaf of $\mathcal{F}'$. Set $L_j' := L_j \times \{\frac12 \}$.

Choose a subset $C_1, \ldots, C_r$ of the $C_{ij}$ so that $C_i$ is contained in $T_i$. For each $j$, fix a homomorphism $\varphi_j: \pi_1(L_j) \to \hbox{Homeo}_+(I)$ which sends each generator determined by the $C_i \subseteq \partial L_j$ to a homeomorphism $f_i \in  \hbox{Homeo}_+(I)$ and contains all other free generators in its kernel. We can refoliate $L_j \times I$ as the quotient of its universal cover $(\tilde L \times I)$ by the diagonal action of $\pi_1(L_j)$ which acts by deck transformations on the first factor and by $\varphi_j$ on the second. Call the new foliation $\mathcal{F}(\varphi_j)$. By construction, $L_j \times \{0\}$ and $L_j \times \{1\}$ are leaves of $\mathcal{F}(\varphi_j)$ and if $C_i \subseteq \partial L_j$, the holonomy of $\mathcal{F}(\varphi_j)$ on $C_i \times I$ is the suspension of $f_i$. Further, $\mathcal{F}'$ is unchanged near the components of $\partial L_j' \times I$ which are not one of the $C_i \times I$. Now replace $\mathcal{F}'$ by the new foliation obtained by substituting $\mathcal{F}(\varphi_j)$ for the product foliation $L_j \times I$. By choosing appropriate $f_i \in  \hbox{Homeo}_+(I)$, we can arrange for the new foliation to be $k_i$ interval hyperbolic on each $T_i$. 
\end{proof}

\begin{lemma} \label{ready for gluing horizontal compact} 
Let $M$ be a compact orientable Seifert fibred manifold as in \S \ref{assumptions seifert}. Consider a co-oriented taut foliation $\mathcal{F}(\rho)$ where $\rho \in \mathcal{R}_0(M)$ takes values in some $\widetilde{PSL}(2,\mathbb R)_k$ and $\rho(\pi_1(T_i))$ contains no parabolics for each $i$. Suppose that $\mathcal{F}(\rho)$ has a compact leaf $F$, so that $[\alpha_*(\mathcal{F})]$ is rational. Given any positive integers $k_1, \ldots, k_r$, there is a co-oriented horizontal foliation $\mathcal{F}'$ in $M$ which detects $[\alpha_*(\mathcal{F})]$ such that $\mathcal{F}'$ is $k_i$ interval-hyperbolic on the boundary component $T_i$ of $\partial M$ subject to the following constraints.

$(a)$ If $r = |\partial M| \geq 2$ and $F$ is planar, then $k_i = k_j$ for some $i \ne j$. 

$(b)$ If $M \cong N_2$, then $k_1$ is odd. 

\end{lemma}

\begin{proof}
Since $F$ is compact and horizontal, it is the fibre of a horizontal locally-trivial fibre bundle $M \to S^1$. We can assume, without loss of generality, that this fibre bundle is $\mathcal{F}(\rho)$. 

The base orbifold of $M$ is orientable since it admits a horizontal co-oriented foliation, say it is $P(a_1, \ldots, a_n)$. If $F$ is planar, it has at least two boundary components, and if two, either $P(a_1, \ldots, a_n) = D^2(2,2)$ and $M \cong N_2$ or $P(a_1, \ldots, a_n)$ is an annulus and $M \cong S^1 \times S^1 \times I$. By assumption, the latter does not occur. Then to prove the lemma, we need to consider the four cases: $F$ has positive genus;  $F$ is planar and $r \geq 2$; $F$ is planar, $r = 1$, and $|\partial F| \geq 3$; $M \cong N_2$. 

Let $g$ be the genus of $F$ and consider the presentation 
\begin{eqnarray}\pi_1(F) = \langle a_1, b_1, \ldots , a_g, b_g, x_1, \ldots , x_m : \big(\Pi_{j} [a_j, b_j]  \big) x_1 \ldots x_m  = 1 \rangle  \nonumber 
\end{eqnarray}
where $m = |\partial F|$ and $x_i$ corresponds to the $i^{th}$ boundary component of $F$. 

First assume that the genus of $F$ is positive and fix a neighbourhood $F \times I$ of $F$ where each $F \times \{t\}$ is a fibre of $\mathcal{F}(\rho)$. Each orientation preserving homeomorphism of $I$ is a commutator (cf. the proof of \cite[Proposition 5.11]{Ghys}), so for each $f_1, \ldots, f_m \in \hbox{Homeo}_+(I)$ we can find  $g, h\in \hbox{Homeo}_+(I)$ such that the composition $f_1 \circ \cdots \circ f_m \circ [g, h]$ is the identity. Define a homomorphism $\varphi: \pi_1(F) \to \hbox{Homeo}_+(I)$ by setting $\varphi(a_1) = g, \varphi(b_1) = h, \varphi(x_i) = f_i$ and $\varphi(a_j) = \varphi(b_j) = 1$ for $2 \leq j \leq r$. For each boundary component $T_j$ of $M$, fix a boundary component $ C_j$ of $F$ contained in $T_j$. By setting $f_l$ to be the identity when $x_l$ does not correspond to some $C_i$ and choosing $f_l$ appropriately otherwise, we can refoliate $F \times I$ as in the proof of Lemma \ref{ready for gluing horizontal non-compact} to introduce $k_i$ interval hyperbolic behaviour on any boundary component of $M$.

Assume next that $F$ is planar, $r \geq 2$ and, after reindexing, that $k_1 = k_2$. Each $f_1, \ldots, f_m \in \hbox{Homeo}_+(I)$ whose composition is the identity determines a homomorphism $\varphi: \pi_1(F) \to \hbox{Homeo}_+(I)$. Fix boundary components $C_1, C_2$ of $F$ where $C_1 \subset T_1$ and $C_2 \subset T_2$. We can assume that the class $x_i \in \pi_1(F)$ corresponds to $C_i$ for $i = 1, 2$. By choosing $f_2 = f_1^{-1}$ and $f_3 = \ldots = f_m = 1$, we can use the operation of the previous case on a product neighbourhood $F \times I$ of $F = F \times \{\frac12\}$ to construct a new co-oriented taut foliation $\mathcal{F}'$  on $M$ which detects $[\alpha_*(F)]$ and which is $k_1$ interval hyperbolic on both $T_1$ and $T_2$, and remains elliptic on $T_3, \ldots, T_r$. By construction, $\mathcal{F}'$ has non-compact leaves in $F \times I$, and since $\mathcal{F}'$ is horizontal, the quotient map from $M$ to its base orbifold is surjective when restricted to any of its leaves. In particular, there are non-compact leaves incident to each of the tori $T_3, \ldots, T_r$. Thickening such leaves preserves the elliptic nature of $\mathcal{F}' \cap T_i$ for $3 \leq i \leq r$ and the $k_i$ interval hyperbolic behaviour on $\mathcal{F}' \cap T_i$ for $1 \leq i \leq 2$. Now apply the construction from the proof of Lemma \ref{ready for gluing horizontal non-compact} to introduce $k_i$ interval hyperbolic behaviour on $\mathcal{F}' \cap T_i$ for $3 \leq i \leq r$ while leaving $\mathcal{F}' \cap T_1$ and $\mathcal{F}' \cap T_2$ alone. 

If $F$ is planar, $r = 1$, and $m \geq 3$, fix a product neighbourhood of $F$ and two boundary components $C_1, C_2$ of $F$ which are successive on $\partial M$. Modify $\mathcal{F}$ by introducing holonomy, as above, which only alters the boundary behaviour of $\mathcal{F}$ near $C_1 \cup C_2$, so that the new foliation $\mathcal{F}'$ is $1$ interval hyperbolic on $\partial M$. Since $\mathcal{F}'$ is horizontal and has non-compact leaves which intersect $\partial M$ arbitrarily close to any component $C$ of $\partial F \setminus (C_1 \cup C_2)$, we can add $(k_1-1)$ interval hyperbolic behaviour near $C$ in such a way that the new foliation is $k_1$ interval hyperbolic on $\partial M$. 

Finally suppose that $r = 1$ and $m = 2$, so $M \cong N_2$. Then $F$ is a horizontal annulus and as $M$ has a connected boundary, it must be $N_2$. We can alter $\mathcal{F}$ in a saturated product neighbourhood of $F$ to to produce a foliation which is $1$ interval hyperbolic on $\partial M$. The reader will verify that if we alter the new foliation by performing a similar operation near another surface fibre, the resulting foliation is not $2$ interval hyperbolic, but that it can then be performed near a third surface fibre to produce an element of $\mathcal{F}(M)$ which is $3$ interval hyperbolic on $\partial M$.  Proceeding this way, we can produce an element of $\mathcal{F}(M)$ which is $k$ interval hyperbolic on $\partial M$ for all odd $k$. This completes the proof. 
\end{proof}

\begin{lemma} \label{ready for gluing vertical} 
Let $M$ be a compact orientable Seifert fibred manifold as in \S \ref{assumptions seifert}. Suppose that $[\alpha_*] \in \mathcal{S}(M)$ is rational, but not horizontal, and foliation detected. Index the boundary components of $M$ so that $[\alpha_i] = [h]$ if and only if $1 \leq i \leq v$. Then for any positive integer $k_0$, there is a co-oriented taut foliation $\mathcal{F}$ on $M$ which detects $[\alpha_*]$ and is $k_i$ interval hyperbolic on $T_i$ $($$1 \leq i \leq r$$)$ where $k_{v+1}, \ldots, k_r$ are arbitrary, $k_i \geq k_0$ for $2 \leq i \leq v$, and $k_1$ is arbitrary unless $v = 1$ and $M$ has a non-orientable base orbifold. In the latter case, $k_1$ be chosen to be an arbitrary odd positive integer. In all cases, $\mathcal{F}$ can be assumed to be transverse to any predetermined, finite family of Seifert fibres of $M$. 
\end{lemma}

\begin{proof}
Consider a foliation $\mathcal{F}_0$ on $M$ which detects $(\{v+1, \ldots, r\}; [\alpha_*])$, as constructed in the proof of Proposition \ref{foliation vertical}. The only compact leaves in the resulting foliation are a finite number of vertical annuli, and we can assume that they avoid any predetermined, finite family of Seifert fibres. 

Recall that in the case that the base orbifold of $M$ is orientable, $v \geq 2$ (Proposition \ref{foliation vertical}). By construction, $\mathcal{F}_0$ has exactly $v-1$ compact leaves $A_1, \ldots, A_{v-1}$ where $A_i$ is a vertical annulus which runs between $T_{i-1}$ and $T_{i}$. After thickening these compact leaves (avoiding the finite set of Seifert fibres), we can apply the constructions used in the proofs of the previous two lemmas to add $a_1$ interval hyperbolic behaviour on $T_1$, $(a_1 + a_2)$ interval hyperbolic behaviour on $T_2$, \ldots, $(a_{v-2} + a_{v-1})$ interval hyperbolic behaviour on $T_{v-1}$, and $a_{v-1}$ interval hyperbolic behaviour on $T_v$ where $a_1, \ldots, a_{v-1}$ are arbitrary positive integers. Take $a_1 = k_1$ and $a_2, \ldots, a_{v-1} > k_0$. Further, for each $v+1 \leq i \leq r$, the new foliation has non-compact leaves incident to $T_i$ on which it is linear of slope $[\alpha_i]$. Hence we can apply the construction of the proof of Lemma \ref{ready for gluing horizontal non-compact} to introduce $k_i$ interval hyperbolic behaviour on $T_i$ for arbitrary $k_i$ ($v+1 \leq i \leq r$). The resulting foliation is transverse to any predetermined, finite family of Seifert fibres of $M$. 

The case that the base orbifold of $M$ is non-orientable is handled similarly. Here $v \geq 1$ and $\mathcal{F}_0$ has exactly $v$ compact leaves $A_0, \ldots, A_{v-1}$, each a vertical annulus, where $A_0$ runs between $T_1$ and itself and $A_i$ runs between $T_{i}$ and $T_{i+1}$ for $1 \leq i \leq v$. We can use $A_0$ to introduce $a_0$ interval hyperbolic behaviour on $T_1$ for an arbitrary odd $k_1$, then $A_1, \ldots, A_{v-1}$ to introduce $a_0 + a_1$ interval hyperbolic behaviour on $T_1$, $(a_1 + a_2)$ interval hyperbolic behaviour on $T_2$, \ldots, $(a_{v-2} + a_{v-1})$ interval hyperbolic behaviour on $T_{v-1}$, and $a_{v-1}$ interval hyperbolic behaviour on $T_v$ where $a_1$ is an arbitrary odd integer and $a_2, \ldots, a_{v-1}$ are arbitrary positive integers. By an appropriate choice of the $a_i$ we produce a co-oriented taut foliation on $M$ which is $k_i$ interval hyperbolic of slope $[h]$ on $T_i$ for $1 \leq i \leq v$ as in the stament of the lemma, and is linear of slope $[\alpha_i]$ on $T_i$ for $v+1 \leq i \leq r$. As before, we can 
apply the construction of the proof of Lemma \ref{ready for gluing horizontal non-compact} to introduce $k_i$ interval hyperbolic behaviour on $T_i$ for arbitrary $k_i$ ($v+1 \leq i \leq r$). Again, the resulting foliation is transverse to any predetermined, finite family of Seifert fibres of $M$. 
\end{proof}

Here is a consequence of Lemmas \ref{ready for gluing horizontal non-compact} and \ref{ready for gluing horizontal compact} which will be applied to study slope detection via Heegaard-Floer homology. It can be considered a special case of Theorem \ref{theorem: gluing}. Recall the definition of the Seifert manifolds $N_t$ from \S \ref{N_t}.

\begin{proposition} \label{foliation iff detected}
Let $M$ be a compact orientable Seifert fibred manifold as in \S \ref{assumptions seifert} and fix $J \subseteq \{1, 2, \ldots , r\}$ and $t \geq 2$. Suppose that $[\alpha_*]$ is a rational  element of $\mathcal{S}(M)$ such that $[\alpha_j] \ne [h]$ for $j \in J$. Let $W_t$ be obtained by the $[\alpha_j]$ Dehn filling of $M$ for $j \in J$ and by gluing $N_{t}$ to $M$ along $T_j$ in such a way that for each $j \notin J$, $h_0$ is identified with $[\alpha_j]$. Then if $W_t$ is a rational homology $3$-sphere, it admits a co-oriented taut foliation if and only if $[\alpha_*]$ is horizontal and lies in $\mathcal{D}_{fol}(M; J)$. 
\end{proposition}

\begin{proof} 
Without loss of generality we can assume that $\alpha_j$ is a primitive element of $H_1(T_j)$ for each $j$. 

Since $M$ is contained in a rational homology $3$-sphere, its rational homology is isomorphic to $\mathbb Q^r$ generated by peripheral classes, one from each $T_j$. It follows  that as $W_t$ is a rational homology $3$-sphere, $\alpha_1, \ldots, \alpha_r$ are linearly independent when considered as classes in $H_1(M; \mathbb Q)$. Hence $v([\alpha_*]) \leq 1$ and if it equals $1$, $M$ has an orientable base orbifold. 

Let $M'$ be the Seifert manifold obtained by the $[\alpha_j]$ Dehn filling of $M$ for $j \in J$ and $[\alpha_*'] \in \mathcal{S}(M')$ the projection of $[\alpha_*]$. Since $[\alpha_j] \ne [h]$ for $j \in J$, the Seifert structure on $M$ extends to one on $M'$ of the sort described in in \S \ref{assumptions seifert}. Further, its base orbifold is obtained from that of $M$ by attaching a disk with a cone point of order $\Delta(\alpha_j, h)$ to the boundary component corresponding to $T_j$ for each $j \in J$.  

First suppose that $v([\alpha_*]) = 1$. Then $[\alpha_*]$ is neither horizontal nor lies in $\mathcal{D}_{fol}(M; J)$ (Proposition \ref{foliation vertical}(2)). On the other hand, if $W_t$ admitted a co-oriented taut foliation, it can be isotoped so that it is transverse to each $T_j$ for $j \not \in J$ and intersects $M'$ and each $N_t$ in co-oriented taut foliations \cite{BR}. Since $\mathcal{D}_{fol}(N_t) = \{[h_0]\}$ (Proposition \ref{tau fibre 2}), this implies that $[\alpha_*']$ is foliation detected in $M'$. But this is impossible since $M'$ has an orientable base orbifold and $v([\alpha_*']) = 1$ (Proposition \ref{foliation vertical}). Thus $W_t$ does not admit a co-oriented taut foliation. 

Next suppose that $v([\alpha_*]) = 0$. If $W_t$ admitted a co-oriented taut foliation, the argument of the previous paragraph shows that $[\alpha_*']$ is foliation detected in $M'$. Since $v([\alpha_*']) = v([\alpha_*]) = 0$, Corollary \ref{horizontal implies horizontal} implies that it is detected by a horizontal foliation $\mathcal{F}'$ on $M'$. Since the cores of the $\alpha_j$ filling tori are transverse to $\mathcal{F}'$ ($j \in J$), we obtain a co-oriented taut foliation on $M$ from $\mathcal{F}'$ which strongly detects $[\alpha_j]$ for $j \in J$. Thus $[\alpha_*] \in \mathcal{D}_{fol}(M; J)$. 

Conversely, if $[\alpha_*] \in \mathcal{D}_{fol}(M; J)$, it is easy to see that $M'$ admits a co-oriented taut foliation which detects $[\alpha_*']$ if $\partial M' \ne \emptyset$.  If $\partial M' = \emptyset$, then $M' = W_t$ so we are done. Assume otherwise and observe that since $[\alpha_*']$ is horizontal, Corollary \ref{horizontal implies horizontal} and Proposition \ref{standard realisation} imply that the base orbifold of $M'$ is orientable and that $[\alpha_*]$ is $\mathcal{F}(\rho)$-detected where $\rho \in \mathcal{R}_0(M)$ takes values in some $\widetilde{PSL}(2,\mathbb R)_k$ and $\rho(\pi_1(T_j))$ contains no parabolics for each $j$. If $\mathcal{F}(\rho)$ contains a compact leaf, $M'$ admits a fibration $\mathcal{F}_1$ which strongly detects $[\alpha_*]$. On the other hand, each attached $N_{t}$ admits a fibration of slope $[h_0]$ (cf. \S \ref{N_t}) which can be glued to $\mathcal{F}_1$ to produce a co-oriented taut foliation in $W_t$. Suppose then that each leaf of $\mathcal{F}(\rho)$ is non-compact. Let $k_0$ denote the smallest odd integer which is greater than $k(\mathcal{F})$. By Lemma \ref{ready for gluing horizontal non-compact} there is a co-oriented taut foliation $\mathcal{F}'$ on $M$ which detects $[\alpha_*']$ and which is $k_0$ interval-hyperbolic on each component of $\partial M'$. By Lemma \ref{ready for gluing horizontal compact}, there is a co-oriented taut foliation on $N_t$ which detects $[h_0]$ and which is $k_0$ interval-hyperbolic on $\partial N_t$. These foliations piece together to give the desired foliation on $W_t$. 
\end{proof}

\section{Detecting rational slopes via L-spaces} \label{section: detecting via L-spaces}

In this section we show to how to detect rational elements of $\mathcal{S}(M)$ using Heegaard-Floer homology. 

\subsection{Some background results on L-spaces}
Here is an elementary fact that we will use below. Its proof follows from the homology exact sequence of the pair $(W, M_1)$. (Compare with \cite[Lemma 3.2]{Watson2008}.)

\begin{lemma} \label{order} 
Let $M_1$ and $M_2$ be two rational homology solid tori and $W = M_1 \cup_f M_2$ where $f: \partial M_1 \to \partial M_2$ is a homeomorphism. Then 
$$|H_1(W)| = d_1d_2|T _1(M_1)||T_1(M_2)| \Delta(\lambda_1, \lambda_2)$$ 
where $\lambda_j$ is the rational longitude of $M_j$, $d_j \geq 1$ is its order in $H_1(M_j)$, and $T_1(M_j)$ is the torsion subgroup of $H_1(M_j)$. 
\qed 
\end{lemma}

Recall the manifolds $N_t$ from \S \ref{N_t} and the basis $\{h_0, h_1\}$ of $H_1(\partial N_t)$ where $h_0$ is the rational longitude of $N_t$. (The classes $h_0, h_1$ are only well-defined up to sign.) In what follows we take $R$ to be a compact, connected orientable $3$-manifold with torus boundary and $f: \partial R \to \partial N_t$ to be a gluing map.  
Set 
$$W_t(f) = R \cup_f N_t$$
We call $W_t(f)$ an $N_t$-{\it filling} of $R$. 
A striking property of $N_2$-fillings  was proved in \cite{BGW}.

\begin{proposition} {\rm (\cite{BGW})} \label{HF solid torus} 
Let $R$ be a compact, connected, orientable $3$-manifold with torus boundary and suppose that $f_1$ and $f_2$ are homeomorphisms $\partial R \to \partial N_2$ such that $f_2$ is obtained by post-composing $f_1$ by a Dehn twist in $\partial  N_2$ along $h_0$. Then $\widehat{HF}(W_2(f_1)) \cong \widehat{HF}(W_2(f_2))$. 
\end{proposition} 

\begin{proof}
This result follows from \cite[Proposition 7]{BGW}. Compare with the proof of Theorem 7 of that paper and the comments which follow it.
\end{proof}

Watson has generalised this result to a wider class of manifolds he calls {\it Heegaard-Floer solid tori}. In particular, for each integer $t \geq 2$ he has shown that the manifold $N_t$ defined in \S \ref{N_t} is a Heegaard-Floer solid torus. 

\begin{proposition} {\rm (Watson \cite{Watson2013})} \label{general HF solid torus} 
Let $R$ be a compact, connected, orientable $3$-manifold with torus boundary and suppose that $f_1$ and $f_2$ are homeomorphisms $\partial N_t \to \partial R$ such that $f_2$ is obtained by precomposing $f_1$ by a Dehn twist in $\partial  N_t$ along $h_0$. Then $\widehat{HF}(W_t(f_1)) \cong \widehat{HF}(W_t(f_2))$. 
 \end{proposition}

A closed, connected $3$-manifold $V$ is an {\it L-space} if it is a rational homology sphere with the property that $\hbox{rank}(\widehat{HF}(V)) = \left|H_1(V)\right|$. Examples of L-spaces include lens spaces and, more generally, connected sums of manifolds with elliptic geometry \cite[Proposition 2.3]{OSz2005-lens}. L-spaces do not admit smooth co-orientable taut foliations (\cite[Theorem 1.4]{OSz2004-genus}). 
Here is an immediate consequence of Lemma \ref{order} and Proposition \ref{HF solid torus}. 

\begin{corollary} \label{preserve l-spaces}
If $f_2$ is obtained by post-composing $f_1$ by a Dehn twist in $\partial  N_t$ along $h_0$, then $N_t \cup_{f_1} R$ is an L-space if and only if $N_t \cup_{f_2} R$ is an L-space. 
\qed
\end{corollary}

If $R$ is a compact, connected $3$-manifold with torus boundary and $\{ \alpha, \beta \}$ is a basis of $H_1(\partial R)$, then $(\alpha, \beta, \alpha + \beta)$ is called a \textit{triad} if 
\[ |H_1(R(\alpha))| + |H_1 (R(\beta))| = |H_1(R(\alpha+\beta))|
\]
A key property of L-space Dehn filling has been proven by Ozsv\'ath and Szab\'o.  

\begin{proposition} {\rm (\cite[Proposition 2.1]{OSz2005-lens}, \cite[Proposition 4]{BGW})} \label{triad}
Suppose that $R$ is a compact, connected, orientable $3$-manifold with torus boundary. If $(\alpha, \beta, \alpha + \beta)$ is a triad such that $R(\alpha)$ and $R(\beta)$ are L-spaces, then $R(u\alpha+v\beta)$ is an L-space for all coprime integer pairs $u, v \geq 0$. 
\qed 
\end{proposition}

Here is an immediate consequence of Lemma \ref{homeo N_t}.

\begin{lemma} \label{matrix normalisation}
Let $\{\alpha_1, \alpha_2\}$ be a basis of $H_1(\partial R)$ and let $f: \partial R \to \partial N_t$ be a gluing map. Then there exists $f'$ such that $W_t(f) \cong W_t(f')$ and $f'_* =   \left(\begin{smallmatrix} a  &   b \\ c   &  d \end{smallmatrix}\right)$, with respect to the bases $\{\alpha_1, \alpha_2\}$ of $H_1(\partial R)$ and $\{h_0, h_1\}$ of $H_1(\partial N_t)$, can be chosen so that $c \geq 0$ and $\det(f_*') = 1$. 
\qed 
\end{lemma}

The next lemma will be used to study L-space $N_t$-filling. 

\begin{lemma} \label{fibre surgery} 
Suppose that $R$ is a compact, connected, orientable $3$-manifold with torus boundary. Fix a basis $\{ \alpha_1, \alpha_2 \}$ of $H_1(\partial R)$ so that the rational longitude of $R$ is of the form $[\lambda_R ]= [p \alpha_1 - q \alpha_2]$ for some $p,q \geq 0$ and set 
$\mathcal{T}(R) = \overline{\{ \frac{u}{v} : u,v \hbox{ are coprime integers and }R( u \alpha_1 - v \alpha_2 ) \mbox{ is not an L-space}\}} \subseteq \mathbb{R} \cup \{ \frac{1}{0} \}$.   
Let $f : \partial R \rightarrow \partial N_t$ be a homeomorphism with $f_* =   \left(\begin{smallmatrix} a  &   b \\ c   &  d \end{smallmatrix}\right)$ with respect to the bases $\{\alpha_1, \alpha_2\}$ of $H_1(\partial R)$ and $\{h_0, h_1\}$ of $H_1(\partial N_t)$ where $c \geq 0$ and $\det(f_*) = 1$. If $\frac{b + td}{a + tc}, \frac{(1-t)b + td}{(1-t)a + tc} \not \in \mathcal{T}(R)$ and $-t < \frac{pa-qb}{pc - qd} < \frac{t}{t-1}$, then $W_t(f)$ is an L-space.
\end{lemma}

\begin{proof} Denote by $C_t$  the exterior in $N_t$ of the exceptional fibre $K_1$ of type $\frac{1}{t}$. Then $\partial C_t = \partial N_t \sqcup T_1$ where $T_1$ is the boundary of a tubular neighbourhood of $K_1$ in $N_t$. Denote an oriented fibre slope on $T_1$ by $\phi_1$ and by $\beta_1$ the oriented slope of $K_1$ of distance $1$ from $\phi_1$ such that $t \beta_1 + \phi_1$ is the meridional slope of $K_1$. Then $t \beta_1 + \phi_1 = 0$
in $H_1(N_t)$. Without loss of generality we can suppose that $h_1$ is oriented so that $h_1 = \phi_1$ in $H_1(C_t)$. 

The exterior of the two exceptional fibres of $N_t$ can be identified with $P \times S^1$, where $P$ is a twice-punctured disk, in such a way that writing $\partial P = \partial_0 \sqcup \partial_1 \sqcup \partial_2$ then 
 
\begin{itemize}

\item $\partial N_t = \partial_0 \times S^1$ and $T_1 = \partial_1 \times S^1$;

\vspace{.2cm} \item $\partial_0, \partial_1$ and $\partial_2$ can be oriented so that $\beta_1 = [\partial_1]$ and $[\partial_1] + [\partial_2] = [\partial_0]$ in $H_1(P)$. 

\vspace{.2cm} \item $t[\partial_2] + (t-1) \phi_2 = 0$ in $H_1(C_t)$ where $\phi_2$ represents the fibre slope on the boundary of a tubular neighbourhood of the exceptional fibre $K_2$ of type $\frac{t-1}{t}$.

\end{itemize} 
 
Then in $H_1(C_t)$ we have 
$$t \beta_1 = t [\partial_0] - t[\partial_2] = t[\partial_0] + (t-1) h_1$$
The rational longitude $h_0$ of $N_t$ corresponds homologously in $C_t$ to the meridional class $t \beta_1 +  \phi_1$ of $K_1$, at least up to a non-zero rational multiple. Hence from above, $h_0$ corresponds to $(t[\partial_0] + (t-1)h_1) +  h_1 = t([\partial_0] + h_1)$. In other words we can take 
$[\partial_0] = h_0 - h_1$ and therefore $t \beta_1 = t[\partial_0] + (t-1) h_1 = t h_0 - h_1$. 
It follows that for any $x, y \in \mathbb Z$ we have 
$$xt \beta_1 + y \phi_1 = xt h_0 + (y - x) h_1$$ 
In particular in $H_1(C_t)$ we have $\phi_1 = h_1$ and 
$$t \beta_1 = t h_0 - h_1$$ 
Further if $\beta_2 = \beta_1 + \phi_1$ then 
$$t \beta_2 = t h_0 + (t - 1)h_1$$
Let $E = R \cup_f C_t$. Then $\partial E = T_1$ and $E(t\beta_1 + \phi_1) = W(f)$. Since $\Delta(\beta_j, \phi_1) = 1$, $C_t(\beta_j)$ is a solid torus for $j = 1, 2$. Thus $E(\beta_j)$ is a Dehn filling of $R$. Indeed, 
$$E(\beta_1) = R(f_*^{-1}(t h_0 - h_1)) = R((b + td)\alpha_1 - (a + tc)\alpha_2)$$ 
and so as $ \frac{b + td}{a + tc} \not \in \mathcal{T}(R)$, $E(\beta_1)$ is an L-space. In particular, $|H_1(E(\beta_1))| \ne 0$. 
Similarly 
$$E(\beta_2) = R(((1-t) b + td)\alpha_1 - ((1-t)a + tc)\alpha_2)$$  
and so as $\frac{(1-t) b + td}{(1-t)a + tc} \not \in \mathcal{T}(R)$, $E(\beta_2)$ is an L-space. Therefore $|H_1(E(\beta_2))| \ne 0$. 

For any slope $\gamma'$ on $\partial R$ we have $|H_1(R(\gamma'); \mathbb Z)| = k \Delta(\gamma', \lambda_R) = k \Delta(\gamma', p \alpha_1 - q \alpha_2)$, where $k$ is a constant determined by $R$ as in Lemma \ref{order}. Hence 
$$|H_1(E(\beta_1))| = |H_1(R((b + td) \alpha_1 - (a + tc) \alpha_2))| = k|t(pc - qd) - (qb-pa)| = k|tx + y|$$
where $x = pc - qd$ and $y = pa - qb$, while
$$|H_1(E(\beta_2))| = |H_1(R(((1-t) b + td) \alpha_1 - ((1-t)a + tc) \alpha_2))| = k |tx + (1-t)y|$$
Since $|H_1(E(\beta_1))|, |H_1(E(\beta_2))| \ne 0$ it follows that $|H_1(E(\beta_1 + \beta_2))|  = |H_1(E(\beta_1))| + |H_1(E(\beta_2))|$ if and only if 
$\hbox{sign}(\beta_1 \cdot \lambda_R) = \hbox{sign}(\beta_2 \cdot \lambda_R)$. Equivalently, $\hbox{sign}(tx + y) = \hbox{sign}(tx + (1-t)y)$. 
Since $\frac{pa-qb}{pc - qd} = \frac{a}{c} + \frac{q}{cx}$, the reader will verify that this occurs if and only if 
$-t < \frac{y}{x} = \frac{pa-qb}{pc - qd} < \frac{t}{t-1}$, 
which we have assumed. Proposition \ref{triad} now implies that $W_t(f) = E(t \beta_1 + \phi_1)$ is an L-space. 
\end{proof} 

\subsection{L-space $N_t$-fillings of $M$ when its base orbifold is non-orientable} 

\begin{proposition} \label{mobius piece}
Let $M$ be a Seifert manifold with base orbifold $Q(a_1, a_2, \ldots , a_n)$ as in \S \ref{assumptions seifert} and let $[\alpha_*] = ([\alpha_1], [\alpha_2], \ldots , [\alpha_r])$ be a rational element of $\mathcal{S}(M)$. Let $t \geq 2$ and $W_t$ be a graph manifold obtained by gluing $N_{t}$ to $M$ along each $T_j$ in such a way that $[h_0]$ is identified with $[\alpha_j]$. Then $W_t$ is an L-space if and only if $[\alpha_*]$ is horizontal. 
\end{proposition}

\begin{proof}  Fix $t \geq 2$.
If some $[\alpha_j]$ is vertical, then $W_t$ is not a rational homology $3$-sphere since the fibre class on $\partial M$ is rationally null-homologous in $M$. In particular, $W_t$ is not an L-space.  Suppose then that $[\alpha_*]$ is horizontal. We show that $W_t$ is an L-space by induction on $r$. Without loss of generality we suppose that each $\alpha_j$ is a primitive element of $H_1(T_j)$. 

{\bf Base case}. Suppose that $r = 1$, so $Q$ is a M\"{o}bius band, and let $T = \partial M$. As usual, $h \in H_1(T)$ represents the slope of the Seifert fibre on $\partial M$ in the given structure and $h^* \in H_1(T)$ is the dual class defined in \S \ref{assumptions seifert}. We apply Lemma \ref{fibre surgery} as follows.

As $Q$ is non-orientable, $h$ represents the rational longitude of $M$. In particular, relative to the basis $\{ h, h^* \}$ the coordinates $p,q$ of the rational longitude are $p=1, q=0$. If $\gamma$ represents a non-longitudinal slope on $\partial M$, then $\Delta(\gamma, h) \geq 1$ and $M(\gamma)$ is a Seifert fibred rational homology $3$-sphere with base orbifold $Q(a_1, \ldots, a_n, \Delta(\gamma, h))$, so it is an L-space \cite[Proposition 5]{BGW}.  Thus $\mathcal{T}(M) = \overline{\{ \frac{r}{s} : M( r h - s h^* ) \mbox{ is not an L-space}\}} = \{\frac{1}{0}\}$.

Let $f : \partial M \rightarrow \partial N_t$ be a gluing map such that $f^{-1}(h_0)$ is horizontal. In other words, $f^{-1}(h_0) \neq \pm h$. By Lemma \ref{matrix normalisation} we can suppose that its matrix $\left(\begin{smallmatrix} a  &   b \\ c   &  d \end{smallmatrix}\right)$ with respect to the bases $\{h, h^*\}$ and $\{h_0, h_1\}$ satisfies $c \geq 0$ and $\det(f_*) = 1$. Since $f^{-1}(h_0) \neq \pm h$, $c > 0$. Corollary \ref{preserve l-spaces} implies that the question of whether $W_t = W_t(f)$ is an L-space depends only on $(c,d)$. Thus after adding an appropriate multiple of the second row of $f$ to the first row we may assume that $-c < a \leq 0$. Then $-t < -1 <  \frac{a}{c} = \frac{pa - qb}{pc - qd} \leq 0 < \frac{t}{t-1}$.  Thus $a + tc \neq 0$ and $(1-t)a + tc \ne 0$, so $\frac{b + td}{a + tc}, \frac{(1-t)b + td}{(1-t)a + tc} \not \in \mathcal{T}(M)$.  By Lemma  \ref{fibre surgery},  $W_t(f)$ is an L-space.

{\bf Inductive case}. Suppose that the result holds when $1 \leq |\partial M| < r$. Let $R$ be a manifold obtained by gluing $r - 1$ copies of $N_t$ to $M$ along $\partial M$ in such a way that for each $1 \leq j \leq r-1$, $[h_0]$ is identified with $[\alpha_j]$. Then $\partial R = T_r$. Let $h \in H_1(\partial R)$ represent the slope of the Seifert fibre of $M$ and let $h^*$ be a dual class to $h$. As $Q$ is non-orientable, $h$ represents the rational longitude of $R$. Our inductive hypothesis implies that any Dehn filling of $R$ along a slope other than $[h]$ is an L-space, so again we will apply Lemma  \ref{fibre surgery} where $\mathcal{T}(R) = \overline{\{ \frac{r}{s} : R( r h - s h^* ) \mbox{ is not an L-space}\}} = \{\frac{1}{0}\}$, $p=1, q=0$. That is, if $f: \partial R \to \partial N_t$ is a gluing map such that  $f^{-1}(h_0) \neq \pm h$, write $f_* =   \left(\begin{smallmatrix} a  &   b \\ c   &  d \end{smallmatrix}\right)$ with respect to the bases $\{h, h^*\}$ and $\{ h_0, h_1\}$   and proceed as in the base case to complete the induction. 
\end{proof}

\subsection{L-space $N_t$-fillings of $M$ when its base orbifold is orientable}

In this section we suppose that $M$ is a Seifert manifold with base orbifold $P(a_1, a_2, \ldots , a_n)$ as in \S \ref{assumptions seifert}. Fix a rational element $[\alpha_*] = ([\alpha_1], [\alpha_2], \ldots , [\alpha_r])$ of $\mathcal{S}(M)$ and let $W_t$ be a graph manifold obtained by gluing $N_{t}$ to $M$ along each $T_j$ in such a way that $[h_0]$ is identified with $[\alpha_j]$.  Recall that $\mathcal{D}_{fol}(M)$ is the set of all $[\alpha_*] \in \mathcal{S}(M)$ which are foliation detected. 

\begin{proposition} \label{foliation detected implies l-space}
If $[\alpha_*] \in \mathcal{D}_{fol}(M)$, then no $W_t$ is an L-space.
\end{proposition}

\begin{proof}
Suppose that some $W_t$ is an L-space and $[\alpha_*] \in \mathcal{D}_{fol}(M)$. If $v([\alpha_*]) > 0$, then $v([\alpha_*]) \geq 2$ by Proposition \ref{foliation vertical}(2), which implies that $W_t$ is not a rational homology $3$-sphere, contrary to the assumption that it is an L-space. Thus $[\alpha_*]$ is horizontal. But then Proposition \ref{foliation iff detected} implies that $W_t$ admits a co-oriented taut foliation and so cannot be an L-space by \cite[Corollary 9.2]{Bn}, \cite[Corollary 1.6]{KR}, or \cite[Theorem 1.1]{BC}. This completes the proof. 
\end{proof}
 
Set 
$$\mathcal{L}_{fol}(M) = \mathcal{S}(M) \setminus \mathcal{D}_{fol}(M)$$

\begin{proposition} \label{planar piece}
Let $M$ be a Seifert manifold with base orbifold $P(a_1, a_2, \ldots , a_n)$ as in \S \ref{assumptions seifert} and fix a rational element $[\alpha_*] = ([\alpha_1], [\alpha_2], \ldots , [\alpha_r])$ of $\mathcal{S}(M)$. Let $W_t$ be a graph manifold obtained by gluing $N_{t}$ to $M$ along each $T_j$ in such a way that $[h_0]$ is identified with $[\alpha_j]$. If 
$v([\alpha_*]) \geq 2$ or $v([\alpha_*]) = 0$ and $[\alpha_*] \in \mathcal{D}_{fol}(M)$, then no $W_t$ is an L-space. On the other hand, if $v([\alpha_*]) = 1$ or $v([\alpha_*]) = 0$ and $[\alpha_*] \in \mathcal{L}_{fol}(M)$, then $W_t$ is an L-space for $t \gg 0$. 
\end{proposition}

\begin{proof}
If $v([\alpha_*]) \geq 2$, then $W_t$ is not a rational homology $3$-sphere since it contains a closed non-separating orientable surface obtained by piecing together vertical annuli in $M$ with fibre surfaces in $N_t$. In particular, $W_t$ is not an L-space. 

If $v([\alpha_*]) =0 $ and $[\alpha_*] \in \mathcal{D}_{fol}(M)$, then $W_t$ is not an L-space by Proposition \ref{foliation detected implies l-space}.   

Conversely suppose that $v([\alpha_*]) = 1$ or $v([\alpha_*]) = 0$ and $[\alpha_*] \in \mathcal{L}_{fol}(M)$. We prove that $W_t$ is an L-space for $t \gg 0$ by induction on $r$.

{\bf Base case}. Suppose that $r = 1$, so $P \cong D^2$. Let $h \in H_1(\partial M)$ represent the slope of the Seifert fibre on $\partial M$  and recall that we have fixed a presentation $\pi_1(M) = \langle y_1,  \ldots , y_n, h : h \hbox{ central, } y_1^{a_1} = h^{b_1}, \ldots , y_n^{a_n} = h^{b_n}  \rangle$. The class $y_1 y_2 \ldots y_n$ is peripheral and represents a dual class $h^* \in H_1(\partial M)$ to $h$. A simple calculation shows that if $q = a_1a_2\ldots a_n$ and $p = \sum_{i=1}^n \frac{qb_i}{a_i}$, then 
$$[\lambda_M] = [ph - qh^*]$$ 
Note that $M(h)$ is a connected sum of lens spaces and hence an L-space. 

Corollary \ref{hen and r = 1} implies that $\mathcal{D}_{fol}(M)$ is a closed interval in $\mathcal{S}(M) \cong S^1$ containing $[\lambda_M]$. Thus $\mathcal{L}_{fol}(M)$ is an open interval containing $[h]$ but not $[\lambda_M]$. Recall $\mathcal{T}(M) = \overline{\{ \frac{r}{s} : M( r h - s h^* ) \mbox{ is not an L-space}\}} \subset \mathbb R$. Then by Propositions \ref{reps to foliations} and \ref{tau fibre 2}, $\mathcal{D}_{fol}(M) = \{[\gamma h - h^*] : \gamma \in \mathcal{T}(M)\}$ so there are rational numbers $\eta \leq \zeta$ such that  
$$\frac{p}{q} \in \mathcal{T}(M) = [\eta, \zeta]$$

Let $f: \partial M \to \partial N_t$ be the gluing map and set $W_t(f) = M \cup_f N_t$ for $t \geq 2$. Write $f_* =   \left(\begin{smallmatrix} a  &   b \\ c   &  d \end{smallmatrix}\right)$ with respect to the bases $\{h, h^*\}$ and $\{ h_0, h_1\}$. We can always assume $f_* =   \left(\begin{smallmatrix} a  &   b \\ c   &  d \end{smallmatrix}\right)$ where $c \geq 0$ and $\det(f_*) = 1$ by Lemma \ref{homeo N_t}. Note that 
$$[\alpha_*] = [f^{-1}(h_0)] = [d h - ch^*]$$

First suppose that $c = 0$, so $[\alpha_*] = [h]$. Then $a = d = \pm 1$ and by a further application of Lemma \ref{homeo N_t} we can take $a = d = 1$. By Proposition \ref{general HF solid torus}, it suffices to deal with the case that $b = \min \{\lfloor \eta \rfloor , 1\}$. Lemma \ref{fibre surgery} implies that $W_t$ is an L-space as long as $b + t, b - \frac{t}{t-1} \not \in [\eta, \zeta]$ and $-t < b - \frac{p}{q} < \frac{t}{t-1}$. Since $\frac{p}{q} > 0$, these inequalities hold for $t \gg 0$. 

If $c = 1$, we can assume that $a = 0$, using Proposition \ref{general HF solid torus}. Then $W_t(f)$ is Seifert with base orbifold of the form $S^2(a_1, a_2, \ldots , a_n,t,t)$ and is therefore {\it not} an L-space if and only if it admits a horizontal foliation (\cite{LS}). But if $W_t(f)$ admits a horizontal foliation, then $[\alpha_*] \in \mathcal{D}_{fol}(M)$ by Proposition \ref{foliation iff detected}, contrary to our assumptions. Hence $W_t(f)$ is an L-space for each $t \geq 2$ when $c = 1$.  

Suppose now that $c \geq 2$. As $[\alpha_*] \in \mathcal{L}_{fol}(M)$, $\frac{d}{c} \not \in [\eta, \zeta]$. Write $\eta = \frac{j}{k}$ where $k > 0$ and $\zeta = \frac{l}{m}$ where $m > 0$. We will show that if $t > km$, then $W_t$ is an L-space. We consider the case that $\frac{d}{c} < \eta$ first. 

\begin{claim} \label{dt}
If $\frac{d}{c} + \frac{1}{c(tc + 1)} < \eta$, then $W_t(f)$ is an L-space. 
\end{claim} 

\begin{proof}
We show that the hypotheses of Lemma \ref{fibre surgery} are satisfied under the conditions of this claim. 

The reader will verify that $\frac{pa-qb}{pc-qd} = \frac{a}{c} + \frac{q}{c(pc-qd)}$ and therefore that the condition ``$-t < \frac{pa-qb}{pc-qd}  < \frac{t}{t-1}$" of Lemma \ref{fibre surgery} is equivalent to ``$-c(tc + a) < \frac{1}{\frac{p}{q} - \frac{d}{c}}  < c\big((\frac{t}{t-1})c - a\big)$". Choosing $a$ so that $-tc < a < -(t-1)c$ (cf. Corollary \ref{preserve l-spaces}), this is equivalent to $\frac{d}{c} + \frac{1}{c\big((\frac{t}{t-1})c - a\big)} < \frac{p}{q}$. With this value of $a$ we have $\frac{d}{c} + \frac{1}{c\big((\frac{t}{t-1})c - a\big)} < \frac{d}{c} + \frac{1}{c(tc +1)} < \eta \leq \frac{p}{q}$ and therefore $-t < \frac{pa-qb}{pc-qd}  < \frac{t}{t-1}$. 

Next observe that $\frac{b + td}{a + tc} = \frac{d}{c} - \frac{1}{c(ct+a)} < \frac{d}{c} < \eta$ while $\frac{(1-t)b + td}{(1-t)a + tc} = \frac{d}{c} + \frac{1}{c\big((\frac{t}{t-1})c - a\big)} < \frac{d}{c} + \frac{1}{c(tc +1)} < \eta$. Hence $\frac{b + td}{a + tc}, \frac{(1-t)b + td}{(1-t)a + tc} \not \in \mathcal{T}(R)$ so that $W_t(f)$ is an L-space by Lemma \ref{fibre surgery}.
\end{proof}

Suppose that $t > km$. Then $\eta - \frac{d}{c} = \frac{cj - dk}{kc} \geq \frac{1}{kc} > \frac{1}{c(tc +1)}$ so in particular, $\frac{d}{c} + \frac{1}{c(tc +1)} < \eta$. Thus $W_t$ is an L-space. 

A similar argument shows that if $\frac{d}{c} > \zeta$ and $t > km$, then $W_t$ is an L-space. This completes the base case of the induction. 

{\bf Inductive case}. Suppose that the result holds when $1 \leq |\partial M| < r$ and let $R_t$ be a manifold obtained by gluing $r - 1$ copies of $N_t$ to $M$ along $\partial M$ in such a way that for each $1 \leq j \leq r-1$, $[h_0]$ is identified with $[\alpha_j]$. Then $\partial R_t = T_r$. Since $v([\alpha_*]) \leq 1$, we can suppose that $[\alpha_j]$ is horizontal for $1 \leq j \leq r-1$. Let $h \in H_1(\partial R_t)$ represent the slope of the Seifert fibre of $M$ and let $h^* \in H_1(\partial R_t)$ be a dual class to $h$ oriented so that there are coprime positive integers $p, q$ such that the rational longitude of $R_t$ can be written $\lambda_{R_t} = ph - qh^*$. Let $\mathcal{D}_{fol}(R_t) = \{[\alpha] \in \mathcal{S}(T_r) : [\alpha] \hbox{ is detected  by a co-oriented, taut foliation on } R_t\}$. It follows from the proof of Proposition \ref{foliation iff detected} that $\mathcal{D}_{fol}(R_t) = \{[\beta] \in \mathcal{S}(T_r) : ([\alpha_1],  \ldots, [\alpha_{r-1}], [\beta]) \in \mathcal{D}_{fol}(M)\}$. Hence, $[\alpha_r] \in \mathcal{S}(R_t) \setminus \mathcal{D}_{fol}(R_t)$. Note as well that if $\tau_* = (\tau_1, \ldots, \tau_{r-1})$ where $[\alpha_j] = [\tau_j h - h_j^*]$ ($1 \leq j \leq r-1$), then by Proposition \ref{reps to foliations} and the discussion in the Appendix we have $\mathcal{D}_{fol}(R_t) = \{[\tau'h - h^*_r] : \tau' \in \mathcal{T}(M; \emptyset; \tau_*)\}$. Thus $\mathcal{D}_{fol}(R_t)$ is homeomorphic to $\mathcal{T}(M; \emptyset; \tau_*)$ which, by Corollary \ref{hen and detected slopes}, is an interval $[\eta, \zeta]$ with rational endpoints. 

We claim that our inductive hypothesis implies that for each rational $[\alpha_r'] \in \mathcal{S}(T_r) \setminus \mathcal{D}_{fol}(R_t)$, Dehn filling $R_t$ along $T_r$ with slope $[\alpha_r']$ is an L-space for large $t$. To see this, first note that for each such $[\alpha_r']$, $([\alpha_1], \ldots, [\alpha_{r-1}], [\alpha_r']) \not \in \mathcal{D}_{fol}(M)$. Hence $v(([\alpha_1], [\alpha_2], \ldots, [\alpha_{r-1}], [\alpha_r'])) \leq 1$ by Proposition \ref{foliation vertical}(2). 

Consider the effect of Dehn filling $M$ along $T_r$ with slope $[\alpha_r']$ to produce a new manifold $M'$. We suppose, first of all, that $[\alpha_r']$ is horizontal. Then $M'$ is Seifert fibred with base orbifold a planar surface with $r-1$ boundary components and possibly with cone points. In particular, it is either $S^1 \times D^2$ or $S^1 \times S^1 \times I$ or satisfies our standard hypotheses (\S \ref{assumptions seifert}). 

Let $[\alpha_*']$ be the projection of $[\alpha_*]$ to $\mathcal{S}(M')$. By construction, $v([\alpha_*']) = 0$. Note as well that $[\alpha_*' ] \in \mathcal{L}_{fol}(M')$ as otherwise, arguing as in the proof of Proposition \ref{foliation iff detected} would show that $([\alpha_1], \ldots, [\alpha_{r-1}], [\alpha_r']) = ([\alpha_*'], [\alpha_r']) \in \mathcal{D}_{fol}(M) = \mathcal{S}(M) \setminus \mathcal{L}_{fol}(M)$, contrary to our assumptions. Let $W_t'$ be a manifold obtained by gluing $r - 1$ copies of $N_t$ to $M'$ along $\partial M'$ in such a way that for each $1 \leq j \leq r-1$, $[h_0]$ is identified with $[\alpha_j'] = [\alpha_j]$.

We claim that $W_t'$ is an L-space for $t \gg 0$. If $M'$ satisfies our standard hypotheses, this is an immediate consequence of the induction hypothesis. If $M' = S^1 \times D^2$, then $[\alpha_1'] \in \mathcal{L}_{fol}(M') = \mathcal{S}(M') \setminus \{[\lambda_{M'}]\}$, so $W_t'$ is a Dehn filling of $N_t$ along a slope other than $[h_0]$. It follows from Proposition \ref{tau fibre 2}(4)(c) that $\mathcal{D}_{fol}(N_t) = \{[h_0]\}$ and therefore the Seifert manifold $W_t'$ does not admit a horizontal foliation. Hence it is an L-space (\cite{LS}). Finally, if $M' = S^1 \times S^1 \times I$, then $([\alpha_1'], [\alpha_2']) = [\alpha_*'] \in \mathcal{L}_{fol}(M')$ implies that the slopes $[\alpha_1']$ and $[\alpha_2']$ are distinct. Hence $W_t'$ is the union along their boundaries of two copies of $N_t$ whose rational longitudes do not coincide. Hence we can apply the base case of the induction to see that the claim holds. Thus Dehn filling $R_t$ along $T_r$ with slope $[\alpha_r']$ is an L-space for large $t$ when $[\alpha_r']$ is horizontal. 

Suppose next that $[\alpha_r'] = [h]$. In this case $M'$ is a connected sum of $r-1$ copies of $S^1 \times D^2$, each with meridional slope a Seifert fibre of $M$. Hence $W_t'$ is the connected sum of $(r-1)$ manifolds obtained by Dehn filling $N_t$ along a slope different from $[\lambda_{N_t}] = [h_0]$  since $[\alpha_i]$ is horizontal for $1 \leq i \leq r-1$. Thus it is an L-space for all $t \geq 2$, which completes the proof that for each rational $[\alpha_r'] \in \mathcal{S}(T_r) \setminus \mathcal{D}_{fol}(R_t)$, Dehn filling $R_t$ along $T_r$ with slope $[\alpha_r']$ is an L-space for large $t$. Put another way, for each pair of coprime integers $u, v$ such that $\frac{u}{v} \not \in [\eta, \zeta]$, $R_t(u h - v h_r^*)$ is an L-space for all large values of $t$. 

Now we complete the induction. Let $f: \partial R_t \to \partial N_t$ be a gluing map which identifies $[\alpha_r]$ to $[h_0]$ and write $f_* =   \left(\begin{smallmatrix} a  &   b \\ c   &  d \end{smallmatrix}\right)$ with respect to the bases $\{h, h^*\}$ and $\{ h_0, h_1\}$. As above we can suppose that $c \geq 0$ and $\det(f_*) = 1$. We have $[\alpha_r] = [f^{-1}(h_0)] = [d h - ch^*]$. 

First suppose that $c = 0$, so $[\alpha_r] = [h]$. Lemma \ref{homeo N_t} implies that we can take $a = d = 1$, and by Proposition \ref{general HF solid torus} it suffices to deal with the case that $b = \min \{\lfloor \eta \rfloor , 1\}$. In this case, $b + t, b - \frac{t}{t-1} \not \in [\eta, \zeta]$ for large $t$. It follows that $R_t((b+t) h - h_r^*)$ and $R_t(((t-1)b - t) h - (t-1)h_r^*)$ are L-spaces for all large values of $t$. Since $-t < b - \frac{p}{q} < \frac{t}{t-1}$ for such $t$, Lemma \ref{fibre surgery} implies that $W_t$ is an L-space for $t \gg 0$. 

If $c = 1$, we can assume that $a = 0$, using Proposition \ref{general HF solid torus}. In this case $M' = M \cup_f N_t$ is Seifert with base orbifold $P'(a_1, a_2, \ldots , a_n, t,t)$ where $P'$ is an $(r-1)$-punctured $2$-sphere. Note that $([\alpha_1], \ldots , [\alpha_{r-1}]) \not \in \mathcal{D}_{fol}(M')$ as otherwise $[\alpha_*] \in \mathcal{D}_{fol}(M)$ by the proof of Proposition \ref{foliation iff detected}. Thus  $([\alpha_1], \ldots , [\alpha_{r-1}]) \in \mathcal{L}_{fol}(M')$, so our inductive hypothesis implies that $W_t$ is an L-space for $t \gg 0$. 

Finally, when $c \geq 2$ we can proceed exactly as in the case $c = 0$ and as in the proof of the claim in the base case to see that $W_t(f)$ is an L-space for $t \gg 0$, which completes the induction.  
\end{proof}

\subsection{L-space $N_t$-fillings of $M$}

\begin{definition} 
{\rm Let $M$ be a compact orientable Seifert fibred manifold as in \S \ref{assumptions seifert}. For $J \subset \{1, 2, \ldots, r\}$ and rational $[\alpha_*] = ([\alpha_1], [\alpha_2], \ldots , [\alpha_r]) \in \mathcal{S}(M)$, let $\mathcal{M}_t(J; [\alpha_*])$ denote the set of manifolds obtained by doing $[\alpha_j]$-Dehn filling of $M$ for $j \in J$, and for each $j \notin J$ attaching $N_{t}$ to $M$ in such a way that the rational longitude $[h_0]$ of $N_{t}$ is identified with $[\alpha_j]$. }
\end{definition}

\begin{remark} \label{filled version}
{\rm Let $M'$ be obtained by Dehn filling $M$ along all horizontal rational slopes $[\alpha_j]$ with $j \in J$. Then $M'$ inherits a Seifert structure from $M$ and $\mathcal{M}_t(J; [\alpha_*]) = \mathcal{M}_t'(J'; [\alpha_*'])$ where $[\alpha_*']$ is the projection of $[\alpha_*]$ to $\mathcal{S}(M')$ and $J' = \{j \in J : [\alpha_j] = [h]\}$. Note that $M'$ is closed if and only if $[\alpha_*]$ is horizontal and $J = \{1, 2, \ldots , r\}$. 
}
\end{remark}
Here are two corollaries of the propositions above. 

\begin{corollary}  \label{not horizontal}
Let $M$ be a compact orientable Seifert fibred manifold as in \S \ref{assumptions seifert}. Fix $t \geq 2$, $J \subset \{1, 2, \ldots, r\}$ and rational $[\alpha_*] = ([\alpha_1], [\alpha_2], \ldots , [\alpha_r]) \in \mathcal{S}(M)$. If $v([\alpha_*]) > 0$, then

$(1)$ No element of $\mathcal{M}_t(J; [\alpha_*])$ is an L-space if either $v([\alpha_*]) \geq 2$ or $v([\alpha_*]) = 1$ and $M$ has base orbifold $Q(a_1, \ldots, a_n)$.

$(2)$ Each element of $\mathcal{M}_2(J; [\alpha_*])$ is an L-space if $v([\alpha_*]) = 1$ and $M$ has base orbifold $P(a_1, \ldots, a_n)$.

\end{corollary}

\begin{proof}
If $v([\alpha_*]) \geq 2$ or if $v([\alpha_*]) = 1$ and $M$ has base orbifold $Q(a_1, \ldots, a_n)$, then each element of $\mathcal{M}_t(J; [\alpha_*])$ admits a co-orientable taut foliation no matter what $J$ is, and so is not an L-space. Suppose that $v([\alpha_*]) = 1$ and $M$ has base orbifold $P(a_1, \ldots, a_n)$. Let $M'$ and $J'$ be as in Remark \ref{filled version}. Note that $M'$ is not closed since $[\alpha_*]$ is not horizontal. If $J' = \emptyset$, each element of $\mathcal{M}_2(J; [\alpha_*]) = \mathcal{M}_2'(J'; [\alpha_*'])$ is an L-space by 
Proposition \ref{mobius piece} applied to $M'$. Otherwise, each element of $\mathcal{M}_2(J; [\alpha_*]) = \mathcal{M}_2'(J'; [\alpha_*'])$ is a connected sum of lens spaces, and so is an L-space. 
\end{proof}

\begin{corollary}  \label{independent}
Let $M$ be a compact orientable Seifert fibred manifold as in \S \ref{assumptions seifert} and fix $J \subset \{1, 2, \ldots, r\}$ and rational $[\alpha_*] = ([\alpha_1], [\alpha_2], \ldots , [\alpha_r]) \in \mathcal{S}(M)$. There exists $t \geq 2$ such that some manifold in $\mathcal{M}_t(J; [\alpha_*])$ is an L-space if and only if each manifold in $\mathcal{M}_t(J; [\alpha_*])$ is an L-space.  
\end{corollary}

\begin{proof}
The previous corollary shows that the result holds with $t = 2$ when $v([\alpha_*]) > 0$. Assume then that $v([\alpha_*]) = 0$ and let $M'$ and $J'$ be as in Remark \ref{filled version}. Clearly $J' = \emptyset$. If $M'$ is closed, $\mathcal{M}_t(J; [\alpha_*]) = \{M'\}$, so the result is obvious. Otherwise it is a direct consequence of Propositions \ref{mobius piece} and \ref{planar piece} applied to $M'$. 
\end{proof}

\subsection{Detecting slopes via L-spaces}

\begin{definition} \label{NLS detection}
{\rm Let $[\alpha_*] = ([\alpha_1], [\alpha_2], \ldots , [\alpha_r])$ be a {\it rational} element of $\mathcal{S}(M)$.  
For $J \subset \{1, 2, \ldots, r\}$, we say that $(J; [\alpha_*])$ is {\it NLS detected} if for each $t \geq 2$, no manifold in $\mathcal{M}_t(J; [\alpha_*])$ is an L-space.
}
\end{definition} 

We shall often simplify the phrase ``$(\emptyset; [\alpha_*])$ is NLS detected" to ``$[\alpha_*]$ is NLS detected". Similarly, we simplify ``$(\{1, 2, \ldots, r\}; [\alpha_*])$ is NLS detected," to ``$[\alpha_*]$ is {\it strongly} NLS detected".

\begin{remark} \label{NLS reduction} 
{\rm  We expect that $[\alpha_*]$ is NLS detected if and only if no manifold in $\mathcal{M}_2(J; [\alpha_*])$ is an L-space, so that $t=2$ suffices in the definition of NLS detection.
}
\end{remark}

Set 
$$\mathcal{D}_{NLS}(M; J) = \{[\alpha_*] \in \mathcal{S}(M) : (J; [\alpha_*]) \mbox{ is NLS detected}\}$$
When $J = \emptyset$ we simplify $\mathcal{D}_{NLS}(M; J)$ to $\mathcal{D}_{NLS}(M)$.

\section{The slope detection theorem} \label{section detection}

In this section we state and prove the slope detection theorem. 

\begin{theorem} \label{theorem: detection} 
Let $M$ be a compact orientable Seifert fibred manifold  as in \S \ref{assumptions seifert} and fix $J \subseteq \{1, 2, \ldots, r\}$. Suppose that $[\alpha_*] \in \mathcal{S}(M)$ is an $r$-tuple of slopes such that $[\alpha_j] \ne [h]$ for $j \in J$. Then the  following statements are equivalent.

$(1)$ $(J; [\alpha_*])$ is order detected. 

$(2)$ $(J; [\alpha_*])$ is foliation detected.

$(3)$ If $[\alpha_*]$ is horizontal, $(J; [\alpha_*])$ is representation detected.  

$(4)$ If $[\alpha_*]$ is rational, $(J; [\alpha_*])$ is NLS-detected.     

\end{theorem}
Theorem \ref{theorem: special detection} is the case $J = \emptyset$ of Theorem \ref{theorem: detection}. (Note that this theorem is easy to verify when $M$ is a solid torus or the product of a torus with an interval. Theorem \ref{theorem: detection} handles the remaining cases.)  

The proof of Theorem \ref{theorem: detection} naturally splits into two cases. 

\subsection{The case that $[\alpha_*]$ is horizontal}
We must show that the following statements are equivalent:
\begin{enumerate}

\item[$\bullet$] $(J; [\alpha_*])$ is order detected. 

\vspace{.2cm} \item[$\bullet$]  $(J; [\alpha_*])$ is representation detected.

\vspace{.2cm} \item[$\bullet$] $(J; [\alpha_*])$ is foliation detected.

\vspace{.2cm} \item[$\bullet$] If $[\alpha_*]$ is rational, $(J; [\alpha_*])$ is NLS-detected.     

\end{enumerate}
The proof is a consequence of the following two propositions. 

\begin{proposition} \label{non-orientable and vertical}
Suppose that $M$ has base orbifold $Q(a_1, \ldots , a_n)$ and $[\alpha_*] \in \mathcal{S}(M)$ is horizontal. Then 

$(1)$ $(J; [\alpha_*])$ is not order detected. 

$(2)$ $(J; [\alpha_*])$ is not representation detected.

$(3)$ $(J; [\alpha_*])$ is not foliation detected.

$(4)$ If $[\alpha_*]$ is rational, $(J; [\alpha_*])$ is not NLS-detected.   
\end{proposition}

\begin{proof}
The underlying space of $Q$ is non-orientable and therefore $M$ admits no co-oriented, horizontal foliation. Thus $(J; [\alpha_*])$ is not foliation detected. It is neither order detected nor representation detected by Proposition \ref{non-orientable implies not detected} and Lemma \ref{non-orientable implies empty}. 

Finally suppose that $[\alpha_*]$ is rational and let $M'$ be the manifold obtained by performing $[\alpha_j]$-Dehn filling of $M$ for $j \in J$. Then $M'$ is a Seifert fibred manifold whose base orbifold has underlying space a (possibly) punctured projective plane. If $M'$ is closed, it is an L-space (\cite[Proposition 5]{BGW}) and so as $J = \{1, 2, \ldots , r\}$ in this case, $(J; [\alpha_*])$ is not NLS detected. Suppose then that $M'$ is not closed and let $[\alpha_*'] \in \mathcal{S}(M')$ be  the projection of $[\alpha_*]$. By construction, $[\alpha_*']$ is horizontal with respect to the induced Seifert structure on $M'$ and therefore $[\alpha_*']$ is not NLS detected by Proposition \ref{mobius piece}. But then $(J; [\alpha_*])$ is not NLS detected, which completes the proof. 
\end{proof} 

\begin{proposition} \label{orientable and vertical}
Suppose that $M$ has base orbifold $P(a_1, \ldots , a_n)$ and $[\alpha_*] \in \mathcal{S}(M)$ is horizontal. Then the following statements are equivalent:

$(1)$ $(J; [\alpha_*])$ is order detected. 

$(2)$ $(J; [\alpha_*])$ is representation detected.

$(3)$ $(J; [\alpha_*])$ is foliation detected.

$(4)$ If $[\alpha_*]$ is rational, $(J; [\alpha_*])$ is NLS-detected.     

\end{proposition}

\begin{proof} 
The equivalence of (1) and (2) is contained in Propositions \ref{reps to orders} and \ref{dynamical detects} while that of (2) and (3) is contained in Proposition \ref{reps to foliations}. Finally we show that (3) is equivalent to (4) when $[\alpha_*]$ is rational. 

Suppose that $[\alpha_*]$ is rational and let $M'$ be the manifold obtained by performing $[\alpha_j]$-Dehn filling of $M$ for $j \in J$. Then $M'$ is a Seifert fibred manifold whose base orbifold has underlying space a (possibly) punctured $2$-sphere. If $M'$ is closed then $J = \{1, 2, \ldots , r\}$ and  $M'$ is Seifert fibred with base orbifold a $2$-sphere with cone points. In this case it was shown in \cite{LS} that $M'$ is not an L-space if and only if it admits a horizontal foliation. As the latter is equivalent to the foliation detectability of $(\{1, 2, \ldots , r\}; [\alpha_*])$, (3) and (4) are equivalent when $M'$ is closed. 

Suppose then that $\partial M' \ne \emptyset$ and define $[\alpha_*'] \in \mathcal{S}(M')$ to be the projection of $[\alpha_*]$. By construction, $[\alpha_*']$ is horizontal with respect to the induced Seifert structure on $M'$. It is clear from Proposition \ref{prop: horizontal foliation} that $(J; [\alpha_*])$ is foliation detected if and only if $[\alpha_*']$ is foliation detected and it follows from the definition of NLS detection that $(J; [\alpha_*])$ is NLS detected if and only if $[\alpha_*']$ is NLS detected. On the other hand, Proposition \ref{planar piece} shows that $[\alpha_*']$ is not NLS detected if and only if it is not foliation detected. Thus $(3)$ is equivalent to $(4)$ when $[\alpha_*]$ is rational and horizontal. 
\end{proof} 

\subsection{The case that $[\alpha_*]$ is not horizontal}

Assume that $v([\alpha_*]) > 0$ and $[\alpha_j] \ne [h]$ for $j \in J$. 
We must show that the following statements are equivalent:
\begin{enumerate}

\item[$\bullet$] $(J; [\alpha_*])$ is order detected. 

\vspace{.2cm} \item[$\bullet$] $(J; [\alpha_*])$ is foliation detected.

\vspace{.2cm} \item[$\bullet$] If $[\alpha_*]$ is rational, $(J; [\alpha_*])$ is NLS detected.     

\end{enumerate}
The proof is contained is the following two propositions. 

\begin{proposition} \label{non-orientable and vertical 2} 
Suppose that $M$ has base orbifold $Q(a_1, \ldots , a_n)$, $v([\alpha_*]) \geq 1$, and $[\alpha_j] \ne [h]$ for $j \in J$. Then

$(1)$ $(J; [\alpha_*])$ is order detected. 

$(2)$ $(J; [\alpha_*])$ is foliation detected.

$(3)$ If $[\alpha_*]$ is rational, $(J; [\alpha_*])$ is NLS-detected.   
\end{proposition}

\begin{proof}
Statements (1) and (2) follow from Propositions \ref{order vertical} and \ref{foliation vertical}. 

Suppose that $[\alpha_*]$ is rational and let $M'$ be the manifold obtained by performing $[\alpha_j]$-Dehn filling of $M$ for $j \in J$. Then $M'$ is a Seifert fibred manifold whose base orbifold has underlying space a punctured projective plane. Let $[\alpha_*'] \in \mathcal{S}(M')$ be  the projection of $[\alpha_*]$. Then $(J; [\alpha_*])$ is NLS detected if and only if $[\alpha_*']$ is NLS-detected. By construction, $v([\alpha_*]) = v([\alpha_*'])\geq 1$. Statement (3) now follows from Proposition \ref{mobius piece} applied to $M'$ and $[\alpha_*']$. 
\end{proof}

\begin{proposition} \label{orientable and vertical 2}
Suppose that $M$ has base orbifold $P(a_1, \ldots , a_n)$, that $[\alpha_j] \ne [h]$ for $j \in J$, and that $v([\alpha_*]) \geq 1$. Then the following are
equivalent.

$(1)$ $v([\alpha_*]) \geq 2$.

$(2)$ $(J; [\alpha_*])$ is order detected.

$(3)$ $(J; [\alpha_*])$ is foliation detected.

$(4)$ If $[\alpha_*]$ is rational, $(J; [\alpha_*])$ is NLS-detected.

\end{proposition}

\begin{proof}
Propositions \ref{order vertical} and \ref{foliation vertical} imply assertions (1), (2) and (3) are equivalent. 

Suppose that $[\alpha_*]$ is rational and let $M'$ be the manifold obtained by performing $[\alpha_j]$-Dehn filling of $M$ for $j \in J$. Since $v([\alpha_*]) > 0$ and $[\alpha_j]$ is horizontal for all $j \in J$, $M'$ is a Seifert fibred manifold with non-empty boundary whose base orbifold has underlying space a punctured $2$-sphere. Let $[\alpha_*'] \in \mathcal{S}(M')$ be  the projection of $[\alpha_*]$. Then $(J; [\alpha_*])$ is NLS detected if and only if $[\alpha_*']$ is NLS-detected. By construction, $v([\alpha_*']) = v([\alpha_*]) \geq 1$. The equivalence of (1) and (4) now follows from Proposition \ref{planar piece} applied to $M'$ and $[\alpha_*']$.
\end{proof}

\section{The gluing theorem} \label{section gluing}

For $M$ a compact orientable Seifert fibred manifold as in \S \ref{assumptions seifert}, $J \subseteq \{1, 2, \ldots, r\}$, and $[\alpha_*] \in \mathcal{S}(M)$ such that $[\alpha_j] \ne [h]$ for $j \in J$ we say that $(J; [\alpha_*])$ is {\it detected} if it is foliation detected. By Theorem \ref{theorem: detection} this is the same as being order detected. (And also to representation detected or NLS detected when both notions are defined.) 

Fix a graph manifold $W$ as in \S \ref{assumptions graph} with JSJ pieces $M_1, M_2, \ldots, M_n$ and JSJ tori $T_1, T_2, \ldots ,T_m$. Recall that for $[\alpha_*] \in \mathcal{S}(W; \mathcal{T})$ and $i \in \{1, 2, \ldots , n\}$ we defined 
$$[\alpha_*^{(i)}] = \Pi_i([\alpha_*])$$ 
See \S \ref{assumptions graph}. For $K \subseteq \{1,2, \ldots, m\}$ and $i \in \{1, 2, \ldots , n\}$ set 
$$K_i = \{k \in K : T_k \subset \partial M_i\}$$ 
and define $[\alpha_*^{K_i}]$ to be the $|K_i|$-tuple of slopes $[\alpha_j]$ where $j \in K_i$. 

Before stating the gluing theorem, we introduce several notions.

\begin{definition}
{\rm Fix $K \subseteq \{1,2, \ldots, m\}$ and  $[\alpha_*] \in \mathcal{S}(W; \mathcal{T})$. We call $(K; [\alpha_*])$  {\it gluing coherent} if $(K_i; [\alpha_*^{(i)}])$ is detected for all $i$. 
}
\end{definition}

It follows from Lemma \ref{strongly rational not vertical} that if $(K; [\alpha_*])$ is gluing coherent and $k \in K$, then $[\alpha_k]$ is horizontal in each piece of $W$ containing $T_k$, at least up to assuming that the Seifert structures on pieces homeomorphic to twisted $I$-bundles over the Klein bottle have orientable base orbifolds. 

Given $K \subseteq \{1,2, \ldots, m\}$ and $[\alpha_*] \in \mathcal{S}(W; \mathcal{T})$ let $M_i([\alpha_*^{K_i}]_{rat})$ be the Seifert manifold obtained by $[\alpha_j]$-Dehn filling $M_i$ along its boundary components $T_j$ such that $j \in K_i$ and $[\alpha_j]$ is rational. Set 
$$K([\alpha_*]) = K \cup \{j : T_j = \partial M_i([\alpha_*^{K_i}]_{rat}) \hbox{ for some $i$ such that } M_i([\alpha_*^{K_i}]_{rat}) \cong S^1 \times D^2\}$$ 

\begin{definition}
{\rm Given $K \subseteq \{1,2, \ldots, m\}$ and $[\alpha_*] \in \mathcal{S}(W; \mathcal{T})$ we say that $(K; [\alpha_*])$ is {\it gluing unobstructed} if $(K([\alpha_*]); [\alpha_*])$ is gluing coherent. Otherwise we say that $(K; [\alpha_*])$ is {\it gluing obstructed}. 
}
\end{definition}

Note that when $K = \emptyset$ or $K = \{1, 2, \ldots, m\}$, $(K; [\alpha_*])$ is gluing unobstructed as long as it is gluing coherent. See Example \ref{obstructed} for an example of a $W$ and $(K; [\alpha_*])$ which is gluing coherent but gluing obstructed.

\begin{definition}
{\rm $(1)$ We say that a co-oriented taut foliation $\mathcal{F}$ on $W$ has {\it $K$-type} if $\mathcal{F}$ is transverse to $T_j$ for each $j$, it restricts to a co-oriented taut foliation on each $M_i$, and $\mathcal{F} \cap T_k$ is linear for $k \in K$. 

$(2)$ We say that a left-order $\mathfrak{o}$ on $\pi_1(W)$ has {\it $K$-type} if there is an $\mathfrak{o}$-convex normal subgroup $C$ of $\pi_1(W)$ such that $C \cap \pi_1(T_k) \cong \langle \alpha_k \rangle \cap \pi_1(T_k)$ for all $k \in K$.
}
\end{definition}

\begin{convention}
For the rest of the paper we take the convention that the parenthetical phrases in the statements of results are to be either simultaneously considered or simultaneously ignored.
\end{convention}

Here is the gluing theorem.

\begin{theorem} \label{theorem: gluing} 
Let $W$ be a graph manifold rational homology $3$-sphere with pieces $M_1, \ldots, M_n$ and JSJ tori $T_1, \ldots , T_m$. Fix $K \subseteq \{1, 2, \ldots,m\}$. 

$(1)$ $\pi_1(W)$ admits a $K$-type left-order $($for which each class represented by a Seifert fibre of a piece of $W$ is cofinal$)$ if and only if there is a $($horizontal$)$ $[\alpha_*] \in \mathcal{S}(W; \mathcal{T})$ such that $(K; [\alpha_*])$ is gluing unobstructed.

$(2)$ $W$ admits a $($horizontal$)$ $K$-type co-oriented taut foliation if and only if there is a $($horizontal$)$ $[\alpha_*] \in \mathcal{S}(W; \mathcal{T})$ such that $(K; [\alpha_*])$ is gluing unobstructed. 

\end{theorem}

We prove these results in the next two sections. For now we use them to deduce Theorems \ref{equiv 1},  \ref{equiv 2}, and \ref{equiv 3}. Theorem \ref{theorem: special gluing} is the case $K = \emptyset$ of Theorem \ref{theorem: gluing}.

\begin{proof}[Proof of Theorem \ref{equiv 1}]
We remarked in the introduction that statements (2) and (3) of Theorem \ref{equiv 1} are known to be equivalent (cf. \cite{Linnell}, \cite[Theorem 1.1(1)]{BRW}). The equivalence of statements (1) and (2) are immediate consequences of the case $J = \emptyset$ of Theorem \ref{theorem: detection} and the case $K = \emptyset$ of Theorem \ref{theorem: gluing}.
\end{proof} 

\begin{proof}[Proof of Theorem \ref{equiv 2}] 
Statements (2) and (3) of Theorem \ref{equiv 2} are equivalent by Remark \ref{orders versus reps}. Suppose that statement (1) holds and let $\mathcal{F}$ be a co-oriented horizontal foliation on $W$. Brittenham has shown that $\mathcal{F}$ is $\mathbb R$-covered. Indeed, he shows that given a Seifert fibre $L$ of a piece of $W$, each leaf of the pull-back $\widetilde{\mathcal{F}}$ of $\mathcal{F}$ to the universal cover of $W$ intersects the inverse image $\widetilde L$ of $L$ in exactly one point. (See \cite[\S 3]{Br3}.) Hence the leaf space $\mathcal{L}$ of $\widetilde{\mathcal{F}}$ can be identified with  $\widetilde L$. Since $L$ is a transverse loop to $\mathcal{F}$, it carries an element of infinite order in $\pi_1(W)$. Thus $\mathcal{L}$ is a line. Now $\pi_1(W)$ acts on $\mathcal{L}$ via deck transformations and from Brittenham's work we see that the class carried by $L$ acts without fixed points. As $L$ was arbitrary, this action determines a homomorphism $\rho: \pi_1(W) \to \hbox{Homeo}_+(\mathcal{L}) \cong \hbox{Homeo}_+(\mathbb R)$ for which the image of the class carried by $L$ is conjugate to $\hbox{sh}(\pm 1)$. Thus statement (3) holds. 

Conversely suppose that statement (3) holds and let $\rho_i = \rho|\pi_1(M_i)$. There is an associated co-oriented horizontal foliation $\mathcal{F}(\rho_i)$ on $M_i$ (cf. the proof of Proposition \ref{reps to foliations}) which detects some $[\alpha_*(\rho_i)] \in \mathcal{S}(M_i)$. The $[\alpha_*(\rho_i)]$ ($1 \leq i \leq n$) piece together to yield a horizontal $[\alpha_*] \in \mathcal{S}(W; \mathcal{T})$ for which $(\emptyset, [\alpha_*])$ is gluing coherent (and therefore gluing unobstructed as $K = \emptyset$). Theorem \ref{theorem: gluing} then implies that $W$ admits a co-oriented horizontal foliation, which completes the proof.  
\end{proof} 

\begin{proof}[Proof of Theorem \ref{equiv 3}] 
Statements (2) and (3) of Theorem \ref{equiv 3} are equivalent by Remark \ref{orders versus reps}. Next observe that $W$ admits a strongly rational co-oriented taut foliation if and only if it  admits a $K$-type co-oriented taut foliation where $K = \{1, 2, \ldots , m\}$. (In both cases the foliations are horizontal by Lemmas \ref{prop: horizontal foliation}, at least up to assuming that the Seifert structures on pieces homeomorphic to twisted $I$-bundles over the Klein bottle have orientable base orbifolds.) Hence statements (1) and (2) of Theorem \ref{equiv 3} are equivalent by Theorems \ref{theorem: detection} and \ref{theorem: gluing}.
\end{proof} 

\section{Proof of the gluing theorem: the foliation case} \label{proof foliation case gluing}

Recall that $W$ is a graph manifold rational homology $3$-sphere as in \S \ref{assumptions graph} with JSJ pieces $M_1, M_2, \ldots, M_n$ and JSJ tori $\mathcal{T} = \{T_1, T_2, \ldots, T_m\}$. We work with a fixed $K \subseteq \{1, 2, \ldots, m\}$ throughout this section. We must show that $W$ admits a (horizontal) $K$-type co-oriented taut foliation if and only if there is a (horizontal) element $[\alpha_*] \in \mathcal{S}(W; \mathcal{T})$ such that $(K; [\alpha_*])$ is gluing unobstructed. 

The forward implication is straightforward: If $W$ admits a (horizontal) $K$-type co-oriented taut foliation $\mathcal{F}$, it induces a (horizontal) element $[\alpha_*] \in \mathcal{S}(W; \mathcal{T})$ such that $(K; [\alpha_*])$ is gluing coherent (\cite{BR}). We claim that $(K; [\alpha_*])$ is also gluing unobstructed. To see,  this observe that for each piece $M_i$ of $W$, $\mathcal{F}$ induces a  co-oriented taut foliation $\mathcal{F}_i$ on $M_i([\alpha_*^{K_i}]_{rat})$. As the only taut co-orientable foliations on a solid torus are $2$-disk fibrations, if $M_i([\alpha_*^{K_i}]_{rat}) \cong S^1 \times D^2$, then $\mathcal{F}$ restricts to a foliation on $T_j = \partial M_i([\alpha_*^{K_i}]_{rat})$ which strongly detects $[\alpha_j]$. It follows that $(K([\alpha_*]); [\alpha_*])$ is gluing coherent and therefore $(K; [\alpha_*])$ is gluing unobstructed.

Now we focus on the reverse implication. We suppose below that there is a (horizontal) element $[\alpha_*] \in \mathcal{S}(W; \mathcal{T})$ such that $(K; [\alpha_*])$ is gluing unobstructed. Lemma \ref{strongly rational not vertical} implies the following fact. 

\begin{lemma}
If $j \in K([\alpha_*])$ and $[\alpha_j]$ is vertical in a JSJ piece $M_i$ of $W$ incident to $T_j$, then $M_i \cong N_2$. 
\qed
\end{lemma}

For $A \subseteq \{1, 2, \ldots, m\}$, let $A^\dagger = A \cup \{ j : [\alpha_j] \hbox{ is irrational}\}$. By Proposition \ref{irrational to rational - foliations}, $(K^\dagger; [\alpha_*])$ is gluing coherent. Since $M_i([\alpha_*^{K_i^\dagger}]_{rat})= M_i([\alpha_*^{K_i}]_{rat})$ for each $i$, we have 
$$K^\dagger([\alpha_*]) = K([\alpha_*])^\dagger,$$ 
and it follows that $(K, [\alpha_*])$ is gluing unobstructed if and only if $(K^\dagger; [\alpha_*])$ is gluing unobstructed. Hence, without loss of generality, we assume that 
$$K = K^{\dagger}$$ 
for the rest of this section. 

\begin{lemma} \label{can assume rational} 
Fix $K \subseteq \{1,2, \ldots, m\}$ and a $($horizontal$)$ $[\alpha_*] \in \mathcal{S}(W; \mathcal{T})$. If $(K; [\alpha_*])$ is gluing unobstructed, there is a $($horizontal$)$ rational $[\alpha_*'] \in \mathcal{S}(W; \mathcal{T})$ such that 

$(1)$ $(K; [\alpha_*'])$ is gluing unobstructed;

$(2)$ $[\alpha_j'] = [\alpha_j]$ if $[\alpha_j]$ is rational; 

$(3)$ if $[\alpha_j]$ is irrational, then $[\alpha_j']$ can be chosen to lie in an arbitrarily small neighborhood of $[\alpha_j]$ in $\mathcal{S}(T_j)$ and to have distance at least $2$ to the Seifert fibre of a piece of $W$ containing $T_j$.  
\end{lemma}

\begin{proof}
First we show that the lemma holds if we replace (1) by (1)$'$: {\it $(K([\alpha_*]); [\alpha_*'])$ is gluing coherent}. 

By hypothesis, $(K([\alpha_*]); [\alpha_*])$ is gluing coherent, so the modified lemma is immediate if Proposition \ref{irrational to rational - foliations}(1) holds for each $(K([\alpha_*])_i; [\alpha_*^{(i)}])$ for which $[\alpha_*^{(i)}]$ is not rational. Otherwise, the pieces $M_i$ of $W$ for which $[\alpha_*^{(i)}]$ is not rational and Proposition \ref{irrational to rational - foliations}(2) holds occur in linear subtrees of the JSJ graph of $W$, the vertices of which correspond to the type of pieces described in Proposition \ref{irrational to rational - foliations}(2) and whose edges correspond to tori $T_j$ such that $[\alpha_j]$ is irrational. In this situation, Proposition \ref{irrational to rational - foliations} implies that we can replace the irrational $[\alpha_j]$ which occur in this linear subtree by horizontal rational slopes of the sort claimed in the modified lemma. 

Let $[\alpha_*']$ be a (horizontal) rational element of $\mathcal{S}(W; \mathcal{T})$ which satisfies (1)$'$, (2) and (3). To complete the proof, we need only show that $K([\alpha_*']) \subseteq K([\alpha_*])$, for then the conclusion of the previous paragraph implies that $(K([\alpha_*']); [\alpha_*'])$ is gluing coherent. In other words, $(K; [\alpha_*'])$ is gluing unobstructed. 

To show that $K([\alpha_*']) \subseteq K([\alpha_*])$, fix  $j \in K([\alpha_*']) \setminus K$. Then there is an $i$ such that $M_i([(\alpha_*')^{K_i}]) \cong S^1 \times D^2$ where $T_j = \partial M_i([(\alpha_*')^{K_i}])$. If $[(\alpha_*')^{K_i}] \ne [\alpha_*^{K_i}]_{rat}$, then at least one coordinate of $[\alpha_*^{K_i}]$ is irrational. Condition (3) of the lemma implies that at most one coordinate is irrational. It follows that $M_i([\alpha_*^{K_i}]_{rat}) \cong S^1 \times S^1 \times I$. Since $(K_i; [\alpha_*^{(i)}])$ is foliation detected in $M_i$, $[\alpha_j]$ is irrational (Proposition \ref{tau fibre 1}). But then $j \in K^{\dagger} = K$, a contradiction. Hence $[(\alpha_*')^{K_i}] = [\alpha_*^{K_i}]_{rat}$, so that $j \in K([\alpha_*])$. It follows that $K([\alpha_*']) \subseteq K([\alpha_*])$, which completes the proof.
\end{proof}

\begin{proof}[Proof of the reverse implication of the foliation case of Theorem \ref{theorem: gluing}] 
First suppose that $K = \emptyset$ and fix a gluing coherent family of rational slopes $[\alpha_*]$. 

If $[\alpha_*]$ is horizontal, each $[\alpha_*^{(i)}]$ is detected by a foliation of the form $\mathcal{F}(\rho_i)$ where $\rho_i$ is chosen as in Proposition \ref{standard realisation}. It follows from the conclusions of Lemmas \ref{ready for gluing horizontal non-compact} and \ref{ready for gluing horizontal compact} that if we choose an odd integer $k \gg 0$, we can find a horizontal co-oriented taut foliation on each $M_i$ which detects $[\alpha_*^{(i)}]$ and which is $k$ interval hyperbolic on each boundary component of $M_i$. By Lemma \ref{standard are isotopic}, these foliations glue together to form a co-oriented horizontal foliation on $W$.

If $[\alpha_*]$ is not horizontal, index the pieces $M_1, \ldots, M_n$ so that each partial union $V_i = M_1 \cup \ldots \cup M_i$ is connected. For each piece $M_i$ for which $[\alpha_*^{(i)}]$ is horizontal, choose a representation $\rho_i$ as in Proposition \ref{standard realisation} so that $\mathcal{F}(\rho_i)$ detects $[\alpha_*^{(i)}]$. Fix a constant $k_0 \gg 0$ so that for each $i$ such that $\mathcal{F}(\rho_i)$ has no compact leaves, $k_0 \geq k(\mathcal{F}(\rho_i))$ (cf. Lemma \ref{ready for gluing horizontal non-compact}). We prove that each $V_i$ admits a co-oriented taut foliation $\mathcal{F}_i$ such that 
\begin{itemize}

\item for each component $T_l$ of $\partial V_i$, $\mathcal{F}_i$ detects $[\alpha_l]$ and is $k(l)$ interval hyperbolic on $T_l$ for some $k(l) > k_0$. Further, $k(l)$ is odd if $[\alpha_l]$ is horizontal in $M_i$. 

\vspace{.2cm} \item $\mathcal{F}_i$ is transverse to any predetermined finite set of Seifert fibres in the pieces of $V_i$ 

\end{itemize}
The establishment of the case $i = n$ will complete the argument for $K = \emptyset$. 

The result holds for $i = 1$ by Lemmas \ref{ready for gluing horizontal non-compact}, \ref{ready for gluing horizontal compact} and \ref{ready for gluing vertical}. Suppose that it holds for $V_i$ where $i < n$ and consider $V_{i+1} = V_i \cup M_{i+1}$. Let $T_l = V_i \cap M_{i+1}$ and let $M_j$ be the piece of $V_i$ which contains $T_l$. Fix a co-oriented taut foliation $\mathcal{F}_i$ on $V_i$ as provided by the induction hypothesis and suppose that it is $k$ interval hyperbolic on $T_l$. Since $[\alpha_l]$ cannot be vertical in both $M_j$ and $M_{i+1}$, there are three cases to consider: 
 
\begin{enumerate} 

\item $[\alpha_l]$ is horizontal in both $M_j$ and $M_{i+1}$;  

\vspace{.2cm}  \item $[\alpha_l]$ is vertical in $M_j$ and horizontal in $M_{i+1}$;  

\vspace{.2cm}  \item $[\alpha_l]$ is horizontal in $M_j$ and vertical in $M_{i+1}$. 

\end{enumerate}
Note that $M_{i+1} \not \cong N_2$ in case (2); otherwise the Seifert structure on $M_{i+1}$ with base orbifold a M\"{o}bius band would extend over $M_j \cup M_{i+1}$. 

In each case, the conclusions Lemmas \ref{ready for gluing horizontal non-compact}, \ref{ready for gluing horizontal compact} and \ref{ready for gluing vertical} allow us to conclude that $M_{i+1}$ admits a co-oriented taut foliation which is $k$ interval hyperbolic on $T_l$ and, when pieced together with $\mathcal{F}_i$, yields a co-oriented taut foliation $\mathcal{F}_{i+1}$ on $V_{i+1}$ which satisfies the inductive hypothesis. This completes the proof when $K = \emptyset$. 

Next suppose that $K \ne \emptyset$ and fix a (horizontal) rational family of slopes $[\alpha_*] \in \mathcal{S}(W; \mathcal{T})$ such that $(K; [\alpha_*])$ is gluing unobstructed. Let $K' = K([\alpha_*])$ and consider the manifold $W_0$ obtained by cutting $W$ open along the components $T_j$ of $\mathcal{T}$ for $j \in K'$. The boundary of $W_0$ is a disjoint union of tori, two for each $T_j$ where $j \in K'$. 

By construction, each of the manifolds $M_i([\alpha_*^{K'_i}])$ is Seifert fibred but none are solid tori. (This is where we use the assumption that $K$ is gluing unobstructed.) In particular, they are boundary incompressible. If some $M_i([\alpha_*^{K'_i}])$ is a product $S^1 \times S^1 \times I$, Proposition \ref{tau fibre 1} implies that the two $[\alpha_l]$ associated to the boundary components of $M_i([\alpha_*^{K'_i}])$ correspond through its $I$-bundle structure. 

Dehn fill the boundary components of $W_0$ along the slopes $[\alpha_j]$ for $j \in K'$ to produce a closed graph manifold $W'$ each component of which is a union of a certain number of the $M_i([\alpha_*^{K'_i}])$. The projection of $[\alpha_*]$ to the components of the boundaries of the $M_i([\alpha_*^{K'_i}])$ contained in such a component  form a gluing coherent family. Hence we can produce a co-oriented taut foliation on each component of $W'$ by using the arguments of the case $K = \emptyset$. Moreover, we can suppose that the resulting foliations are transverse to the cores of the filling tori of $M_i([\alpha_*^{K_i}])$. After an isotopy, they intersect each $M_i$ in co-oriented taut foliations which restrict to linear foliations of slope $[\alpha_j]$ on $T_j$ whenever $j \in K$. Thus we can piece together the resulting foliations on the pieces of $W$  to produce a $K$-type co-oriented taut foliation on $W$. This completes the proof.   
\end{proof} 

\begin{proof}[Proof of Proposition \ref{rational leaves}]
If $W$ admits a (horizontal) $K$-type co-oriented taut foliation, there is a (horizontal) $[\alpha_*] \in \mathcal{S}(W; \mathcal{T})$ such that $(K; [\alpha_*])$ is gluing unobstructed (Theorem \ref{theorem: gluing}). By Lemma \ref{can assume rational} we can suppose that $[\alpha_*]$ is rational (and horizontal) and $(K, [\alpha_*])$ is gluing unobstructed. Then Theorem \ref{theorem: gluing} implies that there is a (horizontal) $K$-type co-oriented taut foliation which intersects each JSJ torus in a foliation of rational slope. 
\end{proof}

\section{Proof of the gluing theorem: the left-order case}

We review the standard notation for graphs of groups, and the theorems available to us.  Our notation follows \cite{Se, Ch}, modified slightly since we are only concerned with trees of groups.

Given a graph $Y$ with vertices $v \in V(Y)$ and edges $e  \in E(Y)$,  there are functions $o, t : E(Y) \rightarrow V(Y)$, the origin and tail of each edge.  The notation $\bar{e}$ indicates the edge $e$ with opposite orientation, so that $o(\bar{e}) = t(e)$,  $t(\bar{e}) = o(e)$. For every graph of groups $(G, Y)$ there are edge groups $\{ G_e | e \in E(Y) \}$ and vertex groups $\{ G_v | v \in V(Y) \}$, together with injective maps $\phi_e: G_e \rightarrow G_{t(e)}$ for all $e \in E(Y)$. We require $G_{\bar{e}} = G_e$.

The fundamental group of a graph of groups $(G, Y)$ is written $\pi_1(G, Y, T)$ where $T$ is a maximal tree in the graph $Y$, but we need only consider the case where $Y$ is a tree (and hence the maximal tree $T$ is $Y$ itself).  We write $\pi_1(G,T)$ for the fundamental group of a tree of groups.  If the tree has edges $E(T)$ and vertices $V(T)$, then $\pi_1(G, T)$ has presentation
\[ \langle G_v, v \in V(T) | \mathrm{rel}(G_v),  v \in V(T) \mbox{ and } \phi_e(g) = \phi_{\bar{e}}(g) \mbox{ for all }  g \in G_e \mbox{ and } e \in E(T) \rangle
\]

If $H$ is left-orderable and $\phi:G \rightarrow H$ is injective, then every left-ordering $\mathfrak{o}$ of $H$ induces a left-ordering $\mathfrak{o}^{\phi}$ on $G$ according to the rule $g<^{\phi} h$ if and only if $\phi(g) < \phi(h)$.  
Recall that when $\phi (h) = ghg^{-1}$ is an inner automorphism,  $\mathfrak{o}^{\phi}$ is denoted by $\mathfrak{o}^g$.  The following definitions are from \cite{Ch}.

Next we extend Definition \ref{compatible families} to graphs of groups.

\begin{definition}

{\rm Suppose that $(G, Y)$ is a graph of groups, and and suppose that $\{ \mathcal{L}_v | v \in V(Y) \}$ is a family of sets of left-orderings of the vertex groups $G_v$.  The family $\{ \mathcal{L}_v | v \in V(Y) \}$ is said to be normal if  $\mathcal{L}_v$ is {\it normal} in $G_v$ for all $v \in V(Y)$.}
\end{definition}

\begin{definition}
{\rm Suppose that $(G, Y)$ is a graph of groups, and for each $v \in V(Y)$ let $\mathfrak{o}_v$ be a left-ordering of $G_v$.  The family of left-orderings $\{ \mathfrak{o}_v | v \in V(Y) \}$ is said to be {\it compatible} for $(G,Y)$ if  $\phi_e \phi_{\bar{e}}^{-1}$ is compatible for the pair $(\mathfrak{o}_{o(e)}, \mathfrak{o}_{t(e)})$ for all $e \in E(Y)$. More generally, suppose that $\{ \mathcal{L}_v | v \in V(Y) \}$ is a family of sets of left-orderings of the vertex groups $G_v$.  The family $\{ \mathcal{L}_v | v \in V(Y) \}$ is said to be {\it compatible} for $(G, Y)$ if for every $e \in E(Y)$,  $\phi_e \phi_{\bar{e}}^{-1}$ is compatible for $(\mathcal{L}_{o(e)}, \mathcal{L}_{t(e)})$.  }
\end{definition}

The main criterion we will use to show the existence of certain types of left-orderings of $\pi_1(W)$, where $W$ is a graph manifold as in Section \ref{assumptions graph}, is the following: 
\begin{theorem} \cite[Lemma 2.2]{Ch}
\label{chiswell_gluing}
Suppose that $Y$ is a finite tree, and $(G, Y)$ is a graph of groups and $\{ \mathcal{L}_v | v \in V(Y) \}$ is a normal family of left-orderings that are compatible for $(G, Y)$.  Then $\pi_1(G, Y)$ is left-orderable.  Moreover for each $v \in V(Y)$ let $\mathfrak{o}_v \in \mathcal{L}_v$ be a left-ordering of $G_v$ such that $\{ \mathfrak{o}_v | v \in V(Y) \}$ is compatible for $(G, Y)$.   Then there exists a left-ordering $\mathfrak{o}$ of $\pi_1(G, Y)$ that restricts to $\mathfrak{o}_v$ on $G_v$ for all $v \in V$.
\end{theorem}

 In the context of graph manifolds, our edge groups will always be $\pi_1(T_j)$ for some JSJ torus $T_j \subset W$, so we introduce notation in order to analyze the left-orderings of $\mathbb{Z} \times \mathbb{Z}$.

Given $\alpha \in H_1(T;\mathbb{R})$, if $[\alpha]$ is rational there are four left-orderings of $\pi_1(T)$ that detect the slope $[\alpha]$, if $[\alpha]$ is irrational there are two.  We denote these left-orderings as follows: $\mathfrak{o}(\alpha)$ is the ordering whose positive cone consists of $x \in \pi_1(T)$ such that the oriented angle between $\alpha$ and $x$ lies in $(0, \pi]$, and $\overline{\mathfrak{o}}(\alpha)$ is the ordering whose positive cone consists of $x \in \pi_1(T)$ such that the oriented angle between $\alpha$ and $x$ lies in $[0, \pi)$.  Recall that if $\mathfrak{o}$ is a left-ordering, we denote the opposite ordering by $\mathfrak{o}_{op}$.  For $[\alpha] \in \mathcal{S}(T)$, set
\[ \mathfrak{O}(\alpha) = \{\mathfrak{o}(\alpha),  \mathfrak{o}_{op}(\alpha), \overline{\mathfrak{o}}(\alpha), \overline{\mathfrak{o}}_{op}(\alpha) \}
\]
Note that when $[\alpha]$ is irrational  the orderings $\mathfrak{o}(\alpha)$ and $ \overline{\mathfrak{o}}(\alpha)$ coincide.

\begin{definition}
{\rm Let $M$ be a Seifert fibred $3$-manifold with torus boundary components $T_1, \ldots, T_r$ as in Section \ref{assumptions seifert}.  A family of left-orderings $\mathcal{L}$  of $\pi_1(M)$ is said to be {\it ready for gluing along $[\alpha_*] = ([\alpha_1], \ldots, [\alpha_r])$} if $\mathcal{L}$ is normal, and for all $j \in \{1, \ldots, r\}$
\[ \{ \mathfrak{o} \in LO(\pi_1(T_j)) |  \mbox{ $\mathfrak{o}= \mathfrak{o}'|\pi_1(T_j)$ for some $\mathfrak{o}' \in \mathcal{L}$} \} =  \mathfrak{O}(\alpha_j)
\]}
\end{definition}

If $W$ is a graph manifold as in Section \ref{assumptions graph}, then the JSJ decomposition induces the structure of a graph of groups on $\pi_1(W)$.  Applying Theorem \ref{chiswell_gluing} in this setting, we have:
\begin{proposition}
\label{gluing ready pieces}
Let $[\alpha_*] \in \mathcal{S}(W; \mathcal{T})$.  If there exist families $\mathcal{L}_i \subset LO(\pi_1(M_i))$ of left-orderings that are ready for gluing along $[\alpha_*^{(i)}]$,  then $\pi_1(W)$ is left-orderable.  Moreover $\pi_1(W)$ admits a left-ordering extending $\mathfrak{o}_i \in \mathcal{L}_i$ whenever $\{ \mathfrak{o}_1, \mathfrak{o}_2, \ldots, \mathfrak{o}_n \} $ is compatible for the graph of groups structure on $\pi_1(W)$.
\end{proposition}
\begin{proof}
Suppose that $[\alpha_*] \in \mathcal{S}(W; \mathcal{T})$  and $\mathcal{L}_i \subset LO(\pi_1(M_i))$ are ready for gluing along $[\alpha_*^{(i)}]$. We need to show that the normal families $\mathcal{L}_i$ are compatible for the graph of groups structure on $\pi_1(W)$.  
Writing $\partial M_i = T_{i1} \cup \cdots \cup T_{i r_i}$, there are maps $f_{ij}: T_{ij} \rightarrow M_i$ inducing homomorphisms $\pi_1(T_{ij}) \rightarrow \pi_1(M_i)$, these homomorphisms are the edge maps that give $\pi_1(W)$ the structure of a graph of groups.  

Compatibility of $\mathcal{L}_i$ for the graph of groups structure on $\pi_1(W)$ is a local condition, in the sense that we need only verify the conditions for an arbitrary edge.  So to simplify notation, we fix $T \in \{ T_1, \ldots, T_m \}$  in $\partial M_i \cap \partial M_j$ with corresponding slope $[\alpha]$, and denote the gluing maps by $f_i: \pi_1(T) \rightarrow \pi_1(M_i)$ and $f_j: \pi_1(T) \rightarrow \pi_1(M_j)$.

We check compatibility of the normal families $\mathcal{L}_i$ and $\mathcal{L}_j$ with the gluing maps of the torus $T$.  Given $\mathfrak{o}_i \in \mathcal{L}_i$, since $\mathcal{L}_i$ is ready for gluing along $[\alpha_*^{(i)}]$ the ordering $\mathfrak{o}_i^{f_i}$ detects the slope $[\alpha]$, and thus  $\mathfrak{o}_i^{f_i} \in \mathfrak{O}(\alpha)$.  Since $\mathcal{L}_j$ is ready for gluing along $[\alpha_*^{(j)}]$, there exists an ordering $\mathfrak{o}_j  \in \mathcal{L}_j$ such that $\mathfrak{o}_j^{f_j} = \mathfrak{o}_i^{f_i}$.   It follows that $f_j^{} f_i^{-1}$ is compatible for $(\mathcal{L}_i, \mathcal{L}_j)$, similarly we can show that $f_if_j^{-1}$ is compatible for $(\mathcal{L}_j, \mathcal{L}_i)$.  Therefore $\mathcal{L}_i \subset LO(\pi_1(M_i))$ is compatible for the graph of groups structure on $\pi_1(W)$, and the result follows by Theorem \ref{chiswell_gluing}. \end{proof}

Thus to prove the gluing theorem we will show that if $[\alpha_*] \in \mathcal{S}(M)$ is $\mathfrak{o}$-detected, then there is a family $\mathcal{L}$ of left-orderings of $\pi_1(M)$ that is ready for gluing along $[\alpha_*]$.  We begin with the horizontal case.

\begin{proposition}
\label{horizontal normal family}
With $M$ as in Section \ref{assumptions seifert}, suppose that $[\alpha_*] \in \mathcal{S}(M)$ is horizontal and $\mathfrak{o}$-detected.   Then there exists a family of left-orderings $\mathcal{L} \subset LO(\pi_1(M))$ that is ready for gluing along $[\alpha_*]$. 
\end{proposition}
For the proof we prepare some lemmas.  Recall that for every boundary torus $T_j$, $\pi_1(T_j) \cong H_1(T_j)$ is identified with a subgroup of $H_1(T_j; \mathbb{R})$.
\begin{lemma}
\label{lem:cofinal peripheral elements} 
Let $T_j \in  \partial M$ and suppose $[\alpha_*] \in \mathcal{S}(M)$ is horizontal and $\mathfrak{o}$-detected.  Then $g \in \pi_1(T_j)$ is cofinal in $\pi_1(M)$ if and only if it is not a power of $\alpha_j(\mathfrak{o})$. In particular, if $\mathfrak{o}'$ is a left-ordering having the same set of cofinal elements as $\mathfrak{o}$, then $\mathfrak{o}'$ detects $[\alpha_*]$.
\end{lemma}
\begin{proof}
If $[\alpha_*]$ is horizontal and $\mathfrak{o}$-detected, by Proposition \ref{non-orientable implies not detected} $M$ has base orbifold $P(a_1, \ldots, a_n)$ and the fibre slope $h$ is $\mathfrak{o}$-cofinal in $\pi_1(M)$.  Then any $g \in \pi_1(T_j)$ that is is not a power of $\alpha_j(\mathfrak{o})$ is cofinal in $\pi_1(T_j)$, so for all $n$ there exists $k$ such that $g^{-k} < h^n < g^{k}$ and thus $g$ is $\mathfrak{o}$-cofinal in $\pi_1(M)$.  On the other hand if $g$ is $\mathfrak{o}$-cofinal in $\pi_1(M)$ it cannot be a power of $\alpha_j(\mathfrak{o})$, which is bounded since $[\alpha_j(\mathfrak{o})]$ is $\mathfrak{o}$-detected.
\end{proof}

\begin{lemma}
\label{lem:convexsubgroup}
Suppose that a group $G$ admits a left-ordering $\mathfrak{o}$. Let $g \in G$ be given.   If $\{ g^k \}_{k \in \mathbb{Z}}$ is bounded above by $f$, then there exists a left-ordering $\mathfrak{o}'$ of $G$ and a $\mathfrak{o}'$-convex subgroup $C \subset G$ with $g \in C$ and $f \notin C$.  Moreover, every positive $\mathfrak{o}$-cofinal element of $G$ is positive and $\mathfrak{o}'$-cofinal.
\end{lemma}
\begin{proof}
Suppose that $\{ g^k \}_{k \in \mathbb{Z}}$ is bounded above in the ordering $\mathfrak{o}$ by $f \in G$.  Consider the family of sets $ \mathcal{X} = \{ S \subset G |  x \in S \mbox{ and } y<x \Rightarrow y \in S \}$,  ordered by inclusion.  It is not hard to check that $G$ acts on $\mathcal{X}$ in an order-preserving way by left multiplication.  Set  $X_0 = \{ x \in G | x< g^k \mbox{ for some } k \in \mathbb{Z} \}$, and define a left-ordering $\mathfrak{o}'$ of $G$ as follows.  Given $h \in G$ declare $h>'1$ if either $X_0 \subset h(X_0)$ or $h(X_0) =X_0$ and $h >1$.  One can verify that the subgroup $C=\mathrm{Stab}_G(X_0)$ is convex in the ordering $\mathfrak{o}'$, and contains $g$ but not $f$.

Now suppose that $h$ is $\mathfrak{o}$-cofinal and $h>1$.  To show that $h$ is $\mathfrak{o}'$-positive and $\mathfrak{o}'$-cofinal, let $x \in G$ be given.  Choose $n >0$ such that $xf<h^n$, so that $f<x^{-1}h^n$.  Since $f$ is an upper bound for $\{ g^k \}_{k \in \mathbb{Z}}$ this means that $x^{-1}h^n$ is also an upper bound for $X_0$.  We conclude that $X_0 \subset x^{-1}h^n(X_0)$, so that $1<'x^{-1}h^n$.  In other words, $x<'h^n $ and so $h$ is $\mathfrak{o}'$-cofinal.  Choosing $x = 1$ in the previous argument shows $h >' 1$.
\end{proof}

\begin{proof}[Proof of Proposition \ref{horizontal normal family}]
Suppose that $[\alpha_*] \in \mathcal{S}(M)$ is horizontal and $\mathfrak{o}$-detected.   We construct a family $\mathcal{L}$ of left-orderings of $\pi_1(M)$ that is ready for gluing along $[\alpha_*]$ as follows. Set $S_0 = \{ \mathfrak{o}, \mathfrak{o}_{op}\}$, and for $j =1, \ldots, r$ define the set $S_j$ inductively as follows.  

If $[\alpha_j]$ is irrational then $S_j = S_{j-1}$. Otherwise if $[\alpha_j]$ is rational we create new left-orderings of $\pi_1(M)$ as follows.   Since $\{ \alpha_j^k\}_{k \in \mathbb{Z}} \subset \pi_1(M)$ is bounded above in the ordering $\mathfrak{o}$ (by the fibre slope $h$, for example), we apply Lemma \ref{lem:convexsubgroup} to create a left-ordering $\mathfrak{o}'$ of $\pi_1(M)$ with a proper, $\mathfrak{o}'$-convex subgroup $C$ containing $[\alpha_j]$ but not $h$.  By construction, the positive $\mathfrak{o}$-cofinal elements of $\pi_1(M)$ are again positive and $\mathfrak{o}'$-cofinal, hence by Lemma \ref{lem:cofinal peripheral elements} the left-ordering $\mathfrak{o}'$ detects the tuple $[\alpha_*]$. Since $C$ is convex, the left cosets $\pi_1(M)/C$ can be given a left-invariant total order.  Define $S_j$ to be $S_{j-1}$ together with the four possible lexicographic left-orderings that arise from the sequence
\[ 1 \rightarrow C \rightarrow \pi_1(M) \rightarrow \pi_1(M)/ C \rightarrow 1
\]
By construction every left-ordering in $S_r$ has the same set of cofinal elements as $\mathfrak{o}$ and so detects $[\alpha_*]$ by Lemma  \ref{lem:cofinal peripheral elements}.
Set 
$$ \mathcal{L} = \bigcup_{g \in \pi_1(M)} g S_r g^{-1}
$$
By construction $\mathfrak{O}(\alpha_j) \subset \{ \mathfrak{o} \in \mathrm{LO}(\pi_1(T_j)) |  \mbox{ $\mathfrak{o}= \mathfrak{o}'|\pi_1(T_j)$  for some $\mathfrak{o}' \in \mathcal{L}$} \}$ and $\mathcal{L}$ is normal.  An arbitrary left-ordering of $ \mathcal{L}$ is of the form $\mathfrak{o}^g$ for some $g \in \pi_1(M)$ and $\mathfrak{o} \in S_r$.  Note that since $h$ is $\mathfrak{o}$-cofinal it is also $\mathfrak{o}^g$ cofinal, so by Lemma \ref{cofinal conjugation}(2) $\mathfrak{o}$ and $\mathfrak{o}^g$ have the same cofinal elements. By Lemma \ref{lem:cofinal peripheral elements} we conclude that $\mathfrak{o}^g$ detects $[\alpha_*]$, and so $\mathfrak{O}(\alpha_j) = \{ \mathfrak{o} \in \mathrm{LO}(\pi_1(T_j)) |  \mbox{ $\mathfrak{o}= \mathfrak{o}'|\pi_1(T_j)$  for some $\mathfrak{o}' \in \mathcal{L}$} \}$ and $\mathcal{L}$ is ready for gluing.
\end{proof}

Next we construct ready for gluing families when $[\alpha_*]$ is not horizontal.

\begin{lemma}
\label{composing space normal family} 
{\rm (cf.  Lemma \ref{composingspace2})}
Suppose that $M$ has base orbifold $P(a_1, \ldots, a_n)$ with $r \geq 2$ boundary tori.  If $r=2$ then there exists a family of left-orderings $\mathcal{L}$ ready for gluing along $([h],[h])$; if $r \geq3$ then for each $[\alpha] \in \mathcal{S}(T_r)$ there exists a family $\mathcal{L}$ of left-orderings of $\pi_1(M)$ that is ready for gluing along $([h], \ldots, [h], [\alpha])$.
\end{lemma}
\begin{proof}
If $r=2$ or if $r \geq 3$ and $[\alpha] = [h]$, consider the short exact sequence \[ 1 \rightarrow K \rightarrow \pi_1(M) \rightarrow \mathbb{Z}^{r-1} \rightarrow 1\]
as in Lemma \ref{composingspace2}.  Let $\mathcal{L}$ denote the set of all lexicographic left-orderings of $\pi_1(M)$ arising from pairs $(\mathfrak{o}', \mathfrak{o})$, where $\mathfrak{o}' \in LO(K)$ and $\mathfrak{o} \in LO(\mathbb{Z}^{r-1})$.  Note that every left-ordering in $\mathcal{L}$ detects $([h], \ldots, [h])$, moreover $\mathcal{L}$ is ready for gluing along $([h], \ldots, [h])$.

On the other hand suppose $[\alpha] \neq [h]$ and $r \geq 3$.  First we show that there exists a left-ordering $\mathfrak{o}$ of $\pi_1(M)$ with $[\alpha_*(\mathfrak{o})]$ horizontal and $[\alpha_r(\mathfrak{o})] = [\alpha]$. To see this, set $[\alpha] = [\alpha_r]$ and choose a tuple of horizontal slopes $([\alpha_2], \ldots, [\alpha_{r}]) \in \mathcal{S}(T_2) \times \ldots \times \mathcal{S}(T_{r})$, then applying Corollary\ref{hen and detected slopes} and Proposition \ref{reps to orders} there exists $\mathfrak{o}$ detecting $([\alpha_1], \ldots , [\alpha_r]) \in \mathcal{S}(T_1) \times \ldots \times \mathcal{S}(T_{r})$ for some slope $[\alpha_1]$.  The slope $[\alpha_1]$ is horizontal by Proposition \ref{order vertical}.  Now by Proposition \ref{horizontal normal family} there exists a family $\mathcal{L}_0$ of left-orderings of $\pi_1(M)$ that is ready for gluing along $([\alpha_1], \ldots , [\alpha_r])$.

Consider the short exact sequence $ 1 \rightarrow K \rightarrow \pi_1(M) \rightarrow \mathbb{Z}^{r-2} \rightarrow 1$
where $\pi_1(T_r) \subset K$, $h \in K$ and $\pi_1(T_j) \not \subset K$ for $j \neq r$.  Let $\mathcal{L}$ denote the family of all lexicographic left-orderings of $\pi_1(M)$ arising from pairs of orderings $(\mathfrak{o}' , \mathfrak{o})$ where $\mathfrak{o}' $ is the restriction to $K$ of an ordering in $\mathcal{L}_0$ and $\mathfrak{o} \in LO(\mathbb{Z}^{r-2})$.  By construction $\mathcal{L}$ is normal, and ready for gluing along $([h], \ldots, [h], [\alpha])$.
\end{proof}

\begin{lemma}
\label{one boundary nonorientable}
Suppose that $M$ is Seifert fibred over $Q(a_1, \ldots, a_n)$.  Then there exists a family of left-orderings $\mathcal{L} \subset LO(\pi_1(M))$ that is ready for gluing along $([h], [h], \ldots, [h])$.
\end{lemma}
\begin{proof}
Note that there is a short exact sequence $1 \rightarrow K \rightarrow \pi_1(M) \rightarrow \mathbb{Z}^r \rightarrow 1$ where $h \in K$ and no dual class $x_j$ is killed by the quotient map $\pi_1(M) \rightarrow \mathbb{Z}^r$.  Now proceed as in the proof of Lemma \ref{composing space normal family}.
\end{proof}

Next we include two lemmas that are necessary to show cofinality of the fibre class in certain left-orderings.  The proof of the next lemma is straightforward and so we omit it.

\begin{lemma}
\label{subgroup_cofinal}
 Let $G$ be a group with left-ordering $\mathfrak{o}$ and $H$ a subgroup of $G$.  If there exists $h \in H$ that is $\mathfrak{o}$-cofinal in $G$, then every element of $H$ that is $\mathfrak{o}$-cofinal in $H$ is also $\mathfrak{o}$-cofinal in $G$.
\end{lemma}

\begin{lemma}
\label{cofinal_lemma_2}
Suppose that $G_1, G_2$ are groups with a common subgroup $H$, and let $h_1, h_2 \in H$ be given. Suppose that $\mathfrak{o}$ is a left-ordering of $G_1 *_H G_2$ and that each $h_i$ is $\mathfrak{o}$-cofinal in $G_i$.   If $g \in G_i$ is $\mathfrak{o}$-cofinal in $G_i$, then $g$ is $\mathfrak{o}$-cofinal in $G_1 *_H G_2$.\end{lemma}

\begin{proof} We will begin by showing that $h_1$ is $\mathfrak{o}$-cofinal in $G_1 *_H G_2$, the case of $h_2$ is identical. 

  Every element of $G_1 *_H G_2$ can be represented by a word $w$ of the form
\[ w = g_1 g_2 \ldots  g_k
\]
where $g_j$ and $g_{j+1}$ are never elements of the same $G_i$.   We need to show that there exists $n \in \mathbb{Z}$ such that $h_1^{-n} < w < h_1^n$, we proceed by induction on $k$.  From Lemma \ref{subgroup_cofinal} we know that $h_1$ is cofinal in both $G_1$ and $G_2$, so if $k = 1$ then $g_1$ is an element of either $G_1$ or $G_2$ and such an $n$ exists.  

Assume for induction that such an $n$ exists for all words of length $k-1$ or less.  Set $w = gw'$ where $w'$ is length $k-1$.  If $g \in G_1$ then choose $n$ such that $h_1^{-n} < w' < h_1^n$ and $r$ such that $gh_1^n < h_1^r$, the latter choice is possible since $gh_1^n \in G_1$ and $h_1$ is cofinal in $G_1$.  Then we get an upper bound $w = gw'< gh_1^n < h_1^r$.  On the other hand, if $g \in G_2$ then $gh_1^n \in G_2$.  Since $h_1$ is cofinal in $G_2$ we can make an identical argument to bound $w$ above in this case. Similar arguments allow us to bound $w$ below by a power of $h_1$ and cofinality of $h_1$ follows.

Now let $g \in G_i$ be given.  Since $g$ is cofinal in $G_i$ for all $k \in \mathbb{Z}$ there exists $n \in \mathbb{Z}$ such that $g^{-n} < h_i^k < g^n$.  Since $h_i$ is $\mathfrak{o}$-cofinal in $G_1 *_H G_2$ it follows that $g$ is as well.
   \end{proof}

\begin{proof}[Proof of the order gluing theorem when $K = \emptyset$]
If $\pi_1(W)$ is left-orderable with ordering $\mathfrak{o}$, then $\mathfrak{o}$ detects some $[\alpha_*]  \in \mathcal{S}(W; \mathcal{T})$.  For each $i$, the restriction ordering $\mathfrak{o}_i$ of $\pi_1(M_i)$ detects $[\alpha_*^{(i)}]$. Moreover if every class represented by the Seifert fibre of a piece is cofinal in $\pi_1(W)$, then the class of the fibre of $M_i$ is cofinal in $\pi_1(M_i)$.  Thus $[\alpha_*^{(i)}]$ is horizontal for all $i$, so $[\alpha_*]$ is horizontal.

Conversely suppose we are given $[\alpha_*] \in \mathcal{S}(W; \mathcal{T})$ and for each $i$ the ordering $\mathfrak{o}_i$  of $\pi_1(M_i)$ detects  $[\alpha_*^{(i)}]$.  We refine the decomposition of $W$ into Seifert fibred pieces by cutting each $M_i$ along vertical tori, depending on whether or not $[\alpha_*^{(i)}]$ is horizontal. If $M_i$ is a Seifert fibred piece for which $[\alpha_*^{(i)}]$ is horizontal or vertical, we make no refinement.  Otherwise when $[\alpha_*^{(i)}]$ is neither horizontal or vertical there are two cases (cf. Proposition \ref{order vertical}):

\noindent{\textbf{Case 1.}} $M_i$ has base orbifold $P(a_1, \ldots, a_n)$ and $r_i \geq 3$.  Set $v = v([\alpha_*^{(i)}]) \geq 2$ and index the boundary tori so that $\mathfrak{o}_i$ detects $[h]$ on  $T_{i1}, \ldots, T_{iv}$ and horizontal slopes on $T_{i(v+1)}, \ldots, T_{ir_i}$. Choose an essential vertical torus $T$ cutting $M_i$ into two Seifert fibred pieces $M_{i1}$ and $M_{i2}$ where $\partial M_{i1} = T_{i1}, \ldots, T_{iv}, T$ and $\partial M_{i2}=T_{i(v+1)}, \ldots, T_{ir_i}, T$.  Consider the restrictions $\mathfrak{o}_{ij}$ of $\mathfrak{o}_i$ to $\pi_1(M_{ij})$ for $j=1,2$.  By construction $\mathfrak{o}_{i1}$ detects $([h], \ldots, [h], [\alpha])$ for some $[\alpha] \in \mathcal{S}(T)$, while $\mathfrak{o}_{i2}$ detects $[\alpha]$ on $T$ and the same horizontal slopes as $\mathfrak{o}_i$ on $T_{i(v+1)}, \ldots, T_{ir_i}$.  By Proposition  \ref{order vertical} (2), $[\alpha] \neq [h]$ so $\mathfrak{o}_{i2}$ detects only horizontal slopes.

\noindent{\textbf{Case 2.}} $M_i$ has base orbifold $Q(a_1, \ldots, a_n)$ and $r_i \geq 2$.
Choose an essential vertical torus $T$ cutting $M_i$ into $M_{i1}$, a Seifert fibred manifold over $Q_0(a_1, \ldots, a_n)$, and $M_{i2}$, a Seifert fibred manifold with base orbifold a planar surface. Consider the restrictions $\mathfrak{o}_{ij}$ of $\mathfrak{o}_i$ to $\pi_1(M_{ij})$ for $j=1,2$. The ordering $\mathfrak{o}_{i1}$ of $\pi_1(M_{i1})$ detects $[h]$  by Proposition  \ref{non-orientable implies not detected}.  If all slopes detected by $\mathfrak{o}_{i2}$ are vertical, make no further refinement.  Otherwise cut $M_{i2}$ into pieces as in Case 1.

Thus by refining the decomposition of $W$, and associating the restriction ordering to each piece in the refined decomposition, we can assume $[\alpha_*] \in \mathcal{S}(W; \mathcal{T})$ is gluing coherent and each piece $M_i$ with ordering $\mathfrak{o}_i$ satisfies one of the following:
\begin{enumerate}
\item $M_i$ is Seifert fibred over $Q_0(a_1, \ldots, a_n)$ and $\mathfrak{o}_i$ detects $([h])$.
\item $M_i$ is Seifert fibred over $P(a_1, \ldots, a_n)$, $r_i \geq2$ and $\mathfrak{o}_i$ detects $([h], \ldots, [h])$.
\item $M_i$ is Seifert fibred over $P(a_1, \ldots, a_n)$, $r_i \geq 3$ and $\mathfrak{o}_i$ detects $([h], \ldots, [h], [\alpha])$.
\item $M_i$ is Seifert fibred over $P(a_1, \ldots, a_n)$, $r_i \geq 3$ and $\mathfrak{o}_i$ detects $[\alpha_*^{(i)}] $, which is horizontal.
\end{enumerate}

With a decomposition into pieces of type (1)--(4), by Lemmas \ref{one boundary nonorientable}, \ref{composing space normal family} and by Proposition \ref{horizontal normal family} respectively, there are families $\mathcal{L}_i \subset LO(\pi_1(M_i))$ that are ready for gluing along $[\alpha_*^{(i)}]$.  The result now follows from Proposition \ref{gluing ready pieces}.  Moreover, if $[\alpha_*]$ is horizontal then the class of the fibre in $M_i$ is cofinal in $\pi_1(M_i)$. By  applying Lemmas \ref{subgroup_cofinal} and \ref{cofinal_lemma_2}, an induction on the number of pieces in $W$ shows that the class of each fibre is cofinal in $W$.
\end{proof}

\begin{proof}[Proof of the order gluing theorem when $K \ne \emptyset$]
Now suppose $K \neq \emptyset$.  Suppose $\pi_1(W)$ admits a $K$-type left-ordering $\mathfrak{o}$ detecting the tuple $[\alpha_*] \in \mathcal{S}(W; \mathcal{T})$, and let $C \subset \pi_1(W)$ denote the $\mathfrak{o}$-convex normal subgroup such that $C \cap \pi_1(T_k) = \langle \alpha_k \rangle \cap \pi_1(T_k)$ for $k \in K$.  Suppose there exists $i$ such that $M_i([\alpha_*^{K_i'}]_{rat}) \cong S^1 \times D^2$, and let $T_j = \partial M_i([\alpha_*^{K_i'}]_{rat})$.  Observe that if $[\alpha_j]$ is rational, this forces $\alpha_j \in C$.  Thus with $K' = K'([\alpha_*])$, observe that $\mathfrak{o}_i = \mathfrak{o}| \pi_1(M_i)$ detects $(K'_i; [\alpha_*^{(i)}])$ with $C_i = \pi_1(M_i) \cap C$ the required convex subgroup (appealing to Remark \ref{irrational implies strong order detection}(1) if there are any strongly detected irrational slopes).  Thus $(K' ; [\alpha_*])$ is gluing unobstructed.  Note that if we assume the class of every fibre is cofinal in $\pi_1(W)$, then $[\alpha_*]$ is horizontal as in the case $K = \emptyset$.

Conversely suppose that $(K;[\alpha_*])$ is gluing unobstructed, we may assume $[\alpha_*]$ is rational (Lemma  \ref{can assume rational}) and proceed as in the proof of the foliation gluing theorem.  Cut open $W$ along $T_j$ for $j \in K'$ and Dehn fill the boundary components along the slopes $[\alpha_j]$ for $j \in K'$ to produce a graph manifold $W'$ with two or more components $W_j$ whose pieces are of the form $M_i([\alpha_*^{K_i}]_{rat})$ for some $i$ (here we use gluing unobstructed). 
By Lemma \ref{order_fill_lemma} each $W_j$ admits a gluing coherent family of left-orderings, one for each piece in $W_j$, and so $W_j$ has left-orderable fundamental group by the gluing theorem in the case $K = \emptyset$.  There is a short exact sequence
\[ 1 \rightarrow C \rightarrow \pi_1(W) \rightarrow \coprod \pi_1(W_j)  \rightarrow 1
\]
where the free product is amalgamated along cyclic subgroups $\pi_1(T_j)/\langle \alpha_j \rangle$ for $j \in K'$, and is therefore left-orderable \cite[Corollary 5.3]{BG}. The subgroup $C$ is left-orderable since $\pi_1(W)$ is left-orderable, by the gluing theorem with $K = \emptyset$.  Thus we can construct a $K$-type left-ordering by lexicographically ordering $\pi_1(W)$ using the short exact sequence above. As in the case $K = \emptyset$, Lemmas \ref{subgroup_cofinal} and \ref{cofinal_lemma_2} can be used inductively to show that if the initial $[\alpha_*]$ was horizontal, then the class of a fibre in any piece is cofinal in the constructed $K$-type left-ordering.
\end{proof}

\section{Examples and remarks on smoothness} \label{examples and smoothness}

\subsection{Examples}

Brittenham, Naimi and Roberts provided examples of various phenomena concerning the existence and non-existence of taut foliations in non-Seifert fibred graph manifolds in their paper \cite{BNR}. In particular, they found examples of such manifolds which do not admit taut foliations using methods similar to those found in this paper. Theorem \ref{theorem: gluing} combined with the results of the Appendix can be used to construct many examples of such graph manifolds. 

The next two examples show that the hypotheses of (i) admitting a co-oriented taut foliation, (ii) admitting a horizontal co-oriented taut foliation, and (iii) admitting a strongly rational co-oriented taut foliation are successively more constraining on a graph manifold $W$ (cf. Theorems \ref{equiv 1}, \ref{equiv 2}, \ref{equiv 3}). 

\begin{example}
{\rm Let $W$ be the union of three pieces: 
 
\begin{itemize}
\item $M_1$: a twisted $I$-bundle over the Klein bottle with rational longitude $h_0$; 

\vspace{.2cm} \item $M_2$: a cable space where $[h_0]$ is identified with the Seifert fibre slope $[h]$ of $M_2$;

\vspace{.2cm} \item $M_3$: a trefoil exterior where $[h]$ is identified with a foliation detected slope chosen so that $W$ is a rational homology $3$-sphere.

\end{itemize}
 
Theorem \ref{theorem: gluing} implies that $W$ admits a co-oriented taut foliation but note that it does not admit a horizontal co-oriented taut foliation since any such foliation $\mathcal{F}$ can be isotoped so that it intersects $M_1$ in a foliation detecting $[h] \equiv [h_0] = \mathcal{D}_{fol}(M_1)$ (Proposition \ref{tau fibre 2}) and therefore could not be horizontal in $M_2$.}
\end{example}

\begin{example}
{\rm Let $W$ be the union of two pieces:
 
\begin{itemize}
\item $M_1$: a twisted $I$-bundle over the Klein bottle with rational longitude $h_0$; 

\vspace{.2cm} \item $M_2$: a trefoil exterior where $[h_0]$ is identified with the meridional slope $[\mu]$ of $M_2$.
\end{itemize}
 
We claim that $[\mu]$ is representation detected, and therefore foliation detected, in $M_2$. To see this recall that $\pi_1(M_2) = \langle y_1, y_2, h : y_1^2 = h, y_2^3 = h^2, h \hbox{ central} \rangle$ where the meridional class corresponds to $y_1y_2h^{-1}$. According to \cite[Proposition 2.2(a)]{JN1} we can find elements $A$ and $B$ of $\widetilde{PSL}(2,\mathbb R)$ such that $A$ is conjugate to $\hbox{sh}(1/2)$, $B$ is conjugate to translation by $\hbox{sh}(2/3)$ and $AB$ is a hyperbolic element of $\widetilde{PSL}(2,\mathbb R)$ of translation number 1. We obtain a representation $\rho \in \mathcal{R}_0(M_2)$ by sending $y_1$ to $A$, $y_2$ to $B$ and $h$ to $\hbox{sh}(1)$. The meridional class is sent to a hyperbolic element of translation number $0$, which proves the claim. Theorem \ref{equiv 1} implies that $W$ admits a co-oriented taut foliation which is in fact horizontal since $[h_0]$ is horizontal in $M_1$ and $[\mu]$ is horizontal in $M_2$ (cf. Proposition \ref{prop: horizontal foliation}). But note that there is no strongly rational co-oriented taut foliation in $W$ since it would have to intersect the torus $M_1 \cap M_2$ in a circle fibration of slope $[h_0] = [\mu]$. Thus $[\mu]$ would be strongly foliation detected in $M_2$, which is impossible since then $M_2(\mu_2) \cong S^3$ would admit a taut foliation.}
\end{example}

In our next example we construct a graph manifold rational homology $3$-sphere $W$ with JSJ tori $T_1, \ldots , T_m$ and a subset $K \subseteq \{1, 2, \ldots,m\}$ such that no $(K; [\alpha_*])$ is gluing unobstructed even though there are $[\alpha_*]$ such that $(K; [\alpha_*])$ is gluing coherent (cf. Theorem \ref{theorem: gluing}). 

\begin{example} \label{obstructed}
{\rm Let $W$ be the union of three pieces: 
 
\begin{itemize}
\item $M_1$: a trefoil exterior with meridional slope $[\mu_1]$; 

\vspace{.2cm} \item $M_2$: a cable space glued to $M_1$ so that $M_1 \cup M_2$ is the exterior of a cable on the trefoil with meridional slope $[\mu_2]$; 

\vspace{.2cm} \item $M_3$: a twisted $I$-bundle over the Klein bottle with rational longitude $[h_0]$ identified to $[\mu_2]$. 

\end{itemize}
 
The JSJ tori of $W$ are $T_1 = M_1 \cap M_2$ and $T_2 = M_2 \cap M_3$ and we take $K = \{2\}$. We noted in the previous example that $[\mu_1]$ is foliation detected in $M_1$. It is not hard to see that $(\{1,2\}; ([\mu_1], [\mu_2]))$ is foliation detected in $M_2$ while $(\{2\}; [h_0] = [\mu_2])$ is foliation detected in $M_3$. Thus $(\{2\}; ([\mu_1], [h_0]))$ is gluing coherent. In fact, if $\{2\} \subseteq K \subseteq \{1, 2\}$, the only $[\alpha_*] \in \mathcal{S}(W; \mathcal{T})$ for which $(K; [\alpha_*])$ is gluing coherent is $([\mu_1], [h_0])$. On the other hand, since $M_2(\mu_2) \cong S^1 \times D^2$, any $\{2\}$-type foliation on $W$ necessarily intersects $T_1$ in a circle fibration of slope $[\mu_1]$, which is impossible since $[\mu_1]$ is not strongly detected in $M_1$ (cf. the last sentence of the previous example). Thus $(K; [\alpha_*])$ is gluing obstructed. }
\end{example}

\subsection{Remarks on smoothness} \label{subset: smoothness}

Let $M$ be a compact orientable Seifert fibred manifold as in \S \ref{assumptions seifert} and $J \subseteq \{1, 2, \ldots, r\}$. If $(J; [\alpha_*])$ is foliation-detected, it is $\mathcal{F}$-detected where $\mathcal {F}$ is analytic by Proposition \ref{standard realisation}. It follows that the foliations on a graph manifold rational homology $3$-sphere $W$ constructed in Theorem \ref{equiv 3} can be taken to be smooth. On the other hand, we cannot expect the foliations constructed in the proof of Theorems \ref{equiv 1} and \ref{equiv 2} to be smooth as the operation of thickening leaves used in their proofs does not preserve this property. That being said, the main results of this paper combine with those of the Appendix to imply that in terms of the gluing of its pieces, a generic graph manifold rational homology $3$-sphere $W$ which admits a co-oriented taut foliation also admits a smooth strongly rational co-oriented taut foliation. 

\appendix

\section{The results of Eisenbud, Hirsch, Neumann, Jankins and Naimi} \label{sec: ehnjn}

Let $M$ be a Seifert manifold with base orbifold $P(a_1, \ldots, a_n)$ as in \S \ref{assumptions seifert} and recall that for $b \in \mathbb Z$, $J \subseteq \{1, 2, \ldots , r\}$, and $(\tau_1, \ldots , \tau_r) \in \mathbb R^r$, we defined the notion of JN-realisability in \S \ref{section: detecting via reps}.

For $(\tau_1, \ldots , \tau_r) \in \mathbb R^r$ set 
$$\bar \tau_j = \tau_j - \lfloor \tau_j \rfloor \in [0, 1)  \hbox{ for } j = 1, \ldots , r$$
and 
$$b = b(\tau_1, \ldots , \tau_r) =  -(\lfloor \tau_1 \rfloor + \ldots + \lfloor \tau_r \rfloor)$$ 
The reader can verify that $(J; 0; \gamma_1, \ldots , \gamma_n; \tau_1 \ldots , \tau_r)$ is JN-realisable if and only if $(J; b; \gamma_1, \ldots , \gamma_n;$ $\bar \tau_1, \ldots , \bar \tau_r)$ is JN-realisable.

First we consider the case when some of the $\tau_j$ are integers.  For a fixed tuple $\tau_*$ we introduce the notation:
\begin{itemize}
\item $r_1 = | \{ j  : \tau_j \notin \mathbb{Z} \} |$, the number of non-integral $\tau_j$;
\item $s_0 = | \{ j : j \notin J \mbox{ and } \tau_j \in \mathbb{Z} \}|$, the number of integral $\tau_j$ whose indices are not in $J$;
\item $r_2 = r_1 + s_0$.
\end{itemize}
We also use $J^0$ to denote $J \setminus \{j : \tau_j \in \mathbb Z\}$. 

Since JN-realisability of $(J; b; \gamma_1, \ldots , \gamma_n;$ $\bar \tau_1, \ldots , \bar \tau_r)$ is invariant under every permutation of the $\tau_j$, we may assume that the $\tau_j$ are indexed so that $\tau_1, \ldots, \tau_{r_1}$ are not integers, $\tau_{r_1 +1} \ldots ,\tau_r$ are integers and $J \cap \{ r_1 +1 , \ldots, r\} = \{ r_2+1, \ldots, r\}$.  Then $(J; b; \gamma_1, \ldots , \gamma_n;$ $\bar \tau_1, \ldots , \bar \tau_r)$ is JN-realisable if and only if $(J^0; b; \gamma_1, \ldots , \gamma_n;$ $\bar \tau_1, \ldots , \bar \tau_{r_2})$ is JN-realisable since $\tau_j \in \mathbb{Z}$ and $j \in J$ forces the function $g_j$ corresponding to $\bar{\tau}_j$ to be the identity.  

Therefore, in the case when some of the $\tau_j$ are integers it suffices to consider the case where $j \in J$ implies $\tau_j \not \in \mathbb Z$.  For this case we have the following theorem:

\begin{theorem} \label{zeros} {\rm (\cite[Theorem 1]{JN2})}
\label{JN_theorem_1}
Suppose that if $j \in J$ then $\tau_j \not \in \mathbb Z$ and let $s$ be the number of $\tau_j$ which are integers. If $s > 0$, then $(J; b; \gamma_1, \ldots , \gamma_n; \bar \tau_1, \ldots , \bar \tau_r)$ is JN-realisable if and only if $2 - s \leq b \leq n + r -2$. It is then even JN-realisable in $\widetilde{PSL_2}(\mathbb R)$.
\qed  
\end{theorem} 

Next we consider the case where no $\tau_j$ is an integer. If $n + r \leq 2$, the reader will verify that $(J; b; \gamma_1, \ldots , \gamma_n;  \bar \tau_1, \ldots , \bar \tau_r)$ is JN-realisable if and only $\gamma_1+ \ldots + \gamma_n + \tau_1 + \ldots +  \tau_r = 0$. For $n + r \geq 3$ we have the following theorem. 

\begin{theorem} \label{when realized 2} {\rm (\cite{EHN}, \cite{JN2}, \cite{Na})}
\label{JN_theorem_2}
Suppose that $n + r \geq 3, J \subseteq \{1, 2, \ldots , r\}, b \in \mathbb Z$ and $0 < \gamma_1, \gamma_2, \ldots , \gamma_n, 
\bar \tau_1, \ldots , \bar \tau_r < 1$.  

$(1)$ If $(J; b; \gamma_1, \ldots , \gamma_n;  \bar \tau_1, \ldots , \bar \tau_r)$ is JN-realisable, then $1 \leq b \leq n+ r - 1$.

$(2)$ If $2 \leq b \leq n + r -2$, then $(J; b; \gamma_1, \ldots , \gamma_n;  \bar \tau_1, \ldots , \bar \tau_r)$ is JN-realisable in $\widetilde{PSL_2}(\mathbb R)$.

$(3)$ $(J; n + r -1; \gamma_1, \ldots , \gamma_n;  \bar \tau_1, \ldots , \bar \tau_r)$ is JN-realisable if and only if $(J; 1; 1 - \gamma_1, \ldots , 1 - \gamma_n; 1 - \bar \tau_1, \ldots , 1 - \bar \tau_r)$ is JN-realisable. 

$(4)$ $(J; 1; \gamma_1, \ldots , \gamma_n;  \bar \tau_1, \ldots , \bar \tau_r)$  is JN-realisable if and only if there are coprime integers $0 < A < N$ and some permutation $(\frac{A_1}{N}, \frac{A_2}{N}, \ldots , \frac{A_{n+r}}{N})$ of $(\frac{A}{N}, 1 - \frac{A}{N}, \frac{1}{N}, \ldots , \frac{1}{N})$ such that 
\vspace{-.3cm} 
\begin{itemize}

\vspace{.3cm} \item $\gamma_i < \frac{A_i}{N} \hbox{ for } 1 \leq i \leq n$;

\vspace{.3cm} \item $\bar \tau_j < \frac{A_{j}}{N}$ for all $j \in J$; 

\vspace{.3cm} \item $\bar \tau_j \leq \frac{A_{j}}{N}$ for all $j \not \in J$.
\qed
\end{itemize}
\end{theorem}
We refer the reader to \cite{CW} for an alternate and more direct approach to this result. 

Now fix $J \subseteq \{1, 2, \ldots , r-1\}$ and $\tau_* = (\tau_1, \ldots , \tau_{r-1}) \in \mathbb R^{r-1}$. Define
$$\mathcal{T}(M; J; \tau_* ) = \{\tau' : (J; 0; \gamma_1, \ldots , \gamma_n; \tau_1 \ldots , \tau_{r-1}, \tau') \hbox{ is JN-realisable}\}$$
$$\mathcal{T}_{str}(M; J; \tau_* ) = \{\tau' : (J \cup \{r\}; 0; \gamma_1, \ldots , \gamma_n; \tau_1 \ldots , \tau_{r-1}, \tau') \hbox{ is JN-realisable}\}$$
Both of these intervals are related to the set $\mathcal{D}_{rep}(M ;J)$. The relationship is described in \S \ref{section: detecting via reps}.

Theorems \ref{JN_theorem_1} and \ref{JN_theorem_2} allow us to determine $\mathcal{T}(M; J; \tau_* )$ and $\mathcal{T}_{str}(M; J; \tau_* )$ precisely. 
As above we will reindex the tuple $\tau_*$ and define $r_1 = r_1(\tau_*)$, $s_0 = s_0(\tau_*)$ and $r_2 = r_2(\tau_*) = r_1(\tau_*) + s_0(\tau_*)$ so that $\tau_1, \ldots, \tau_{r_1}$ are not integers, $\tau_{r_1 +1} \ldots, \tau_{r-1}$ are integers and $J \cap \{ r_1 +1 , \ldots, r-1\} = \{ r_2+1, \ldots, r-1\}$.  We also introduce the notation:
\begin{itemize}

\item $b_0 = -(\lfloor \tau_1 \rfloor + \ldots + \lfloor \tau_{r-1} \rfloor )$;

 \item $b(\tau') = b_0 - \lfloor \tau' \rfloor$;

\item $m_0 = b_0 - (n + r_1 + s_0 -1)$; 

\item  $m_1 = b_0 + s_0 -1$.

\end{itemize}
 
We take $J^0$ to denote either $J \setminus \{j : \tau_j \in \mathbb Z\}$ or $(J \setminus \{j : \tau_j \in \mathbb Z\}) \cup \{r\}$. We have already observed in the discussion before Theorem \ref{JN_theorem_1} that $(J; 0; \gamma_1,\ldots , \gamma_n; \tau_1, \ldots , \tau_{r-1}, \tau')$ is JN-realisable if and only if
$(J^0; b(\tau'); \gamma_1,\ldots , \gamma_n;  \bar \tau_1, \ldots , \bar \tau_{r_2}, \overline{\tau'})$ is JN-realisable.   Having reduced to this case we may then apply our previous observation that when no $\tau_j$ is an integer and $n+r \leq 2$, $(J; b; \gamma_1, \ldots , \gamma_n;  \bar \tau_1, \ldots , \bar \tau_r)$ is JN-realisable if and only $\gamma_1+ \ldots + \gamma_n + \tau_1 + \ldots +  \tau_r = 0$.  This yields

\begin{proposition}  \label{tau fibre 1}
Fix $J \subseteq \{1, 2, \ldots , r-1\}$ and $\tau_* = (\tau_1, \ldots , \tau_{r-1}) \in \mathbb R^{r-1}$ where $r \geq 1$. 
Suppose that $n + r_1 + s_0 \leq 1$. Then
$$\mathcal{T}_{str}(M; J; \tau_* )  = \mathcal{T}(M; J; \tau_* ) = \{-\big[(\gamma_1 + \ldots + \gamma_n) + (\tau_1 + \ldots + \tau_{r-1})\big]\}$$ 
\qed 
\end{proposition}

In general,  $\mathcal{T}_{str}(M; J; \tau_* )$ and $ \mathcal{T}(M; J; \tau_* )$ are determined by the following proposition.

\begin{proposition}  \label{tau fibre 2}
Fix $J \subseteq \{1, 2, \ldots , r-1\}$ and $\tau_* = (\tau_1, \ldots , \tau_{r-1}) \in \mathbb R^{r-1}$ where $r \geq 1$. 
Suppose that $n + r_1 + s_0 \geq 2$.  

$(1)$ $(a)$ $(m_0, m_1) \subseteq \mathcal{T}_{str}(M; J; \tau_* ) \subseteq \mathcal{T}(M; J; \tau_* ) \subset (m_0 - 1, m_1 + 1)$.

\indent \hspace{5mm} $(b)$ $[m_0, m_1] \subseteq \mathcal{T}(M; J; \tau_* )$.

\indent \hspace{5mm} $(c)$ If $s_0 > 0$ then $m_0 < m_1$ and $(m_0, m_1) = \mathcal{T}_{str}(M; J; \tau_* ) \subset \mathcal{T}(M; J; \tau_* ) = [m_0 , m_1]$.

$(2)$ $(a)$ If $\mathcal{T}(M; J; \tau_* ) \cap (m_0 - 1, m_0) \ne \emptyset$, then \\
\indent \hspace{11mm} $($i$)$ $s_0 = 0$;  \\
\indent \hspace{11mm} $($ii$)$ $|\{ i : \gamma_i \leq \frac12\}| + |\{ j \in J : 0 < \bar \tau_j \leq \frac12\}| + | \{ j \not \in J : 0 <  \bar \tau_j < \frac12\}| \leq 1$;  \\
\indent \hspace{11mm} $($iii$)$ there is some $\eta \in (m_0 - 1, m_0] \cap \mathbb Q$ such that $\mathcal{T}_{str}(M; J; \tau_* ) \cap (m_0 - 1, m_0] = (\eta, m_0]$ \\ \indent  \hspace{11mm} and $\mathcal{T}(M; J; \tau_* ) \cap (m_0 - 1, m_0] = [\eta, m_0]$.

\indent \hspace{5mm} $(b)$ If $n + r_1 = 2$, then $\mathcal{T}(M; J; \tau_* ) \cap (m_0 - 1, m_0) \ne \emptyset$ if and only if either \\
\indent \hspace{11mm} $($i$)$ $-\big[(\gamma_1 + \ldots + \gamma_n) + (\tau_1 + \ldots + \tau_{r-1})\big] < m_0$, or  \\
\indent \hspace{11mm} $($ii$)$ $n = 0, -\sum_j \tau_j = m_0, J \cap \{j : \tau_j \not \in \mathbb Z\} = \emptyset$, and $\tau_j  \in \mathbb Q$ for all $j$.

$(3)$ $(a)$ If $\mathcal{T}(M; J; \tau_* ) \cap (m_1, m_1 + 1) \ne \emptyset$, then\\
\indent \hspace{11mm} $($i$)$ $s_0 = 0$; \\
\indent \hspace{11mm} $($ii$)$  $|\{ i : \gamma_i \geq \frac12\}| + |\{ j \in J : \bar \tau_j \geq \frac12\}| + |\{ j \not \in J : \bar \tau_j > \frac12\}| \leq 1$. \\
\indent \hspace{11mm} $($iii$)$ there is some $\xi \in (m_1, m_1 + 1)\cap \mathbb Q$ such that $\mathcal{T}_{str}(M; J; \tau_* ) \cap [m_1, m_1 + 1) = [m_1, \xi)$  \\ \indent \hspace{1.1cm} and $\mathcal{T}(M; J; \tau_* ) \cap [m_1, m_1 + 1) = [m_1, \xi]$.

\indent \hspace{5mm} $(b)$ If $n + r_1 = 2$, then $\mathcal{T}(M; J; \tau_* ) \cap (m_1, m_1 + 1) \ne \emptyset$ if and only if either \\
\indent \hspace{11mm} $($i$)$ 
$m_0 < -\big[(\gamma_1 + \ldots + \gamma_n) + (\tau_1 + \ldots + \tau_{r-1})\big]$, or \\
\indent \hspace{11mm} $($ii$)$ $n = 0, -\sum_j \tau_j = m_0, J \cap \{j : \tau_j \not \in \mathbb Z\} = \emptyset$, and $\tau_j  \in \mathbb Q$ for all $j$.  

$(4)$ $(a)$ $\mathcal{T}(M; J; \tau_* )$ is a closed subinterval of $(m_0 - 1, m_1 + 1)$ whose endpoints are rational  \\ \indent \hspace{1.1cm} numbers. 

\indent \hspace{5mm} $(b)$ Either $\mathcal{T}_{str}(M; J; \tau_* )$ is the interior of $\mathcal{T}(M; J; \tau_* )$ or $s_0 = 0, n + r_1 = 2$ and \\ \indent \hspace{1.1cm} $\mathcal{T}_{str}(M; J; \tau_* )  = \mathcal{T}(M; J; \tau_* ) = \{m_0\}$. 

\indent \hspace{5mm} $(c)$  $\mathcal{T}_{str}(M; J; \tau_* ) = \{m_0\}$ if and only if $s_0 = 0$, $n + r_1 = 2, m_0 = -\big[(\gamma_1 + \ldots + \gamma_n) + (\tau_1 + $ \\ \indent \hspace{1.1cm} $ \ldots + \tau_{r-1})\big]$, and either $n \ne 0$, or $J \cap \{j : \tau_j \not \in \mathbb Z\} \ne \emptyset$, or $\tau_j  \not \in \mathbb Q$ for some $j$.
\end{proposition}

\begin{remark}
\label{inequalities_remark}
The conditions 
$$-\big[(\gamma_1 + \ldots + \gamma_n) + (\tau_1 + \ldots + \tau_{r-1})\big] < m_0, > m_0, = m_0$$ 
appearing in parts (2), (3), and (4) of the proposition are equivalent (under the assumptions of that particular subcase), respectively, to the conditions 
$$(\gamma_1 + \ldots + \gamma_n) + (\bar \tau_1 + \ldots + \bar \tau_{r-1} )> 1, < 1, = 1$$
\end{remark}

\begin{proof}[Proof of Proposition \ref{tau fibre 2}]
We use the notation introduced in the discussion preceding Proposition \ref{tau fibre 1}.
As we saw there, we need only consider JN-realisability of $$(J^0; b(\tau'); \gamma_1,\ldots , \gamma_n;  \bar \tau_1, \ldots , \bar \tau_{r_2}, \overline{\tau'})$$ which we examine through consideration of the two cases $s_0 > 0$ and $s_0 = 0$. 

{\bf Case 1}. $s_0  > 0$ 

In this case we show that $(m_0, m_1) =  \mathcal{T}_{str}(M; J; \tau_* )$ and $\mathcal{T}(M; J; \tau_* ) = [m_0, m_1]$; note that when $s_0>0$ we have $m_1 - m_0 = (n + r_1 + s_0)+(s_0 - 2) \geq 2+s_0-2 >0$. Hence $m_1 > m_0$. 

{\bf Subcase 1.1}. $\tau' \not \in \mathbb Z$

In this case, $(J; 0; \gamma_1,\ldots , \gamma_n; \tau_1 \ldots , \tau_{r-1}, \tau') $ is JN-realisable if and only if  
$(J^0; b(\tau'); \gamma_1,\ldots , \gamma_n; \bar \tau_1,\ldots , \bar \tau_{r_1}, $\\ $ 0, \ldots, 0,  \overline{\tau'})$ is JN-realisable 
where there are $s_0$ zeros in the latter. Since $n + r_1 + s_0 + 1 \geq 3$ we can apply Theorem \ref{zeros} to see that $(J^0; b(\tau'); \gamma_1, \ldots , \gamma_n; \bar \tau_1,\ldots , \bar \tau_{r_1}, 0 , \ldots, 0,  \overline{\tau'})$ is JN-realisable if and only if 
$2 - s_0 \leq  b(\tau') \leq (n + r_1 + s_0 +1)- 2$. (Here the $r$ of the theorem corresponds to $r_1 + s_0 +1$.) Equivalently, when $s_0 > 0$ and $\tau' \not \in \mathbb Z$, $(J; 0; \gamma_1,\ldots , \gamma_n; \tau_1 \ldots , \tau_{r-1}, \tau') $ is JN-realisable if and only if $\tau' \in (m_0, m_1) \setminus \mathbb Z$.  

{\bf Subcase 1.2}. $\tau' \in \mathbb Z$ and $r \not \in J^0$

In this case, $(J; 0; \gamma_1,\ldots , \gamma_n; \tau_1 \ldots , \tau_{r-1}, \tau') $ is JN-realisable if and only if  
$(J^0; b(\tau'); \gamma_1,\ldots , \gamma_n;  \bar \tau_1,\ldots , \bar \tau_{r_1}, $\\$ 0 , \ldots, 0)$ is JN-realisable 
where there are $s_0 + 1$ zeros in the latter. Since $n + r_1 + s_0 + 1 \geq 3$ we can apply Theorem \ref{zeros} to see that $(J^0; b(\tau'); \gamma_1,\ldots , \gamma_n;$ $\bar \tau_1,\ldots , \bar \tau_{r_1}, 0 , \ldots, 0,  \overline{\tau'})$ is JN-realisable if and only if 
$1 - s_0 \leq  b(\tau') \leq n + r_1 + s_0 - 1$. Equivalently, when $s_0 > 0, \tau' \in \mathbb Z$ and $r \not \in J$, $(J; 0; \gamma_1,\ldots , \gamma_n; \tau_1 \ldots , \tau_{r-1}, \tau') $ is JN-realisable if and only if $\tau' \in [m_0, m_1] \cap \mathbb Z$.   

{\bf Subcase 1.3}. $\tau' \in \mathbb Z$ and $r \in J^0$

In this case, $(J; 0; \gamma_1,\ldots , \gamma_n; \tau_1 \ldots , \tau_{r-1}, \tau') $ is JN-realisable if and only if
$(J^0; b(\tau'); \gamma_1,\ldots , \gamma_n;  \bar \tau_1,\ldots , \bar \tau_{r_1}, $\\$ 0 , \ldots, 0)$ is JN-realisable 
where there are $s_0$ zeros in the latter. Theorem \ref{zeros} implies that when $n + r_1+ s_0 \geq 3$, $(J^0; b(\tau'); \gamma_1,\ldots , \gamma_n,$ $\bar \tau_1,\ldots , \bar \tau_{r_1}, 0 , \ldots, 0)$ is JN-realisable if and only if $2 - s_0 \leq  b(\tau') \leq n + r_1+ s_0 - 2$. Equivalently, if and only if $\tau' \in (m_0, m_1)  \cap \mathbb Z$. We claim that the same holds when $n + r_1+ s_0 = 2$. To see this, first note that in this case $(n, r_1, s_0)$ is either $(0,0,2), (0, 1, 1)$ or $(1,0,1)$ and $m_1 - m_0 = s_0 \in \{1, 2\}$. We must show that $(J; 0; \gamma_1,\ldots , \gamma_n; \tau_1 \ldots , \tau_{r-1}, \tau')$ is not JN-realisable when $s_0 = 1$ and is only JN-realisable for $\tau = m_0 + 1$ when $s_0 = 2$.   

If $(n, r_1, s_0) = (0,0,2)$, then $(J; 0; \gamma_1,\ldots , \gamma_n; \tau_1 \ldots , \tau_{r-1}, \tau') $ is JN-realisable if and only if  
$(\emptyset; b_0 - \tau'; \emptyset; 0, 0)$ is JN-realisable which, the reader will verify, occurs if and only if $\tau' = m_0 + 1$. If $(n, r_1, s_0) = (0, 1, 1)$, $(J; 0; \gamma_1,\ldots , \gamma_n; \tau_1 \ldots , \tau_{r-1}, \tau') $ is JN-realisable if and only if $(J^0; b_0 - \tau'; \emptyset; \bar \tau_1, 0)$ is JN-realisable where $\bar \tau_1 \not \in \mathbb Z$, which is impossible. Similarly if $(n, r_1, s_0) = (1, 0, 1)$, then $(J; 0; \gamma_1,\ldots , \gamma_n; \tau_1 \ldots , \tau_{r-1}, \tau') $ is JN-realisable if and only if $(\emptyset; b_0 - \tau'; \gamma_1; 0)$ is JN-realisable, which is impossible since $\gamma_1 \not \in \mathbb Z$.

\begin{proof}[Proof of Proposition \ref{tau fibre 2} when $s_0 > 0$] 
Combining the three subcases above we see that 
$$(m_0, m_1) =  \mathcal{T}_{str}(M; J; \tau_* ) \subset \mathcal{T}(M; J; \tau_* ) = [m_0, m_1]$$ 
Hence the proposition holds when $s_0 > 0$. 
\end{proof} 

{\bf Case 2}. $s_0  = 0$ 

Here $r_1 = r_2$, so
{\small$$(J; 0; \gamma_1,\ldots , \gamma_n; \tau_1 \ldots , \tau_{r-1}, \tau') \hbox{ is JN-realisable } \Leftrightarrow 
(J^0; b(\tau'); \gamma_1,\ldots , \gamma_n;  \bar \tau_1,\ldots , \bar \tau_{r_1}, \overline{\tau'}) \hbox{ is JN-realisable}$$} 
\hspace{-1mm}Note that $m_1 - m_0 = n + r_1 - 2 \geq 0$.  By examining subcases we prepare the necessary results.

{\bf Subcase 2.1}. $\tau' \not \in \mathbb Z$

Here, $(J; 0; \gamma_1,\ldots , \gamma_n; \tau_1, \ldots , \tau_{r-1}, \tau')$ is JN-realisable if and only if  
$(J^0; b(\tau'); \gamma_1,\ldots , \gamma_n;$ $\bar \tau_1,\ldots , \bar \tau_{r_1}, \overline{\tau'})$ is JN-realisable. 
We can apply statement (1) of Theorem \ref{when realized 2} to see that if $(J^0; b(\tau'); \gamma_1,\ldots , \gamma_n;$ $\bar \tau_1,\ldots , \bar \tau_{r_1}, \overline{\tau'})$ is JN-realisable then $m_1 +1  > \tau' > m_0 -1$.   As such, $(J^0; b(\tau'); \gamma_1,\ldots , \gamma_n;$ $\bar \tau_1,\ldots , \bar \tau_{r_1}, \overline{\tau'})$ is JN-realisable if and only if one of the following holds:
\begin{enumerate}

\item  $m_0 < \tau' < m_1$;  this follows from applying statement (2) of Theorem \ref{when realized 2} with $r = r_1 +1$ and $b = b(\tau')$.

\vspace{.3cm} \item  $m_0 - 1 < \tau' < m_0$, in this case $b(\tau') = n+r_1$  and $(J^0; b(\tau'); \gamma_1,\ldots , \gamma_n;$ $\bar \tau_1,\ldots , \bar \tau_{r_1}, \overline{\tau'})$ is JN-realisable if and only if $(J^0; 1; 1-\gamma_1,\ldots , 1-\gamma_n;$ $1-\bar \tau_1,\ldots , 1-\bar \tau_{r_1},1- \overline{\tau'})$ is JN-realisable by (3) of  Theorem \ref{when realized 2}.  By (4) of  Theorem \ref{when realized 2}, this happens if and only if there are coprime integers $0 < A < N$ and a permutation $(\frac{A_1}{N}, \ldots, \frac{A_n}{N}, \frac{B_1}{N}, \ldots, \frac{B_{r_1}}{N},  \frac{C}{N})$ of $(\frac{A}{N}, 1 - \frac{A}{N}, \frac{1}{N}, \ldots , \frac{1}{N})$ such that  
\begin{enumerate}

\vspace{.3cm} \item  $1 - \frac{A_i}{N} < \gamma_i \hbox{ for } 1 \leq i \leq n$; 

\vspace{.2cm} \item $1 - \frac{B_{j}}{N} < \bar \tau_j$ for all $j \in J^0$ and $1 - \frac{B_{j}}{N} \leq \bar \tau_j$ for all $j \not \in J^0$;

\vspace{.2cm} \item $1 - \frac{C}{N} < \overline{\tau'}$ if $r  \in J^0$ and $1 - \frac{C}{N} \leq \overline{\tau'}$ if $r \not \in J^0$.

\end{enumerate}

\vspace{.3cm} \item $m_1 < \tau' < m_1 + 1 $, in this case $b(\tau')=1$ and there are coprime integers $0 < A < N$ and a permutation $(\frac{A_1}{N}, \ldots, \frac{A_n}{N}, \frac{B_1}{N}, \ldots, \frac{B_{r_1}}{N},  \frac{C}{N})$ of $(\frac{A}{N}, 1 - \frac{A}{N}, \frac{1}{N}, \ldots , \frac{1}{N})$ such that  
\begin{enumerate}

\vspace{.3cm} \item  $\gamma_i < \frac{A_i}{N} \hbox{ for } 1 \leq i \leq n$; 

\vspace{.2cm} \item $\bar \tau_j < \frac{B_{j}}{N}$ for all $j \in J^0$ and $\bar \tau_j \leq \frac{B_{j}}{N}$ for all $j \not \in J^0$;

\vspace{.2cm} \item $\overline{\tau'} < \frac{C}{N}$ if $r \in J^0$ and $\overline{\tau'} \leq \frac{C}{N}$ if $r \not \in J^0$.

\end{enumerate}

\end{enumerate}

{\bf Subcase 2.2}. $\tau' \in \mathbb Z$ and $r \not \in J^0$

Here we repeat the argument of Subcase 1.2 with $s_0 =0$ and conclude that when
$s_0 = 0, \tau' \in \mathbb Z$ and $r \not \in J$, $(J; 0; \gamma_1,\ldots , \gamma_n; \tau_1 \ldots , \tau_{r-1}, \tau') $ is JN-realisable if and only if $\tau' \in [m_0, m_1] \cap \mathbb Z$.

{\bf Subcase 2.3}. $\tau' \in \mathbb Z$ and $r \in J^0$

Here, $(J; 0; \gamma_1,\ldots , \gamma_n; \tau_1 \ldots , \tau_{r-1}, \tau') $ is JN-realisable if and only if  
$(J^0; b(\tau'); \gamma_1,\ldots , \gamma_n;  \bar \tau_1,\ldots , $ $\bar \tau_{r_1})$ is JN-realisable. 

{\bf Subsubcase 2.3.1}. $n + r_1 = 2$ 

In this case, the remarks preceding Proposition \ref{tau fibre 1} imply that $(J^0; b(\tau'); \gamma_1,\ldots , \gamma_n;$ $\bar \tau_1,\ldots , \bar \tau_{r_1})$ is JN-realisable if and only if  
$\tau' = -[(\gamma_1 + \ldots + \gamma_n) + (\tau_1 + \ldots + \tau_{r-1})]$.

{\bf Subsubcase 2.3.2}. $n + r_1 \geq 3$

In this case we can apply Theorem \ref{when realized 2} to see that $(J^0; b(\tau'); \gamma_1,\ldots , \gamma_n;$ $\bar \tau_1,\ldots , \bar \tau_{r_1})$ is JN-realisable if and only if one of the following three conditions hold:
\begin{enumerate}

\item  $\tau' \in [m_0 + 1, m_1 - 1] \cap \mathbb Z$;

\vspace{.3cm} \item  $\tau' = m_0$ and there are coprime integers $0 < A < N$ and a permutation \\ $(\frac{A_1}{N}, \ldots, \frac{A_n}{N}, \frac{B_1}{N}, \ldots, \frac{B_{r_1}}{N})$ of $(\frac{A}{N}, 1 - \frac{A}{N}, \frac{1}{N}, \ldots , \frac{1}{N})$ such that  
\begin{enumerate}

\vspace{.3cm} \item  $1 - \frac{A_i}{N} < \gamma_i \hbox{ for } 1 \leq i \leq n$; 

\vspace{.2cm} \item $1 - \frac{B_{j}}{N} < \bar \tau_j$ for all $j \in J^0$ and $1 - \frac{B_{j}}{N} \leq \bar \tau_j$ for all $j \not \in J^0$.

\end{enumerate}

\vspace{.3cm} \item $\tau' = m_1$ and there are coprime integers $0 < A < N$ and a permutation \\ $(\frac{A_1}{N}, \ldots, \frac{A_n}{N}, \frac{B_1}{N}, \ldots, \frac{B_{r_1}}{N})$ of $(\frac{A}{N}, 1 - \frac{A}{N}, \frac{1}{N}, \ldots , \frac{1}{N})$ such that 
\begin{enumerate}

\vspace{.3cm} \item $\gamma_i < \frac{A_i}{N} \hbox{ for } 1 \leq i \leq n$; 

\vspace{.2cm} \item $\bar \tau_j < \frac{B_{j}}{N}$ for all $j \in J^0$ and $\bar \tau_j \leq \frac{B_{j}}{N}$ for all $j \not \in J^0$.

\end{enumerate}

\end{enumerate} 

\begin{proof}[Proof of Proposition \ref{tau fibre 2} when $s_0 = 0$] 

Assertion (1) of the proposition holds by Subcases 2.1, 2.2, and 2.3. 

Next we prove assertion (2). Assertion (2)(a)(i) follows from the fact that $\mathcal{T}(M; J; \tau_* ) = [m_0, m_1]$ when $s_0 > 0$, as proven in Case 1. 

Next suppose that $\tau' \in \mathcal{T}(M; J; \tau_* ) \cap (m_0 - 1, m_0)$. Then $\tau'$ satisfies the condition Subcase 2.1(2). Hence there are coprime integers $0 < A < N$ and a permutation $(\frac{A_1}{N}, \ldots, \frac{A_n}{N}, \frac{B_1}{N}, \ldots, \frac{B_{r_1}}{N},  \frac{C}{N})$ of $(\frac{A}{N}, 1 - \frac{A}{N}, \frac{1}{N}, \ldots , \frac{1}{N})$ satisfying the inequalities of the subsubcase. Since at most one of $\frac{A}{N}, 1 - \frac{A}{N}$ and  $1 - \frac{1}{N}$ is less than $\frac12$, we have $|\{ i : \gamma_i \leq \frac12\}| + |\{ j \in J : 0 < \bar \tau_j \leq \frac12\}| + |\{ j \not  \in J : 0 < \bar \tau_j < \frac12\}| \leq 1$. Thus assertion (2)(a)(ii) holds.

For assertion (2)(a)(iii), observe that by Subcase 2.1, 
$$(m_0 - 1, m_0)  \cap \mathcal{T}_{str}(M; J; \tau_*) = \bigcup \big(m_0 - \frac{C}{N}, m_0\big)$$ 
and 
$$(m_0 - 1, m_0)  \cap \mathcal{T}(M; J; \tau_*) = \bigcup \big[m_0 - \frac{C}{N}, m_0\big)$$ 
where the union is over coprime pairs $0 < A < N$ satisfying the constraints of Subcase 2.1(2). 

Set $$D = \frac{1}{1-\max\{\gamma_1, \ldots , \gamma_n,  \bar \tau_1, \ldots , \bar \tau_{r_1}\}}$$
If $n + r_1 \geq 3$ then for every coprime pair $0 < A < N$ satisfying the constraints of Subcase 2.1(2) either $1-\frac{1}{N} < \gamma_i$ for some $1 \leq i \leq n$ or $1- \frac{1}{N} \leq \bar \tau_j$ for some $1 \leq j \leq r_1$, in particular $N \leq D$.  Hence the unions above are over a finite index set, and assertion (2)(a)(iii) follows.

Suppose $n + r_1 = 2$.  For every coprime pair $0 < A < N$ satisfying the constraints of Subcase 2.1(2), if $N >D$ then $1 - \frac{1}{N} > \max \{ \gamma_1, \ldots, \gamma_n, \bar \tau_1, \ldots , \bar \tau_{r_1}\}$ and hence $1 - \frac{1}{N} < \overline{ \tau'}$ and  $\tau' \in [m_0 - \frac{1}{N}, m_0]$.  In particular
for coprime pairs $0<A<N$ with $N>D$ the intervals $[m_0 - \frac{C}{N}, m_0]$ (resp. $(m_0 - \frac{C}{N}, m_0]$)  in the union above are of the form $[m_0 - \frac{1}{N}, m_0]$ (resp. $(m_0 - \frac{1}{N}, m_0]$), and so there is a largest interval $[m_0 - \frac{1}{N'}, m_0]$ (resp. $(m_0 - \frac{1}{N'}, m_0]$), where
$$N' = \min\{N > D : \mbox{ a solution $0 < A < N$ to the condition of Subcase 2.1(2) exists}\}$$
Hence assertion (2)(a)(iii) holds in this case as well.

Next we prove assertion (2)(b). Suppose that $n + r_1 = 2$ and write $(\gamma_1, \ldots, \gamma_n,\bar \tau_1, \ldots, \bar \tau_{r_1}) = (\sigma_1, \sigma_2)$. If $\mathcal{T}(M; J; \tau_* ) \cap (m_0 - 1, m_0) \ne \emptyset$, Subcase 2.1 implies that there are coprime integers $0 < A < N$ and a permutation  $(\frac{A_1}{N}, \frac{A_2}{N}, \frac{A_3}{N})$ of $(\frac{A}{N}, 1 - \frac{A}{N}, \frac{1}{N})$ such that for $k = 1, 2$, 
\begin{itemize}
\item $1 - \frac{A_k}{N} < \sigma_k$ if $\sigma_k = \gamma_i$ for some $i$ or $\sigma_k = \bar \tau_j$ for some $j \in J^0$; 

\vspace{.2cm} \item $1 - \frac{A_k}{N} \leq \sigma_k$ otherwise. 

\end{itemize}
Hence $\sigma_1 + \sigma_2 \geq 1$ and if $\sigma_1 + \sigma_2 = 1$ then $n= 0$, $J \cap \{1, 2\} = \emptyset$, and $\sigma_j = 1 - \frac{A_k}{N}$ for both $k$. Remark \ref{inequalities_remark} then implies that $-\big[(\gamma_1 + \ldots + \gamma_n) + (\tau_1 + \ldots + \tau_{r-1})\big] \leq m_0$ with equality implying that $n= 0$, $J \cap \{1, 2\} = \emptyset$, and $\tau_j \in \mathbb Q$ for all $j$. 

Conversely suppose that $\sigma_1 + \sigma_2 \geq 1$. If $\sigma_1 + \sigma_2 > 1$, choose coprime integers $0 < A < N$ such that $1 - \sigma_1 < \frac{A}{N} <  \sigma_2$. Then $1 - \frac{A}{N} < \sigma_1$  and $1 - (1 - \frac{A}{N}) = \frac{A}{N} < \sigma_2$. Hence $[m_0 - \frac{1}{N}, m_0) \subset \mathcal{T}(M; J; \tau_*)$. On the other hand, if $n = 0, J \cap \{1,2\} = \emptyset$, $\{\sigma_1, \sigma_2\} \subset \mathbb Q$ and $\sigma_1 + \sigma_2 = 1$, choose coprime $0 < A < N$ such that $\sigma_1 = \frac{A}{N}$. Then $\sigma_2 = 1 -  \frac{A}{N}$ and $(1 - \frac1N, m_0)  \subset \mathcal{T}(M; J; \tau_*)$. This completes the proof of assertion (2)(b).  
      
Assertion (3) of the proposition follows similarly.   

Assertion (4)(a) is a consequence of assertions (1), (2) and (3). 

Consider Assertion (4)(b) and write $\mathcal{T}(M; J; \tau_*) = [\eta, \xi]$ where $\xi, \eta \in \mathbb Q$. We will show that if $\xi > \eta$, then neither $\xi$ nor $\eta$ is contained in $\mathcal{T}_{str}(M; J; \tau_*)$. Assertion (1) shows that $\eta \leq m_0$ and $\xi \geq m_1$. Assertions (1), (2) and (3) show that we are done as long as $\eta < m_0$ and $\xi > m_1$. Assume otherwise, say $\xi = m_1 \in \mathcal{T}_{str}(M; J; \tau_*)$. Then when $\tau' = \xi$ we are in the situation described in Subcase 2.3. If $n + r_1 \geq 3$, Subsubcase 2.3.2(3) implies that there are coprime integers $0 < A < N$ and a permutation $(\frac{A_1}{N}, \ldots, \frac{A_n}{N}, \frac{B_1}{N}, \ldots, \frac{B_{r_1}}{N})$ of $(\frac{A}{N}, 1 - \frac{A}{N}, \frac{1}{N}, \ldots , \frac{1}{N})$ such that 
\begin{itemize}

\item $\gamma_i < \frac{A_i}{N} \hbox{ for } 1 \leq i \leq n$; 

\vspace{.2cm} \item $\bar \tau_j < \frac{B_{j}}{N}$ for all $j \in J^0$ and $\bar \tau_j \leq \frac{B_{j}}{N}$ for all $j \not \in J^0$.

\end{itemize}
But then by Subcase 2.1(3), we have $(m_1, m_1 + \frac{1}{N}] \subset \mathcal{T}(M; J; \tau_*)$, contrary to hypothesis. Thus we must have $n + r_1 = 2$ and therefore, $m_0 = m_1 = b_0 - 1$. Subsubcase 2.3.1 then shows that $m_1= \xi = -\big[(\gamma_1 + \ldots + \gamma_n) + (\tau_1 + \ldots + \tau_{r-1})\big]$. By hypothesis, $\mathcal{T}(M; J; \tau_*) \cap (m_1, m_1 + 1) = \emptyset$, so now assertions (2)(b) and (3)(b) imply that $\mathcal{T}(M; J; \tau_*) \cap (m_0 - 1, m_0) =\emptyset$ and therefore $\mathcal{T}(M; J; \tau_*)= \{m_0\}$, contrary to our assumptions. We conclude that $\mathcal{T}_{str}(M; J; \tau_* )$ is the interior of $\mathcal{T}(M; J; \tau_* )$ when the latter is a non-degenerate interval. 

To complete the proof of (4)(b) suppose that $\eta = \xi$, that is, $\mathcal{T}(M; J; \tau_* ) = \{m_0\}$. Then $m_0 = m_1$ so that $n + r_1 = 2$. On the other hand since the base orbifold of $M$ is orientable, the rational longitude $\lambda_M$ of $M$ is horizontal. It follows that $[\lambda_M]$ is strongly representation detected. Thus $\emptyset \ne \mathcal{T}_{str}(M; J; \tau_* ) \subseteq \mathcal{T}(M; J; \tau_* ) = \{m_0\}$. Hence $\mathcal{T}_{str}(M; J; \tau_* ) = \mathcal{T}(M; J; \tau_* ) = \{m_0\}$.

Consider (4)(c). If $\mathcal{T}_{str}(M; J; \tau_*)= \{m_0\}$ then by (4)(b) we see that $n + r_1 = 2$ and therefore Assertions (2)(b) and (3)(b) combine with Remark \ref{inequalities_remark} to show that $m_0 = -\big[(\gamma_1 + \ldots + \gamma_n) + (\tau_1 + \ldots + \tau_{r-1})\big]$ and either $n \ne 0$, or $J \cap \{j : \tau_j \not \in \mathbb Z\} \ne \emptyset$, or $\tau_j  \not \in \mathbb Q$ for some $j$.  

Conversely, if $n + r_1 = 2, -\big[(\gamma_1 + \ldots + \gamma_n) + (\tau_1 + \ldots + \tau_{r-1})\big] = m_0$, and either $n \ne 0$, or $J \cap \{j : \tau_j \not \in \mathbb Z\} \ne \emptyset$, or $\tau_j  \not \in \mathbb Q$ for some $j$, then Assertions (2)(b) and (3)(b) together with Subsubcase 2.3.1 imply that $\mathcal{T}_{str}(M; J; \tau_*)= \{m_0\}$. 
\end{proof} 
This completes the proof of Proposition \ref{tau fibre 2}.
\end{proof}

\begin{corollary} \label{hen and detected slopes} 
Let $M$ be a Seifert manifold with base orbifold $P(a_1, \ldots, a_n)$ as in \S \ref{assumptions seifert} and fix horizontal $[\alpha_j] \in \mathcal{S}(T_j)$ for $1 \leq j \leq r-1$. For each $J \subseteq \{1, 2, \ldots , r\}$, $\{[\alpha] \in \mathcal{S}(T_r) : ([\alpha_1], \ldots , [\alpha_{r-1}], [\alpha]) \in \mathcal{D}_{rep}(M; J)\} $ is a non-empty subinterval of the set of horizontal slopes in $\mathcal{S}(T_r)$ which is closed when $r \not \in J$. Further, either 

$(1)$ the endpoints of this subinterval are rational, or 

$(2)$ $M$ has no singular fibres and, after reindexing, $J \supseteq \{2, 3, \ldots,  r-1\}$, $[\alpha_j] = [\tau_j h - h_j^*]$ where $\tau_1$ is irrational and $\tau_2, \ldots, \tau_{r-1} \in \mathbb Z$. Moreover, $\mathcal{D}_{rep}(M; J) = \mathcal{D}_{rep}(M; J \cup \{r\})$ and $\{[\alpha] \in \mathcal{S}(T_r) : ([\alpha_1], \ldots , [\alpha_{r-1}], [\alpha]) \in \mathcal{D}_{rep}(M; J)\}$ consists of a single irrational slope $[\alpha] = [(\tau_1 + \tau_2 + \ldots + \tau_{r-1})h + h_r^*]$
\end{corollary}

\begin{proof}
Suppose that $[\alpha_*] = ([\alpha_1], \ldots , [\alpha_r])$ is horizontal, say $[\alpha_j] = [\tau_j h - h_j^*]$ for some $\tau_j \in \mathbb R$ and each $j$. Then $([\alpha_*]; J)$ is representation detected if and only if $\tau_r \in \mathcal{T}_{str}(M; J ; (\tau_1, \ldots, \tau_{r-1}))$ when $r \in J$ and if and only if $\tau_r \in \mathcal{T}(M; J \setminus\{r\}; (\tau_1, \ldots, \tau_{r-1}))$ when $r \not \in J$. The corollary now follows from the previous two propositions. 
\end{proof}

Note that when $r = 1$, $M$ fibres over the circle with fibre slope $[\lambda_M]$. Thus Corollary \ref{hen and detected slopes}  immediately implies the following result.   

\begin{corollary} \label{hen and r = 1} 
Let $M$ be a Seifert manifold with base orbifold $P(a_1, \ldots, a_n)$ as in \S \ref{assumptions seifert} and suppose that $\partial M$ is connected. Then $\mathcal{D}_{fol}(M)$ is a closed subinterval with rational endpoints in the set of slopes in $\mathcal{S}(\partial M)$. Further, $[\lambda_M] \in \mathcal{D}_{fol}(M)$.  
\qed
\end{corollary}

\begin{proposition} \label{no parabolics} 
Suppose that $(J; b; \gamma_1, \ldots , \gamma_n; \tau_1, \ldots , \tau_r)$ is JN-realisable in $\widetilde{PSL}(2, \mathbb R)_k$ by $f_1, \ldots , f_n,$ \\ $ g_1, \ldots , g_r$ (cf. \S \ref{section: detecting via reps}), then it is JN-realisable in $\widetilde{PSL}(2, \mathbb R)_k$ by $f_1', \ldots , f_n', g_1', \ldots , g_r'$ where no $g_j'$ is parabolic. 
\end{proposition}

\begin{proof}
Without loss of generality we suppose that $g_1, \ldots, g_s$ are elliptic, $g_{s+1}, \ldots, g_t$ are hyperbolic, and $g_{t+1}, \ldots, g_r$ are parabolic. Then $J \subseteq \{1, \ldots, s\}$.  

If $r = t$ we are done, so assume otherwise.  If $t > s$ write
$$g_{t+1} = g_t^{-1} \circ  h \circ h' \circ \hbox{sh}(b)$$
where $h = (f_1 \circ \cdots \circ f_n \circ g_1 \circ \cdots \circ g_{t-1})^{-1}$ and $h' = (g_{t+2} \circ \cdots \circ g_{r})^{-1}$. There is an open neighbourhood of $g_t$ in $\widetilde{PSL}(2, \mathbb R)_k$ consisting entirely of hyperbolics of the same translation number. As $g_t$ varies in this neighbourhood, the product $g_t^{-1} \circ  h \circ h' \circ \hbox{sh}(b)$ varies over an open neighbourhood of $g_{t+1}$. As such a neighbourhood contains hyperbolics of the same translation number as $g_{t+1}$, we can arrange for $g_{t+1}$ to be hyperbolic up to replacing $g_t$ as above and leaving the $f_i$ and remaining $g_j$ alone. An induction on $r - t$ then completes the proof when $t > s$.

Assume then that $t = s$, so $g_{s+1}, \ldots , g_r$ are parabolic. After conjugating in $\hbox{Homeo}_+(\mathbb R)$ by $F(x) = \frac{x}{k}$ we have $f_1, \ldots , f_n, g_1, \ldots , g_r \in \widetilde{PSL}(2, \mathbb R)$, $f_i$ is conjugate to $\hbox{sh}(k\gamma_i)$ for $1 \leq i \leq n$, $g_j$ is conjugate to $\hbox{sh}(k\tau_j)$ for $j \in J$ and has translation number $k \tau_j$ otherwise, and $f_1 \circ \ldots \circ f_n \circ g_1 \circ \ldots \circ g_r = \hbox{sh}(kb)$. 

If $r = s+1$ there are at least two non-integers among $\gamma_1, \ldots, \gamma_n, \tau_1, \ldots, \tau_{r-1}$ as otherwise the identity $f_1 \circ \cdots \circ f_n \circ g_1 \circ \cdots \circ g_{r} = \hbox{sh}(kb)$ would imply that $g_r$ is conjugate to $\hbox{sh}(x)$ for some real number $x$.  This is not possible since $g_r$ is parabolic. The result is then a straightforward consequence of \cite[Corollary 2.3]{JN1}. Examination of the figures in \cite[Corollary 2.3]{JN1} shows that for each parabolic element of the shaded regions, there is a hyperbolic element in that region of the same translation number. 

If $r > s + 1$, then up to conjugation we can suppose that $g_{s+1}$ and $g_{s+2}$ are lifts of the elements $\pm \left(\begin{smallmatrix} 1  &   \epsilon \\  0   &  1 \end{smallmatrix}\right)$ and $\pm \left(\begin{smallmatrix} a  &   b \\ c   &  2 -a \end{smallmatrix}\right)$ in $PSL(2, \mathbb R)$ where $\epsilon \in \{\pm 1\}$.  For $x$ arbitrarily close to $1$ and $y$ arbitrarily close but not equal to $0$ such that $cy \leq  0$, consider 
\begin{eqnarray} \pm \left(\begin{matrix}1  &   \epsilon \\  0   &  1 \end{matrix}\right) \left(\begin{matrix} a  &   b \\ c   &  2 -a \end{matrix}\right) &=& \pm \left(\begin{matrix} 1  &   \epsilon \\  0   &  1 \end{matrix}\right) \left(\begin{matrix} x  &   y \\  0   &  \frac{1}{x} \end{matrix}\right) \left(\begin{matrix} \frac{1}{x}  &  -y \\  0   & x  \end{matrix}\right)  \left(\begin{matrix} a  &   b \\ c   &  2 -a \end{matrix}\right) \nonumber \\ 
&=& \pm \left(\begin{matrix} x  &  y +  \frac{\epsilon}{x} \\  0   &  \frac{1}{x} \end{matrix}\right) \left(\begin{matrix} \frac{a}{x}  - cy &   \frac{b}{x} - (2-a)y\\ cx   &  (2 -a)x \end{matrix}\right) \nonumber 
\end{eqnarray} 

If $c=0$ then note that $a=1$ and both matrices on the second line are hyperbolic.  On the other hand if $c \neq 0$, then the matrix with trace $1/x +x$ is clearly hyperbolic, whereas the other matrix has trace $t(x) = a/x-cy+(2-a)x$.  Since $t(1)  = 2-cy >2$, for $x$ sufficiently close to $1$ this matrix is hyperbolic as well.  It follows that we can always find an element $h \in \widetilde{PSL}(2, \mathbb R)$ arbitrarily close to the identity so that if we set $g_{s+1}' = g_{s+1} h$ and $g_{s+2}' = h^{-1} g_{s+2}$, then $g_{s+1}'g_{s+2}' = g_{s+1}g_{s+2}$ where $g_{s+1}'$, resp. $g_{s+2}'$, is hyperbolic of the same translation number as $g_{s+1}$, resp. $g_{s+2}$. We can now apply the case $t > s$ to complete the proof.
\end{proof}


\begin{thebibliography}{CCGLS} 


{\footnotesize


\bibitem[BG]{BG} V.~V.~Bludov and A.~M.~W.~Glass, {\it Word problems,
embeddings, and free products of right-ordered groups with amalgamated
subgroup}, Proc. London Math. Soc. {\bf 99} (2009), 585--608.

\bibitem[BB]{BB} M.~Boileau and S.~Boyer, {\it Graph manifolds $\mathbb Z$-homology $3$-spheres and taut foliations}, J. Topology {\bf 8} 1 (2015), 571--585.

\bibitem[Bn]{Bn} J.~Bowden, {\it Approximating $C^0$-foliations by contact structures}, Geom. Funct. Anal., {\bf 26} 5 (2016), 1255--1296.

\bibitem[BC]{BC} S.~Boyer and A.~Clay, {\it Slope detection, foliations in graph manifolds, and L-spaces}, preprint 2015, arXiv:1510.02378.

\bibitem[BGW]{BGW} S.~Boyer, C.~McA.~Gordon, and L.~Watson, {\it On L-spaces and left-orderable fundamental groups}, Math. Ann. {\bf 356} (2013), 1213--1245. 

\bibitem[BRW]{BRW} S.~Boyer, D.~Rolfsen, and B.~Wiest, {\it Orderable 3-manifold groups}, Ann. Inst. Fourier {\bf 55} (2005), 243--288.

\bibitem[Br1]{Br1} M.~Brittenham, {\it Essential laminations in Seifert-fibered spaces}, Topology {\bf 32} (1993), 61--85. 

\bibitem[Br2]{Br2} \bysame, {\it Essential laminations in Seifert-fibered spaces: Boundary behaviour}, Top. Appl. {\bf 95} (1999), 47--62. 

\bibitem[Br3]{Br3} \bysame, {\it Tautly foliated $3$-manifolds with no $\mathbb R$-covered foliations}, in {\bf Foliations: geometry and dynamics} (Warsaw, 2000), 213--224, World Sci. Publ., River Edge, NJ, 2002.

\bibitem[BNR]{BNR} M.~Brittenham, R.~Naimi, and R.~Roberts, {\it Graph manifolds and taut foliations}, J. Diff. Geom. {\bf 47} (1997), 446--470. 

\bibitem[BR]{BR} M.~Brittenham and R.~Roberts, {\it When incompressible tori meet essential laminations}, Pac. J. Math. {\bf 190} (1999), 21-40. 

\bibitem[Ca]{Ca}
D.~Calegari, {\bf Foliations and the Geometry of $3$-manifolds}, Oxford Mathematical Monographs, Oxford University Press, Oxford, UK, 2007. 

\bibitem[CD]{CD} D.~Calegari and N.~Dunfield, {\it Laminations and groups of homeomorphisms of the circle}, Invent. Math. {\bf 152} (2003), 149--204.

\bibitem[CW]{CW}
D.~Calegari and A.~Walker, {\it Ziggurats and rotation numbers}, J. Mod. Dynamics {\bf 5} (2011), 711--746.

\bibitem[CC1]{CC1} A.~Candel and L.~Conlon, {\bf Foliations I}, Graduate Studies in Mathematics 23, Amer. Math. Soc., 2000. 

\bibitem[CC2]{CC2} \bysame, {\bf Foliations II}, Graduate Studies in Mathematics 60, Amer. Math. Soc., 2003. 

\bibitem[Ch]{Ch}
I.~Chiswell, {\it Right orderability and graphs of groups}, J. Group Theory {\bf 14} (2011), 589--601. 

\bibitem[CLW]{CLW} A.~Clay, T.~Lidman, and L.~Watson, {\it Graph
manifolds, left-orderability and amalgamation}, Alg.~\&~Geom.~Top. {\bf 13} 4 (2013), 2347--2368.

\bibitem[CR]{CR}
A.~Clay, D.~Rolfsen, {\it Ordered groups, eigenvalues, knots, surgery and {L}-spaces}, Math. Proc. Cambridge Philos. Soc. {\bf 152} 1 (2012), 115--129. 

\bibitem[Co]{Co}
P.~F.~Conrad, {\it Right-Ordered Groups}, Michigan Math.\ J., 
{\bf 6} (1959), 267--275.

\bibitem[De]{De} C.~Delman, {Essential laminations and Dehn surgery on $2$-bridge knots}, Top. Appl. {\bf 63} (1995), 201--221. 

\bibitem[EHN]{EHN}
D.~Eisenbud, U.~Hirsch, and W.~Neumann, 
{\it Transverse foliations on Seifert bundles and self-homeomorphisms 
of the circle}, Comm. Math. Helv. {\bf 56} (1981), 638--660.

\bibitem[ET]{ET}
Y.~Eliashberg, W.~Thurston
{\bf Confoliations}, University Lecture Series {\bf 13} (1998), Amer. Math. Soc., Providence, RI, USA.

\bibitem[Ga1]{Ga1}
D.~Gabai,
{\it Foliations and the topology of $3$-manifolds}, J. Differential Geom. {\bf 18} (1983), 445--503. 

\bibitem[Ga2]{Ga2} D.~Gabai, {\it Taut foliations of $3$-manifolds and suspensions of $S^1$}, Ann. Inst. Fourier. {\bf 42} (1992), 193--208.

\bibitem[Gh]{Ghys}
E.~Ghys, {\em Groups acting on the circle}, 
L'Enseignement Math. {\bf 22} (2001), 329--407.

\bibitem[JN1]{JN1}
M.~Jankins and W.~Neumann,
{\it Homomorphisms of Fuchsian groups to $PSL(2, \mathbb R)$}, Comm. Math. Helv.
{\bf 60} (1985), 480--495.

\bibitem[JN2]{JN2}
\bysame, 
{\it Rotation numbers and products of circle homomorphisms}, Math. Ann.
{\bf 271} (1985), 381--400.

\bibitem[Ju]{Ju} A.~Juh\'asz, {\it A survey of Heegaard Floer homology},  in {\bf New ideas in low-dimensional topology}, World Scientific (2014), 237--296. 

\bibitem[KR]{KR} W.~Kazez and R.~Roberts, {\it $C^0$ Approximations of foliations}, preprint 2015, arXiv:1509.08382.

\bibitem[Li]{Linnell}
P.~Linnell, {\it Left-ordered amenable and locally indicable groups}, J.
London Math. Soc. (2)
{\bf 60} (1999), 133--142.

\bibitem[LS]{LS}
P.~Lisca and A.~Stipsicz, {\em Ozsv\'ath-{S}zab\'o invariants and tight contact 3-manifolds {III}},
J. Symplectic Geom. {\bf 5} (2007), 357--384.

\bibitem[Na]{Na}
R.~Naimi, {\it Foliations transverse to fibers of Seifert manifolds},
Comm.\ Math.\ Helv.\ {\bf 69} (1994), 155--162.

\bibitem[Nav]{Navas}
A.~ Navas, {\em On the dynamics of (left) orderable groups}, Ann. Inst. Fourier {\bf 60} (2010), 1685--1740.

\bibitem[OSz1]{OSz2004-genus}
P.~Ozsv{\'a}th and Z.~Szab{\'o}, {\it Holomorphic disks and genus bounds},
Geom. Topol. {\bf 8} (2004), 311--334 (electronic).

\bibitem[OSz2]{OSz2005-lens}
\bysame,
{\it On knot {F}loer homology and lens space surgeries},
Topology {\bf 44} (2005), 1281--1300.

\bibitem[Pl]{Pl} J.~F.~Plante, {\it Foliations with measure preserving holonomy}, Ann. of Math. {\bf 102} (1975), 327--361.

\bibitem[Ro]{Ro} R.~Roberts, {\it Constructing taut foliations}, Comment. Math. Helvetica {\bf 70} (1995), 516--545.

\bibitem[Se]{Se} J.-P.~Serre, {\bf Trees}, Springer Monographs in Mathematics,  Springer-Verlag, Berlin, 2003.

\bibitem[Si]{Si}
 A.S. Sikora, {\em Topology on the spaces of orderings of groups},
Bull. London Math. Soc. {\bf 36} (2004), 519--526.

\bibitem[Wa1]{Watson2008}
L.~Watson, {\em Surgery obstructions from {K}hovanov homology},
Selecta Math. (N.S.), {\bf 18} (2012), 417--472.

\bibitem[Wa2]{Watson2013}
\bysame, {\em Heegaard Floer homology solid tori},
invited presentation in AMS Session: Knots, Links, and Three-manifolds, San Diego, January 2013.

}
\end{thebibliography}
\end{document}